\documentclass[a4paper,12pt]{amsart}
\usepackage[latin1]{inputenc}
\usepackage{amsmath}
\usepackage{amsthm}
\usepackage{amsfonts}
\usepackage{amssymb}
\usepackage{enumerate}
\usepackage{mathrsfs}
\usepackage[all]{xy}
\usepackage{epic}
\usepackage{graphicx}
\usepackage{hyperref}
\usepackage{placeins}
\usepackage{color}
\usepackage[labelfont=rm]{caption}
\usepackage{subcaption}
\usepackage{tikz}
\usetikzlibrary{matrix,calc,arrows,decorations.markings,arrows.new}

\newcommand\bk{{\mathbb K}}
\newcommand\bp{{\mathbb P}}
\newcommand\bs{\mathbb{S}}
\newcommand\bfff{{\mathbb F}}
\newcommand\br{{\mathbb R}}
\newcommand\bc{{\mathbb C}}
\newcommand\bq{{\mathbb Q}}
\newcommand\bn{{\mathbb N}}
\newcommand\bz{{\mathbb Z}}
\newcommand\bgg{{\mathbb G}}
\newcommand\bb{{\mathbb B}}
\newcommand\ca{{\mathcal A}}
\newcommand\calp{{\mathcal P}}
\newcommand\calF{{\mathcal F}}
\newcommand\calm{{\mathcal M}}
\newcommand\cac{{\mathcal C}}
\newcommand\scc{{\mathscr C}}
\newcommand\scl{{\mathscr L}}
\newcommand\scm{{\mathscr M}}
\newcommand\sch{{\mathscr H}}
\newcommand\scb{{\mathscr B}}
\newcommand\scg{{\mathscr G}}
\newcommand\scp{{\mathscr P}}
\newcommand\scf{{\mathscr F}}
\newcommand\sca{{\mathscr A}}
\newcommand\scrr{{\mathscr R}}
\newcommand\gs{\mathbf{S}}
\newcommand\gn{\mathbf{N}}
\newcommand\codim{{\text{codim}}}
\newcommand\pref[1]{(\ref{#1})}
\DeclareMathOperator\gr{\text{gr}}
\DeclareMathOperator\coker{coker}
\DeclareMathOperator\Hom{\text{Hom}}
\DeclareMathOperator\im{\text{Im}}
\DeclareMathOperator\Aut{\text{Aut}}
\DeclareMathOperator\GL{\text{GL}}
\DeclareMathOperator\PGL{\text{PGL}}
\DeclareMathOperator\ML{\text{\sc ML}}
\DeclareMathOperator\RYB{\text{\sc RYB}}
\DeclareMathOperator\Stab{\text{Stab}}
\DeclareMathOperator\Adm{\text{Adm}}
\DeclareMathOperator\Hh{H}

\newtheorem{thm}{Theorem}[section]

\newtheorem{prop}[thm]{Proposition}
\newtheorem{cor}[thm]{Corolary}
\newtheorem{lema}[thm]{Lemma}

\theoremstyle{remark}
\newtheorem{obs}[thm]{Remark}
\theoremstyle{definition}
\newtheorem{dfn}[thm]{Definition}

\newtheorem{ejm}[thm]{Example}
\newtheorem{cvt}[thm]{Convention}
\numberwithin{equation}{section}
\numberwithin{figure}{section}

\makeatletter
\renewcommand\p@figure{\thesection.\arabic{figure}\expandafter\@gobble}
\renewcommand\p@subfigure{\thefigure(\alph{subfigure})\expandafter\@gobble}
\makeatother

\title{Heegaard splittings of graph manifolds.}
\author[E. Artal]{Enrique Artal Bartolo}

\address{Departamento de Matem\'aticas-IUMA\\
Universidad de Zaragoza\\
Campus Plaza San Francisco s/n\\
E-50009 Zaragoza SPAIN}
\email{artal@unizar.es
}

\author[S.~Isaza]{Sim{\'o}n Isaza Pe{\~n}aloza}
\address{Departamento de {\'A}lgebra\\
Universidad Complutense\\
Plaza de las Ciencias, n. 3\\
E-28040 Madrid SPAIN}
\email{psisaza@mat.ucm.es}

\author[M.~Marco]{Miguel A. Marco-Buzun{\'a}riz}
\address{Departamento de Matem\'aticas-IUMA\\
Universidad de Zaragoza\\
Campus Plaza San Francisco s/n\\
E-50009 Zaragoza SPAIN}
\email{mmarco@unizar.es}

\newcommand\enet[1]{\renewcommand\theenumi{#1}
\renewcommand\labelenumi{\theenumi}}

\begin{document}
\begin{abstract}
In this paper we give a method to construct Heegaard splittings
of oriented graph manifolds with orientable bases.
A graph manifold is a closed $3$-manifold admitting only Seifert-fibered pieces
in its Jaco-Shalen decomposition; for technical reasons,
we restrict our attention to the \emph{fully}
oriented case, i.e. both the pieces and the bases are oriented.
\end{abstract}
\maketitle

In this paper we deal with graph manifolds. A closed $3$-manifold~$M$
is said to be a graph manifolds if its Jaco-Shalen decomposition
admits only Seifert-fibered pieces. These manifolds were classified
by F.~Waldhausen~\cite{wal:67,wal:67a} and they are completely
determined by a normalized weighted graph (up to a controlled family of exceptions).
For technical reasons  we restrict our attention to the \emph{fully}
oriented case, i.e. we assume $M$~oriented and we also assume
that the bases of the Seifert fibrations are oriented surfaces. This is only
a mild restriction and this family contains the class of $3$-manifolds
which appear naturally in complex geometry: boundary of regular neighbourhoods of complex curves
in complex surfaces, and, in particular links of normal surface complex singularities. 
These manifolds admit another classification in terms
of plumbing graphs, see the work of W.~Neumann~\cite{neu:81}.

A \emph{Heegaard splitting} of a closed orientable $3$-manifold $M$ is a decomposition of~$M$
as a union of two handle bodies sharing a common boundary. This common
boundary is a closed orientable surface $\Sigma$. The genus of the splitting
is defined as the genus $g$ of $\Sigma$. Note that, if we see $\Sigma$ as
the boundary of a handle body, there are $g$ distinguished curves in it, that
correspond to the boundaries of $g$ disks such that, cutting along them, a
closed ball is obtained. In a Heegaard splitting, the same surface is seen
as the boundary of two different handle bodies, so there are two families of
distinguished curves. These two families of curves are enough to determine the
two handle bodies, and hence they also determine the manifold $M$ and the splitting
itself. An oriented closed surface of genus $g$, with two families of $g$ curves
is called a \emph{Heegaard diagram}, which represents a Heegaard splitting. Every closed oriented $3$-manifold admits
a Heegaard splitting~\cite{hgd}, and \cite{rolf} for details. The \emph{Heegaard genus} of such a manifold
is the minimal genus of the Heegaard splittings of~$M$.

There are a lot of works about Heegaard splittings of Seifert fibered manifolds
(the \emph{bricks} of graph manifolds), see e.g. \cite{BZ:84,BOt:91,MorSchu:98}.
In these works, \emph{vertical} and \emph{horizontal} splittings are defined;
our approach will make use of horizontal splittings. These ideas were also transferred
to the case of graph manifolds in~\cite{Schu:04}, where 
the structure of Heegaard splittings is studied. 

The contribution of this work is to provide an explicit method to construct  Heegaard splittings of a graph manifold
from its plumbing graph, namely, we give a closed oriented surface with two
systems of cutting curves. Recall from \cite{neu:81} that some moves are allowed for plumbing
graphs that provide the same manifold; we can use these moves to decrease the genus of the
provided Heegaard splitting even though, in general, our method does not provide a minimal splitting.

Osv\'ath and Szab\'o~\cite{OzSz:04a,OzSz:04b} defined a Floer homology for $3$-manifolds using Heegaard
diagrams (the so-called Heegaard-Floer homology). Since then, Heegaard splittings
have regained interest, specially when having combinatorial methods for its computation 
from a Heegaard diagram, see Sarkar and Wang in~\cite{sarkar-wang}.  
An interesting particular case is its application to the study of normal surface singularity invariants,
specially those whose links are rational homology spheres, as in the series of works of
N\'emethi \emph{et al.}~\cite{LaNe:15,BoNe,Ne:17}.

The paper is organized as follows. We start in \S\ref{sec:introduction} with an example on how to associate to a graph  manifold 
a Heegaard splitting. No proof is given at this time, but the main steps of the construction are illustrated.
In \S\ref{sec:graph-manifold}, we recall the construction of a graph manifold from its plumbing graph for further use. In \S\ref{sec:TopConst}, the main topological constructions which are needed 
for the Heegaard splittings are given, specially the  concept of \emph{float gluings}.
The case of $\mathbb{S}^1$-fiber bundles with Euler number $\pm 1$ is the next goal:
in \S\ref{sec-euler1} the splitting is constructed while in \S\ref{sec:diagramae1}
the Heegaard diagram is described. We follow the same structure
for general $\mathbb{S}^1$-fiber bundles in~\S\ref{sec-eulern}. In~\S\ref{sec:2vertices}, we study the splittings
of the simplest graph manifolds which are not fibered bundles, i.e., corresponding to a simplicial graph with one edge.
The general case is studied in~\S\ref{sec:arbitrary}. This escalonated procedure allows us to split the technical difficulties. Finally, in \S\ref{sec:ejm} we provide explicit examples, including a genus~$3$ splitting of
Poincar{\'e} sphere (link of the $\mathbb{E}_8$ singularity).

\section{illustrative example}\label{sec:introduction}

The goal of this paper is to describe an explicit Heegaard splitting of a
graph manifold. It is presented in the form of a method, that we will now
summarize by describing a surface with two systems of curves
starting from of a decorated graph. We illustrate this with a suitable
example.

\begin{figure}[ht]
\begin{center}
\begin{tikzpicture}[scale=2,
vertice/.style={draw,circle,fill,minimum size=0.2cm,inner
sep=0}
]
\draw (-1,0) --(1,0) to[out=90,in=90] (-1,0)  node[above,pos=.25] {$+$} node[above,pos=.75] {$-$} node[above=30pt,pos=.5] {$+$};
\node[vertice] at (-1,0) {};
\node[vertice] at (0,0) {};
\node[vertice] at (1,0) {};
\node[below] at (0,0) {$[1],-1$};
\node[below] at (1,0) {$[0],0$};
\node[below] at (-1,0) {$[0],2$};
\end{tikzpicture}
\caption{Example graph}
\label{fig:ex_graph}
\end{center}
\end{figure}
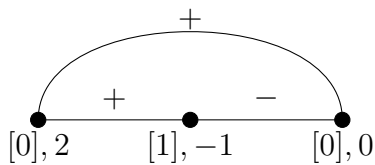

We start with a connected decorated graph. Each vertex $v$ is decorated with two numbers:
a nonnegative integer $[g_v]$ and an integer $e_v$. Each edge is decorated with
a sign.

From the graph, we will construct a surface, and two systems of curves inside of it
(refered to as the system of \emph{blue} curves and the system of \emph{red}
curves), following a process that mymics the construction of the graph from its
elements.
In this process we fix a spanning tree that determines two types of edges:
edges in the tree and edges that close cycles. In our example we fix as spanning tree the straight edges.
The steps to follow are the following:

\begin{enumerate}
\enet{(G\arabic{enumi})}

\item\label{g1} For each vertex~$v$, we consider a pair of
 closed oriented surfaces of genus $g_v$
 (called \emph{top} and \emph{bottom}) as in Figure~\ref{fig:sup-vertices}
 for the example.

\begin{figure}[ht]
\begin{center}
\begin{tikzpicture}[scale=.5]

\def\geno{
//Surface up
\draw[very thick] (2.25+1,1) to [out=10,in=270] (3+1,2.0)
to [out=90, in=0] (0+1,4.0)
to [out=180, in=90] (-3-1, 2)
to [out=270, in=170] (-2.25-1,1) to[out=-10,in=190] (2.25+1,1);}
\def\asa{
\fill[white]
(1.5,3.0)
to [out=150, in=270] (1.0, 4.0)
to [out=90, in=0.0] (0.0, 5.0)
to [out=180, in=90] (-1.0,4.0)
to [out=270, in=30] (-1.5,3.0)--
(-0.25,3.0)
to [out=150, in=270] (-0.5,4.0)
to [out=90, in=180] (0,4.5)
to [out=0, in=90] (0.5,4.0)
to [out=270, in=30] (0.25, 3.0)
;

\draw[very thick] (1.5,3.0)
to [out=150, in=270] (1.0, 4.0)
to [out=90, in=0.0] (0.0, 5.0)
to [out=180, in=90] (-1.0,4.0)
to [out=270, in=30] (-1.5,3.0);

\draw[very thick] (0.25,3.0)
to [out=30, in=270] (0.5,4.0)
to [out=90, in=0] (0,4.5)
to [out=180, in=90] (-0.5,4.0)
to [out=270, in=150] (-0.25, 3.0);
;}

\geno
\asa

\begin{scope}[yscale=-1]
\geno
\asa
\end{scope}

\begin{scope}[xshift=-10cm]
\geno
\end{scope}

\begin{scope}[xshift=-10cm,yscale=-1]
\geno
\end{scope}

\begin{scope}[xshift=10cm]
\geno
\end{scope}

\begin{scope}[xshift=10cm,yscale=-1]
\geno
\end{scope}

\end{tikzpicture}
\caption{}
\label{fig:sup-vertices}
\end{center}
\end{figure}

 \item\label{g2} We join the surfaces of each pair by some cylinders, see
 Figure~\ref{fig:sup-tubos}. To each one of these
cylinders it will be assigned a sign, satisfying the condition that the sum of
these signs in each pair of surfaces matches the number $e_v$. The number of
these cylinders can be chosen freely, as long as the previous condition holds,
and there are enough of them to perform the rest of the steps in the algorithm.
Besides, one of the cylinders in each pair of surfaces is chosen as a
\emph{main cylinder} (larger in Figure~\ref{fig:sup-tubos}).

\begin{figure}[ht]
\begin{center}
\begin{tikzpicture}[scale=.5]

\def\geno{
    //Surface up
    \draw[very thick] (2.25+1,1) to [out=10,in=270] (3+1,2.0)
    to [out=90, in=0] (0+1,4.0)
    to [out=180, in=90] (-3-1, 2)
    to [out=270, in=170] (-2.25-1,1) to[out=-10,in=190] (2.25+1,1);}
\def\asa{
    \fill[white]
    (1.5,3.0)
    to [out=150, in=270] (1.0, 4.0)
    to [out=90, in=0.0] (0.0, 5.0)
    to [out=180, in=90] (-1.0,4.0)
    to [out=270, in=30] (-1.5,3.0)--
    (-0.25,3.0)
    to [out=150, in=270] (-0.5,4.0)
    to [out=90, in=180] (0,4.5)
    to [out=0, in=90] (0.5,4.0)
    to [out=270, in=30] (0.25, 3.0)
    ;

    \draw[very thick] (1.5,3.0)
    to [out=150, in=270] (1.0, 4.0)
    to [out=90, in=0.0] (0.0, 5.0)
    to [out=180, in=90] (-1.0,4.0)
    to [out=270, in=30] (-1.5,3.0);

    \draw[very thick] (0.25,3.0)
    to [out=30, in=270] (0.5,4.0)
    to [out=90, in=0] (0,4.5)
    to [out=180, in=90] (-0.5,4.0)
    to [out=270, in=150] (-0.25, 3.0);
    ;}

\def\tubo{
    \fill[white] (-0.75,1.25) -- (-0.5,1)
    to [out=300, in=90] (-0.3,0)
    to [out=270, in = 60] (-0.5,-1)
    -- (-0.75,-1.25)-- (0.75,-1.25) -- (0.5,-1)
    to [out=-240, in=-90] (0.3,0)
    to [out=-270, in = -120] (0.5,1)
    -- (0.75,1.25)-- (-0.75,1.25);

    \draw[ very thick] (-0.75,1.25) -- (-0.5,1)
    to [out=300, in=90] (-0.3,0)
    to [out=270, in = 60] (-0.5,-1)
    -- (-0.75,-1.25);
    \draw[very thick] (0.75,1.25) -- (0.5,1)
    to [out=240, in=90] (0.3,0)
    to [out=270, in = 120] (0.5,-1)
    -- (0.75,-1.25);
    }

\geno
\asa

\begin{scope}[yscale=-1]
\geno
\asa
\end{scope}

\begin{scope}[xscale=1.25]
\tubo
\node at (0,0) {$-$};
\end{scope}
\begin{scope}[xshift=-2cm]
\tubo
\node at (0,0) {$+$};
\end{scope}

\begin{scope}[xshift=2cm]
\tubo
\node at (0,0) {$-$};
\end{scope}

\begin{scope}[xshift=-10cm]
\geno

\begin{scope}[yscale=-1]
\geno
\end{scope}

\begin{scope}[xshift=-1cm,xscale=1.25]
\tubo
\node at (0,0) {$+$};
\end{scope}

\begin{scope}[xshift=1cm]
\tubo
\node at (0,0) {$+$};
\end{scope}
\end{scope}

\begin{scope}[xshift=10cm]
\geno

\begin{scope}[yscale=-1]
\geno
\end{scope}

\begin{scope}[xshift=-1cm]
\tubo
\node at (0,0) {$-$};
\end{scope}

\begin{scope}[xshift=1cm,xscale=1.25]
\tubo
\node at (0,0) {$+$};
\end{scope}
\end{scope}

//Vertical tubes

%

\end{tikzpicture}
\caption{}
\label{fig:sup-tubos}
\end{center}
\end{figure}

\item\label{g3} For each handle in a surface, see Figure~\ref{fig:handle-red}, we add a red curve that turns around the
handle meridian, passes to the other surface in the pair through the main
cylinder, follows the same path in the other surface (reversing direction) and
returns back to the starting point traversing again the main cylinder (without
self intersections). Another red curve is constructed
in the same way but following the handle longitudes instead
of the meridians.

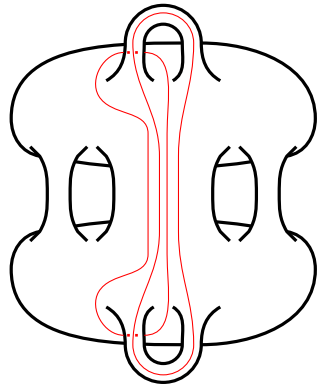
\begin{figure}[ht]
\begin{center}
\begin{tikzpicture}[scale=0.5]

\def\geno{

    \draw[very thick] (2.25+1,1) to [out=10,in=270] (3+1,2.0)
    to [out=90, in=0] (0+1,4.0)
    to [out=180, in=90] (-3-1, 2)
    to [out=270, in=170] (-2.25-1,1) to[out=-10,in=190] (2.25+1,1);}
\def\asa{
    \fill[white]
    (1.5,3.0)
    to [out=150, in=270] (1.0, 4.0)
    to [out=90, in=0.0] (0.0, 5.0)
    to [out=180, in=90] (-1.0,4.0)
    to [out=270, in=30] (-1.5,3.0)--
    (-0.25,3.0)
    to [out=150, in=270] (-0.5,4.0)
    to [out=90, in=180] (0,4.5)
    to [out=0, in=90] (0.5,4.0)
    to [out=270, in=30] (0.25, 3.0)
    ;

    \draw[very thick] (1.5,3.0)
    to [out=150, in=270] (1.0, 4.0)
    to [out=90, in=0.0] (0.0, 5.0)
    to [out=180, in=90] (-1.0,4.0)
    to [out=270, in=30] (-1.5,3.0);

    \draw[very thick] (0.25,3.0)
    to [out=30, in=270] (0.5,4.0)
    to [out=90, in=0] (0,4.5)
    to [out=180, in=90] (-0.5,4.0)
    to [out=270, in=150] (-0.25, 3.0);
    ;}

\def\tubo{
    \fill[white] (-0.75,1.25) -- (-0.5,1)
    to [out=300, in=90] (-0.3,0)
    to [out=270, in = 60] (-0.5,-1)
    -- (-0.75,-1.25)-- (0.75,-1.25) -- (0.5,-1)
    to [out=-240, in=-90] (0.3,0)
    to [out=-270, in = -120] (0.5,1)
    -- (0.75,1.25)-- (-0.75,1.25);

    \draw[ very thick] (-0.75,1.25) -- (-0.5,1)
    to [out=300, in=90] (-0.3,0)
    to [out=270, in = 60] (-0.5,-1)
    -- (-0.75,-1.25);
    \draw[very thick] (0.75,1.25) -- (0.5,1)
    to [out=240, in=90] (0.3,0)
    to [out=270, in = 120] (0.5,-1)
    -- (0.75,-1.25);
    }

\def\tubogordo{
    \fill[white] (-1.75,1.25) -- (-1.5,1)
    to [out=300, in=90] (-1.3,0)
    to [out=270, in = 60] (-1.5,-1)
    -- (-1.75,-1.25)-- (1.75,-1.25) -- (1.5,-1)
    to [out=-240, in=-90] (1.3,0)
    to [out=-270, in = -120] (1.5,1)
    -- (1.75,1.25)-- (-1.75,1.25);

    \draw[ very thick] (-1.75,1.25) -- (-1.5,1)
    to [out=300, in=90] (-1.3,0)
    to [out=270, in = 60] (-1.5,-1)
    -- (-1.75,-1.25);
    \draw[very thick] (1.75,1.25) -- (1.5,1)
    to [out=240, in=90] (1.3,0)
    to [out=270, in = 120] (1.5,-1)
    -- (1.75,-1.25);
    }

\geno
\asa

\begin{scope}[yscale=-1]
\geno
\asa
\end{scope}

\tubogordo

\begin{scope}[xshift=-2.75cm]
\tubo
\end{scope}

\begin{scope}[xshift=2.75cm]
\tubo
\end{scope}

\draw[red] (0.0, 4.8)
to [out=180, in=90] (-0.8,4.0)
to [out=270, in=90] (-0.1, 1.0)
to [out=270, in=90] (-0.1,-1.0)
to [out=270, in=90] (-0.8, -4.0)
to [out=270, in=180] (0.0, -4.8)
to [out=0, in=270] (0.8, -4.0)
to [out=90, in=270] (0.4, -1.0)
to [out=90, in=270] (0.4, 1.0)
to [out=90, in=270] (0.8, 4.0)
to [out=90, in=0] (0,4.8);

\draw[red] (-1.04,3.75)
to [out=180, in=120] (-1.7,2.6)
to [out=300, in=90] (-0.4,1.6)
to [out=270, in=90] (-0.4,-1.6)
to [out=270, in=60] (-1.7, -2.6)
to [out=240, in=180] (-1.04, -3.75);
\draw[line width=1,red, dotted] (-0.95, -3.75)
-- (-0.55, -3.75);

\draw[red] (-0.5, -3.75)
to [out=0, in=270] (0.1,-1.6)
to [out=90, in=270] (0.1,1.6)
to [out=90, in=0.0] (-0.5,3.75);
\draw[line width=1,red, dotted] (-0.95, 3.75)
-- (-0.55, 3.75);

\end{tikzpicture}

\caption{Handle red curves of step~\ref{g3} for the surfaces of the genus~$1$ vertex.}
\label{fig:handle-red}
\end{center}
\end{figure}

\item\label{g4} For each cylinder $C$  which is not a main cylinder, we add a red curve
that goes through the main cylinder and returns through $C$.

\begin{figure}[ht]
\begin{center}
\begin{tikzpicture}[scale=0.5]

\def\geno{

    \draw[very thick] (2.25+1,1) to [out=10,in=270] (3+1,2.0)
    to [out=90, in=0] (0+1,4.0)
    to [out=180, in=90] (-3-1, 2)
    to [out=270, in=170] (-2.25-1,1) to[out=-10,in=190] (2.25+1,1);}
\def\asa{
    \fill[white]
    (1.5,3.0)
    to [out=150, in=270] (1.0, 4.0)
    to [out=90, in=0.0] (0.0, 5.0)
    to [out=180, in=90] (-1.0,4.0)
    to [out=270, in=30] (-1.5,3.0)--
    (-0.25,3.0)
    to [out=150, in=270] (-0.5,4.0)
    to [out=90, in=180] (0,4.5)
    to [out=0, in=90] (0.5,4.0)
    to [out=270, in=30] (0.25, 3.0);

    \draw[very thick] (1.5,3.0)
    to [out=150, in=270] (1.0, 4.0)
    to [out=90, in=0.0] (0.0, 5.0)
    to [out=180, in=90] (-1.0,4.0)
    to [out=270, in=30] (-1.5,3.0);

    \draw[very thick] (0.25,3.0)
    to [out=30, in=270] (0.5,4.0)
    to [out=90, in=0] (0,4.5)
    to [out=180, in=90] (-0.5,4.0)
    to [out=270, in=150] (-0.25, 3.0);
    }

\def\tubo{
    \fill[white] (-0.75,1.25) -- (-0.5,1)
    to [out=300, in=90] (-0.3,0)
    to [out=270, in = 60] (-0.5,-1)
    -- (-0.75,-1.25)-- (0.75,-1.25) -- (0.5,-1)
    to [out=-240, in=-90] (0.3,0)
    to [out=-270, in = -120] (0.5,1)
    -- (0.75,1.25)-- (-0.75,1.25);

    \draw[ very thick] (-0.75,1.25) -- (-0.5,1)
    to [out=300, in=90] (-0.3,0)
    to [out=270, in = 60] (-0.5,-1)
    -- (-0.75,-1.25);
    \draw[very thick] (0.75,1.25) -- (0.5,1)
    to [out=240, in=90] (0.3,0)
    to [out=270, in = 120] (0.5,-1)
    -- (0.75,-1.25);
    }

\def\tubogordo{
    \fill[white] (-1.75,1.25) -- (-1.5,1)
    to [out=300, in=90] (-1.3,0)
    to [out=270, in = 60] (-1.5,-1)
    -- (-1.75,-1.25)-- (1.75,-1.25) -- (1.5,-1)
    to [out=-240, in=-90] (1.3,0)
    to [out=-270, in = -120] (1.5,1)
    -- (1.75,1.25)-- (-1.75,1.25);

    \draw[ very thick] (-1.75,1.25) -- (-1.5,1)
    to [out=300, in=90] (-1.3,0)
    to [out=270, in = 60] (-1.5,-1)
    -- (-1.75,-1.25);
    \draw[very thick] (1.75,1.25) -- (1.5,1)
    to [out=240, in=90] (1.3,0)
    to [out=270, in = 120] (1.5,-1)
    -- (1.75,-1.25);
    }

\geno
\asa

\begin{scope}[yscale=-1]
\geno
\asa
\end{scope}

\tubogordo

\begin{scope}[xshift=-2.75cm]
\tubo
\end{scope}

\begin{scope}[xshift=2.75cm]
\tubo
\end{scope}

\begin{scope}[xshift=-10cm]
\geno

\begin{scope}[yscale=-1]
\geno
\end{scope}

\begin{scope}[xshift=-1cm,xscale=1.25]
\tubo

\end{scope}

\begin{scope}[xshift=1cm]
\tubo

\end{scope}
\end{scope}

\begin{scope}[xshift=10cm]
\geno

\begin{scope}[yscale=-1]
\geno
\end{scope}

\begin{scope}[xshift=-1cm]
\tubo

\end{scope}

\begin{scope}[xshift=1cm,xscale=1.25]
\tubo

\end{scope}
\end{scope}

\def\curvaroja{
    \draw[red] (0.0, 1.6)
    to [out=180, in=90] (-0.8, 0.0)
    to [out=270, in=180] (0.0, -1.6)
    to [out=0, in=270] (0.8, 0.0)
    to [out=90, in=0] (0.0, 1.6);
    }
\begin{scope}[xshift=-1.9cm]
\curvaroja
\end{scope}

\begin{scope}[xshift=1.9cm]
\curvaroja
\end{scope}

\begin{scope}[xshift=9.95cm]
\curvaroja
\end{scope}

\begin{scope}[xshift=-9.95cm]
\curvaroja
\end{scope}

\end{tikzpicture}
\caption{Red curves in step~\ref{g4}.}
\end{center}
\end{figure}
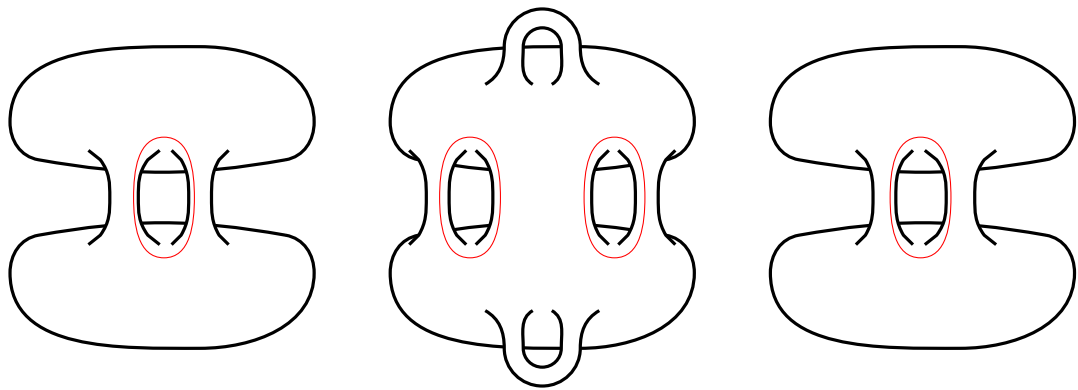

\item\label{g5} For each red curve, we add a blue curve. These blue curves are parallel to
the red curves, except for performing a Dehn twist around each cylinder they
cross. The direction of the Dehn twist is given by the sign of the cylinder.

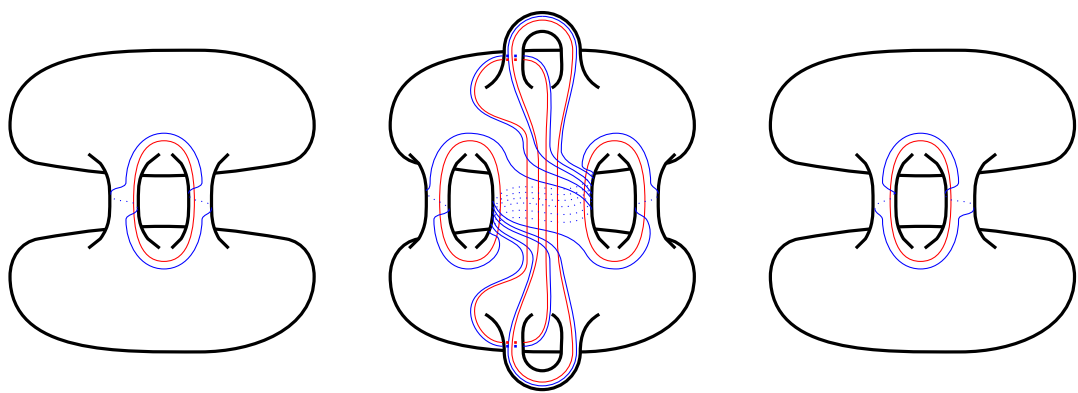
\begin{figure}[ht]
\begin{center}

\begin{tikzpicture}[scale=0.5]

\def\geno{

    \draw[very thick] (2.25+1,1) to [out=10,in=270] (3+1,2.0)
    to [out=90, in=0] (0+1,4.0)
    to [out=180, in=90] (-3-1, 2)
    to [out=270, in=170] (-2.25-1,1) to[out=-10,in=190] (2.25+1,1);}
\def\asa{
    \fill[white]
    (1.5,3.0)
    to [out=150, in=270] (1.0, 4.0)
    to [out=90, in=0.0] (0.0, 5.0)
    to [out=180, in=90] (-1.0,4.0)
    to [out=270, in=30] (-1.5,3.0)--
    (-0.25,3.0)
    to [out=150, in=270] (-0.5,4.0)
    to [out=90, in=180] (0,4.5)
    to [out=0, in=90] (0.5,4.0)
    to [out=270, in=30] (0.25, 3.0);

    \draw[very thick] (1.5,3.0)
    to [out=150, in=270] (1.0, 4.0)
    to [out=90, in=0.0] (0.0, 5.0)
    to [out=180, in=90] (-1.0,4.0)
    to [out=270, in=30] (-1.5,3.0);

    \draw[very thick] (0.25,3.0)
    to [out=30, in=270] (0.5,4.0)
    to [out=90, in=0] (0,4.5)
    to [out=180, in=90] (-0.5,4.0)
    to [out=270, in=150] (-0.25, 3.0);
    }

\def\tubo{
    \fill[white] (-0.75,1.25) -- (-0.5,1)
    to [out=300, in=90] (-0.3,0)
    to [out=270, in = 60] (-0.5,-1)
    -- (-0.75,-1.25)-- (0.75,-1.25) -- (0.5,-1)
    to [out=-240, in=-90] (0.3,0)
    to [out=-270, in = -120] (0.5,1)
    -- (0.75,1.25)-- (-0.75,1.25);

    \draw[ very thick] (-0.75,1.25) -- (-0.5,1)
    to [out=300, in=90] (-0.3,0)
    to [out=270, in = 60] (-0.5,-1)
    -- (-0.75,-1.25);
    \draw[very thick] (0.75,1.25) -- (0.5,1)
    to [out=240, in=90] (0.3,0)
    to [out=270, in = 120] (0.5,-1)
    -- (0.75,-1.25);
    }

\def\tubogordo{
    \fill[white] (-1.75,1.25) -- (-1.5,1)
    to [out=300, in=90] (-1.3,0)
    to [out=270, in = 60] (-1.5,-1)
    -- (-1.75,-1.25)-- (1.75,-1.25) -- (1.5,-1)
    to [out=-240, in=-90] (1.3,0)
    to [out=-270, in = -120] (1.5,1)
    -- (1.75,1.25)-- (-1.75,1.25);

    \draw[ very thick] (-1.75,1.25) -- (-1.5,1)
    to [out=300, in=90] (-1.3,0)
    to [out=270, in = 60] (-1.5,-1)
    -- (-1.75,-1.25);
    \draw[very thick] (1.75,1.25) -- (1.5,1)
    to [out=240, in=90] (1.3,0)
    to [out=270, in = 120] (1.5,-1)
    -- (1.75,-1.25);
    }

\geno
\asa

\begin{scope}[yscale=-1]
\geno
\asa
\end{scope}

\tubogordo

\begin{scope}[xshift=-2.75cm]
\tubo
\end{scope}

\begin{scope}[xshift=2.75cm]
\tubo
\end{scope}

\begin{scope}[xshift=-10cm]
\geno

\begin{scope}[yscale=-1]
\geno
\end{scope}

\begin{scope}[xshift=-1cm,xscale=1.25]
\tubo

\end{scope}

\begin{scope}[xshift=1cm]
\tubo

\end{scope}
\end{scope}

\begin{scope}[xshift=10cm]
\geno

\begin{scope}[yscale=-1]
\geno
\end{scope}

\begin{scope}[xshift=-1cm]
\tubo

\end{scope}

\begin{scope}[xshift=1cm,xscale=1.25]
\tubo

\end{scope}
\end{scope}

\def\curvaroja{
    \draw[red] (0.0, 1.6)
    to [out=180, in=90] (-0.8, 0.0)
    to [out=270, in=180] (0.0, -1.6)
    to [out=0, in=270] (0.8, 0.0)
    to [out=90, in=0] (0.0, 1.6);
    }
\begin{scope}[xshift=-1.9cm]
\curvaroja
\end{scope}

\begin{scope}[xshift=1.9cm]
\curvaroja
\end{scope}

\begin{scope}[xshift=9.95cm]
\curvaroja
\end{scope}

\begin{scope}[xshift=-9.95cm]
\curvaroja
\end{scope}

\begin{scope}[xshift=-9.95cm]
\draw[blue] (0.65, 0.2)
to [out=90, in=270] (1.0, 0.5)
to [out=90, in=0] (0.0, 1.8)
to [out=180, in=90] (-1.0, 0.5)
to [out=270, in=90] (-1.4, 0.2);
\draw[blue, dotted] (-1.4, 0.2)
to [out=270, in=90] (-0.7, -0.2);
\draw[blue]  (-0.7,-0.2)
to [out=270, in=90] (-1.0, -0.5)
to [out=270, in=180] (0.0, -1.8)
to [out=0, in=270] (1.0, -0.5)
to [out=90, in=270] (1.25, -0.2);
\draw[blue, dotted] (1.25, -0.2)
to [out=90, in=270] (0.65, 0.2);
\end{scope}

\begin{scope}[xshift=9.95cm]
\draw[blue] (0.65, 0.2)
to [out=90, in=270] (1.0, 0.5)
to [out=90, in=0] (0.0, 1.8)
to [out=180, in=90] (-1.0, 0.5)
to [out=270, in=0] (-0.7, 0.2);
\draw[blue, dotted] (-0.7, 0.2)
to [out=270, in=90] (-1.2, -0.2);
\draw[blue]  (-1.2,-0.2)
to [out=270, in=90] (-1.0, -0.5)
to [out=270, in=180] (0.0, -1.8)
to [out=0, in=270] (1.0, -0.5)
to [out=90, in=270] (1.4, -0.2);
\draw[blue, dotted] (1.4, -0.2)
to [out=90, in=270] (0.65, 0.2);
\end{scope}

\begin{scope}[xshift=-1.95cm]
\draw[blue] (3.25, -0.2)
to [out=120, in=290] (1., 1.2)
to [out=110, in=0] (0.0, 1.8)
to [out=180, in=90] (-1.0, 0.5)
to [out=270, in=0] (-1.1, 0.3);
\draw[blue, dotted] (3.25, -0.2)
to [out=240, in=60] (0.55, -0.8);
\draw[blue]  (-0.5,-0.2)
to [out=270, in=90] (-0.9, -0.5)
to [out=270, in=180] (0.0, -1.8)
to [out=0, in=290] (0.8, -1.0)
to [out=110, in=290] (0.55, -0.8);
\draw[blue, dotted] (-1.1, 0.3)
to [out=270, in=90] (-0.5, -0.3);
\end{scope}

\begin{scope}[xshift=1.95cm]
\draw[blue] (1.1, 0.2)
to [out=90, in=270] (0.9, 0.5)
to [out=90, in=0] (0.0, 1.8)
to [out=180, in=90] (-0.9, 0.95)
to [out=270, in=60] (-0.62, 0.7);
\draw[blue, dotted] (-0.65, 0.7)
to [out=240, in=60] (-3.2, -0.0);
\draw[blue]  (-3.25,-0.0)
to [out=290, in=120] (-0.8, -1.4)
to [out=300, in=180] (0.0, -1.8)
to [out=0, in=300] (0.8, -0.5)
to [out=120, in=270] (0.5, -0.2);
\draw[blue, dotted] (0.5, -0.2)
to [out=90, in=270] (1.1, 0.2);
\end{scope}

\draw[red] (0.0, 4.8)
to [out=180, in=90] (-0.8,4.0)
to [out=270, in=90] (-0.1, 1.0)
to [out=270, in=90] (-0.1,-1.0)
to [out=270, in=90] (-0.8, -4.0)
to [out=270, in=180] (0.0, -4.8)
to [out=0, in=270] (0.8, -4.0)
to [out=90, in=270] (0.4, -1.0)
to [out=90, in=270] (0.4, 1.0)
to [out=90, in=270] (0.8, 4.0)
to [out=90, in=0] (0,4.8);

\draw[blue] (0.0, 4.9)
to [out=180, in=90] (-0.9,4.0)
to [out=270, in=106] (-0.22, 1.6)
to [out=283, in=120] (1.3,0.20);
\draw[blue, dotted] (1.25,0.2)
to [out=240, in=60] (-1.25,-0.4);
\draw[blue] (-1.3,-0.4)
to [out=270, in=80] (-0.25,-1.6)
to [out=260, in=90] (-0.9, -4.0)
to [out=270, in=180] (0.0, -4.9)
to [out=0, in=270] (0.9, -4.0)
to [out=90, in=280] (0.5,-1.6)
to [out=100, in=290] (-1.3, -0.15);
\draw[blue, dotted] (1.3,0.55)
to [out=240, in=60] (-1.25,-0.15);
\draw[blue] (1.35,0.60)
to [out=130, in=290] (0.5, 1.3)
to [out=90, in=270] (0.9, 4.0)
to [out=90, in=0] (0,4.9);

\draw[red] (-1.04,3.75)
to [out=180, in=120] (-1.7,2.6)
to [out=300, in=90] (-0.4,1.6)
to [out=270, in=90] (-0.4,-1.6)
to [out=270, in=60] (-1.7, -2.6)
to [out=240, in=180] (-1.04, -3.75);
\draw[line width=1,red, dotted] (-0.95, -3.75)
-- (-0.55, -3.75);

\draw[red] (-0.5, -3.75)
to [out=0, in=270] (0.1,-1.6)
to [out=90, in=270] (0.1,1.6)
to [out=90, in=0.0] (-0.5,3.75);
\draw[line width=1,red, dotted] (-0.95, 3.75)
-- (-0.55, 3.75);

\draw[blue] (-1.04,3.85)
to [out=180, in=120] (-1.8,2.6)
to [out=300, in=90] (-0.5,1.5)
to [out=270, in=120] (1.3,0.05);
\draw[blue, dotted] (1.3, 0.0)
to [out=240, in=60] (-1.3,-0.6);
\draw[blue] (-1.35,-0.6)
to [out=290, in=90] (-0.5,-1.5)
to [out=270, in=60] (-1.8, -2.6)
to [out=240, in=180] (-1.04, -3.85);
\draw[line width=1,blue, dotted] (-0.95, -3.85)
-- (-0.55, -3.85);

\draw[blue] (-0.5, -3.85)
to [out=0, in=270] (0.2,-1.6)
to [out=90, in=290] (-1.3, -0.3);
\draw[blue, dotted] (-1.3, -0.3)
to [out=60, in=240] (1.3, 0.4);
\draw[blue] (1.3,0.4)
to [out=120, in=270] (0.2,1.5)
to [out=90, in=0.0] (-0.5,3.85);
\draw[line width=1,blue, dotted] (-0.95, 3.85)
-- (-0.55, 3.85);

\end{tikzpicture}

 \caption{All lines added after step~\ref{g5}.}
\end{center}

\end{figure}

\item\label{g6} Now we add the edges of the graph one by one, starting
from the edges in the tree. To add an edge of sign~$s$ in the tree,
we choose one cylinder with sign~$s$ in each
of the corresponding pair of surfaces.
 These cylinders should be crossed only
by one blue line (i.e. distinct from the main one, when the corresponding surface has either more than two cylinders or positive genus).
Then we substitute these two cylinders with one cylinder
that joins the upper surfaces, and another one that joins the lower
ones. The red lines are just directly glued. The blue lines are also
glued to form a new one. This new blue line goes parallel to the new red
line in one of the new cylinders, but performs a Dehn
twist around the other one. The direction of the Dehn twist will be
given by the sign $s$ of the edge.

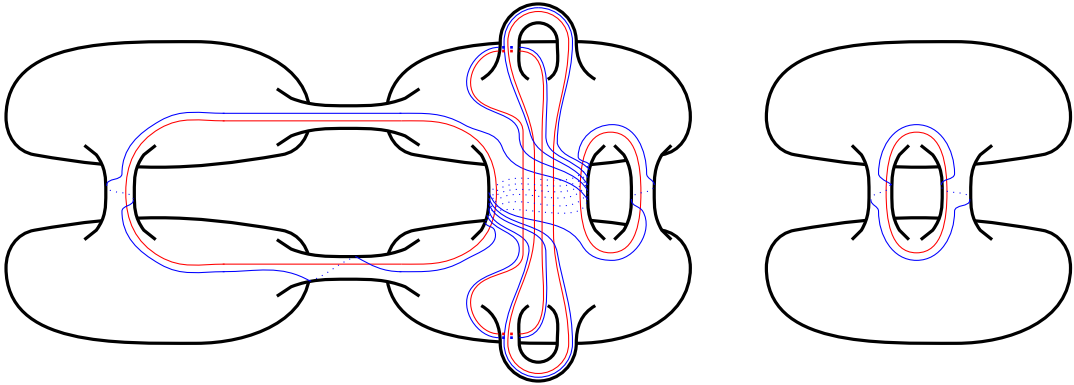
\begin{figure}[ht]
 \begin{center}
 \begin{tikzpicture}[scale=0.5]

 \def\geno{

     \draw[very thick] (2.25+1,1) to [out=10,in=270] (3+1,2.0)
     to [out=90, in=0] (0+1,4.0)
     to [out=180, in=90] (-3-1, 2)
     to [out=270, in=170] (-2.25-1,1) to[out=-10,in=190] (2.25+1,1);}
 \def\asa{
     \fill[white]
     (1.5,3.0)
     to [out=150, in=270] (1.0, 4.0)
     to [out=90, in=0.0] (0.0, 5.0)
     to [out=180, in=90] (-1.0,4.0)
     to [out=270, in=30] (-1.5,3.0)--
     (-0.25,3.0)
     to [out=150, in=270] (-0.5,4.0)
     to [out=90, in=180] (0,4.5)
     to [out=0, in=90] (0.5,4.0)
     to [out=270, in=30] (0.25, 3.0);

     \draw[very thick] (1.5,3.0)
     to [out=150, in=270] (1.0, 4.0)
     to [out=90, in=0.0] (0.0, 5.0)
     to [out=180, in=90] (-1.0,4.0)
     to [out=270, in=30] (-1.5,3.0);

     \draw[very thick] (0.25,3.0)
     to [out=30, in=270] (0.5,4.0)
     to [out=90, in=0] (0,4.5)
     to [out=180, in=90] (-0.5,4.0)
     to [out=270, in=150] (-0.25, 3.0);
     }

 \def\tubo{
     \fill[white] (-0.75,1.25) -- (-0.5,1)
     to [out=300, in=90] (-0.3,0)
     to [out=270, in = 60] (-0.5,-1)
     -- (-0.75,-1.25)-- (0.75,-1.25) -- (0.5,-1)
     to [out=-240, in=-90] (0.3,0)
     to [out=-270, in = -120] (0.5,1)
     -- (0.75,1.25)-- (-0.75,1.25);

     \draw[ very thick] (-0.75,1.25) -- (-0.5,1)
     to [out=300, in=90] (-0.3,0)
     to [out=270, in = 60] (-0.5,-1)
     -- (-0.75,-1.25);
     \draw[very thick] (0.75,1.25) -- (0.5,1)
     to [out=240, in=90] (0.3,0)
     to [out=270, in = 120] (0.5,-1)
     -- (0.75,-1.25);
     }

 \def\tubogordo{
     \fill[white] (-1.75,1.25) -- (-1.5,1)
     to [out=300, in=90] (-1.3,0)
     to [out=270, in = 60] (-1.5,-1)
     -- (-1.75,-1.25)-- (1.75,-1.25) -- (1.5,-1)
     to [out=-240, in=-90] (1.3,0)
     to [out=-270, in = -120] (1.5,1)
     -- (1.75,1.25)-- (-1.75,1.25);

     \draw[ very thick] (-1.75,1.25) -- (-1.5,1)
     to [out=300, in=90] (-1.3,0)
     to [out=270, in = 60] (-1.5,-1)
     -- (-1.75,-1.25);
     \draw[very thick] (1.75,1.25) -- (1.5,1)
     to [out=240, in=90] (1.3,0)
     to [out=270, in = 120] (1.5,-1)
     -- (1.75,-1.25);
     }

 \geno
 \asa

 \begin{scope}[yscale=-1]
 \geno
 \asa
 \end{scope}

 \tubogordo

 \begin{scope}[xshift=2.75cm]
 \tubo
 \end{scope}

 \begin{scope}[xshift=-10cm]
 \geno

 \begin{scope}[yscale=-1]
 \geno
 \end{scope}

 \begin{scope}[xshift=-1cm,xscale=1.25]
 \tubo

 \end{scope}

 \end{scope}

 \begin{scope}[xshift=10cm]
 \geno

 \begin{scope}[yscale=-1]
 \geno
 \end{scope}

 \begin{scope}[xshift=-1cm]
 \tubo
 \end{scope}

 \begin{scope}[xshift=1cm,xscale=1.25]
 \tubo
 \end{scope}

 \end{scope}

 \begin{scope}[xshift=-5cm, yshift=2cm, rotate=90, yscale=1.5]
 \tubo
 \end{scope}

 \begin{scope}[xshift=-5cm, yshift=-2cm, rotate=90, yscale=1.5]
 \tubo
 \end{scope}

 \def\curvaroja{
     \draw[red] (0.0, 1.6)
     to [out=180, in=90] (-0.8, 0.0)
     to [out=270, in=180] (0.0, -1.6)
     to [out=0, in=270] (0.8, 0.0)
     to [out=90, in=0] (0.0, 1.6);
     }
 \begin{scope}[xshift=-1.9cm]
 \draw[red]  (-1.7,-1.90)
 to [out=0, in=210] (0.0, -1.6)
 to [out=30, in=270] (0.8, 0.0)
 to [out=90, in=330] (0.0, 1.6)
 to [out=150, in=0] (-1.7, 1.9);

 \end{scope}

 \begin{scope}[xshift=1.9cm]
 \curvaroja
 \end{scope}

 \begin{scope}[xshift=9.95cm]
 \curvaroja
 \end{scope}

 \begin{scope}[xshift=-9.95cm]
 \draw[red] (1.7, 1.9)
 to [out=180, in=30] (0.0, 1.7)
 to [out=210, in=90] (-0.9, 0.0)
 to [out=270, in=150] (0.0, -1.7)
 to [out=330, in=180] (1.7, -1.9);

 \end{scope}

 \begin{scope}[xshift=-9.95cm]
 \draw[blue] (1.7, 2.1)
 to [out=180, in=30] (0.0, 1.9)
 to [out=210, in=90] (-1.0, 0.5)
 to [out=270, in=90] (-1.4, 0.2);
 \draw[blue, dotted] (-1.4, 0.2)
 to [out=270, in=90] (-0.7, -0.2);
 \draw[blue]  (-0.7,-0.2)
 to [out=270, in=90] (-1.0, -0.5)
 to [out=270, in=150] (0.0, -1.9)
 to [out=330, in=180] (1.7, -2.1);

 \end{scope}

 \begin{scope}[xshift=9.95cm]
 \draw[blue] (0.65, 0.2)
 to [out=90, in=270] (1.0, 0.5)
 to [out=90, in=0] (0.0, 1.8)
 to [out=180, in=90] (-1.0, 0.5)
 to [out=270, in=0] (-0.7, 0.2);
 \draw[blue, dotted] (-0.7, 0.2)
 to [out=270, in=90] (-1.2, -0.2);
 \draw[blue]  (-1.2,-0.2)
 to [out=270, in=90] (-1.0, -0.5)
 to [out=270, in=180] (0.0, -1.8)
 to [out=0, in=270] (1.0, -0.5)
 to [out=90, in=270] (1.4, -0.2);
 \draw[blue, dotted] (1.4, -0.2)
 to [out=90, in=270] (0.65, 0.2);
 \end{scope}

 \begin{scope}[xshift=-1.95cm]
 \draw[blue] (3.25, -0.2)
 to [out=120, in=290] (1., 1.2)
 to [out=110, in=330] (0.0, 1.8)
 to [out=150,  in=0] (-1.7, 2.1);

 \draw[blue]  (-1.7,-2.1)
 to [out=0,  in=210] (0.0, -1.8)
 to [out=30, in=290] (0.8, -1.0)
 to [out=110, in=290] (0.55, -0.8);
 \draw[blue, dotted] (3.25, -0.1)
 to [out=270, in=90] (0.55, -0.9);
 \end{scope}

 \begin{scope}[xshift=1.95cm]
 \draw[blue] (1.1, 0.2)
 to [out=90, in=270] (0.9, 0.5)
 to [out=90, in=0] (0.0, 1.8)
 to [out=180, in=90] (-0.9, 0.95)
 to [out=270, in=60] (-0.62, 0.7);
 \draw[blue, dotted] (-0.65, 0.7)
 to [out=240, in=60] (-3.2, -0.0);
 \draw[blue]  (-3.25,-0.0)
 to [out=290, in=120] (-0.8, -1.4)
 to [out=300, in=180] (0.0, -1.8)
 to [out=0, in=300] (0.8, -0.5)
 to [out=120, in=270] (0.5, -0.2);
 \draw[blue, dotted] (0.5, -0.2)
 to [out=90, in=270] (1.1, 0.2);
 \end{scope}

 \draw[red] (0.0, 4.8)
 to [out=180, in=90] (-0.8,4.0)
 to [out=270, in=90] (-0.1, 1.0)
 to [out=270, in=90] (-0.1,-1.0)
 to [out=270, in=90] (-0.8, -4.0)
 to [out=270, in=180] (0.0, -4.8)
 to [out=0, in=270] (0.8, -4.0)
 to [out=90, in=270] (0.4, -1.0)
 to [out=90, in=270] (0.4, 1.0)
 to [out=90, in=270] (0.8, 4.0)
 to [out=90, in=0] (0,4.8);

 \draw[blue] (0.0, 4.9)
 to [out=180, in=90] (-0.9,4.0)
 to [out=270, in=106] (-0.22, 1.6)
 to [out=283, in=120] (1.3,0.20);
 \draw[blue, dotted] (1.25,0.2)
 to [out=240, in=60] (-1.25,-0.4);
 \draw[blue] (-1.3,-0.4)
 to [out=270, in=80] (-0.25,-1.6)
 to [out=260, in=90] (-0.9, -4.0)
 to [out=270, in=180] (0.0, -4.9)
 to [out=0, in=270] (0.9, -4.0)
 to [out=90, in=280] (0.5,-1.6)
 to [out=100, in=290] (-1.3, -0.15);
 \draw[blue, dotted] (1.3,0.55)
 to [out=240, in=60] (-1.25,-0.15);
 \draw[blue] (1.35,0.60)
 to [out=130, in=290] (0.5, 1.3)
 to [out=90, in=270] (0.9, 4.0)
 to [out=90, in=0] (0,4.9);

 \draw[red] (-1.04,3.75)
 to [out=180, in=120] (-1.7,2.6)
 to [out=300, in=90] (-0.4,1.6)
 to [out=270, in=90] (-0.4,-1.6)
 to [out=270, in=60] (-1.7, -2.6)
 to [out=240, in=180] (-1.04, -3.75);
 \draw[line width=1,red, dotted] (-0.95, -3.75)
 -- (-0.55, -3.75);

 \draw[red] (-0.5, -3.75)
 to [out=0, in=270] (0.1,-1.6)
 to [out=90, in=270] (0.1,1.6)
 to [out=90, in=0.0] (-0.5,3.75);
 \draw[line width=1,red, dotted] (-0.95, 3.75)
 -- (-0.55, 3.75);

 \draw[blue] (-1.04,3.85)
 to [out=180, in=120] (-1.8,2.6)
 to [out=300, in=90] (-0.5,1.5)
 to [out=270, in=120] (1.3,0.05);
 \draw[blue, dotted] (1.3, 0.0)
 to [out=240, in=60] (-1.3,-0.6);
 \draw[blue] (-1.35,-0.6)
 to [out=290, in=90] (-0.5,-1.5)
 to [out=270, in=60] (-1.8, -2.6)
 to [out=240, in=180] (-1.04, -3.85);
 \draw[line width=1,blue, dotted] (-0.95, -3.85)
 -- (-0.55, -3.85);

 \draw[blue] (-0.5, -3.85)
 to [out=0, in=270] (0.2,-1.6)
 to [out=90, in=290] (-1.3, -0.3);
 \draw[blue, dotted] (-1.3, -0.3)
 to [out=60, in=240] (1.3, 0.4);
 \draw[blue] (1.3,0.4)
 to [out=120, in=270] (0.2,1.5)
 to [out=90, in=0.0] (-0.5,3.85);
 \draw[line width=1,blue, dotted] (-0.95, 3.85)
 -- (-0.55, 3.85);

 \draw[red] (-8.3,1.9) -- (-3.6,1.9);
 \draw[blue] (-8.3,2.1) -- (-3.6,2.1);
 \draw[red] (-8.3,-1.9) -- (-3.6,-1.9);
 \draw[blue] (-8.3, -2.1)
 to [out=0, in=150] (-6, -2.35);
 \draw[blue, dotted] (-6.0, -2.35)
 to  [out=20, in=200] (-4.8, -1.7);
 \draw[blue] (-4.8, -1.7)
 to [out=330, in =180] (-3.6, -2.1);

 \end{tikzpicture}

 \caption{Curves after adding one edge in step~\ref{g6}.}
 \end{center}

\end{figure}

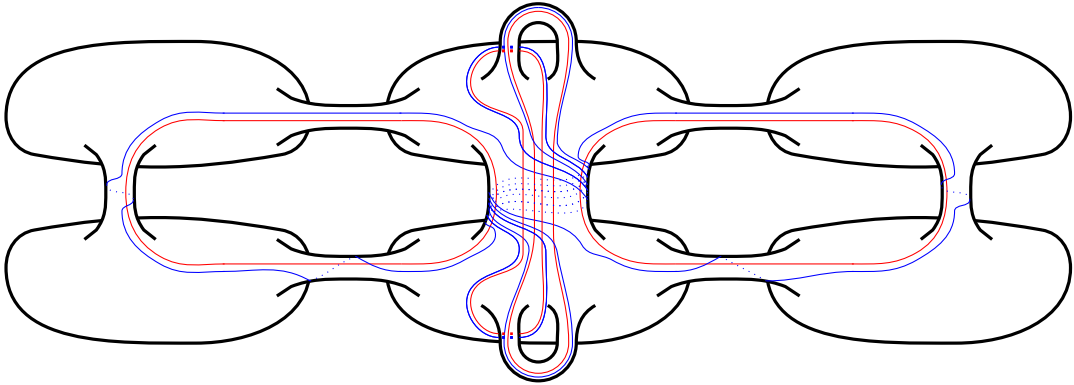
\begin{figure}[ht]
 \begin{center}
 \begin{tikzpicture}[scale=0.5]

 \def\geno{

     \draw[very thick] (2.25+1,1) to [out=10,in=270] (3+1,2.0)
     to [out=90, in=0] (0+1,4.0)
     to [out=180, in=90] (-3-1, 2)
     to [out=270, in=170] (-2.25-1,1) to[out=-10,in=190] (2.25+1,1);}
 \def\asa{
     \fill[white]
     (1.5,3.0)
     to [out=150, in=270] (1.0, 4.0)
     to [out=90, in=0.0] (0.0, 5.0)
     to [out=180, in=90] (-1.0,4.0)
     to [out=270, in=30] (-1.5,3.0)--
     (-0.25,3.0)
     to [out=150, in=270] (-0.5,4.0)
     to [out=90, in=180] (0,4.5)
     to [out=0, in=90] (0.5,4.0)
     to [out=270, in=30] (0.25, 3.0);

     \draw[very thick] (1.5,3.0)
     to [out=150, in=270] (1.0, 4.0)
     to [out=90, in=0.0] (0.0, 5.0)
     to [out=180, in=90] (-1.0,4.0)
     to [out=270, in=30] (-1.5,3.0);

     \draw[very thick] (0.25,3.0)
     to [out=30, in=270] (0.5,4.0)
     to [out=90, in=0] (0,4.5)
     to [out=180, in=90] (-0.5,4.0)
     to [out=270, in=150] (-0.25, 3.0);
     }

 \def\tubo{
     \fill[white] (-0.75,1.25) -- (-0.5,1)
     to [out=300, in=90] (-0.3,0)
     to [out=270, in = 60] (-0.5,-1)
     -- (-0.75,-1.25)-- (0.75,-1.25) -- (0.5,-1)
     to [out=-240, in=-90] (0.3,0)
     to [out=-270, in = -120] (0.5,1)
     -- (0.75,1.25)-- (-0.75,1.25);

     \draw[ very thick] (-0.75,1.25) -- (-0.5,1)
     to [out=300, in=90] (-0.3,0)
     to [out=270, in = 60] (-0.5,-1)
     -- (-0.75,-1.25);
     \draw[very thick] (0.75,1.25) -- (0.5,1)
     to [out=240, in=90] (0.3,0)
     to [out=270, in = 120] (0.5,-1)
     -- (0.75,-1.25);
     }

 \def\tubogordo{
     \fill[white] (-1.75,1.25) -- (-1.5,1)
     to [out=300, in=90] (-1.3,0)
     to [out=270, in = 60] (-1.5,-1)
     -- (-1.75,-1.25)-- (1.75,-1.25) -- (1.5,-1)
     to [out=-240, in=-90] (1.3,0)
     to [out=-270, in = -120] (1.5,1)
     -- (1.75,1.25)-- (-1.75,1.25);

     \draw[ very thick] (-1.75,1.25) -- (-1.5,1)
     to [out=300, in=90] (-1.3,0)
     to [out=270, in = 60] (-1.5,-1)
     -- (-1.75,-1.25);
     \draw[very thick] (1.75,1.25) -- (1.5,1)
     to [out=240, in=90] (1.3,0)
     to [out=270, in = 120] (1.5,-1)
     -- (1.75,-1.25);
     }

 \geno
 \asa

 \begin{scope}[yscale=-1]
 \geno
 \asa
 \end{scope}

 \tubogordo

 \begin{scope}[xshift=-10cm]
 \geno

 \begin{scope}[yscale=-1]
 \geno
 \end{scope}

 \begin{scope}[xshift=-1cm,xscale=1.25]
 \tubo

 \end{scope}

 \end{scope}

 \begin{scope}[xshift=10cm]
 \geno

 \begin{scope}[yscale=-1]
 \geno
 \end{scope}

 \begin{scope}[xshift=1cm,xscale=1.25]
 \tubo

 \end{scope}
 \end{scope}

 \begin{scope}[xshift=-5cm, yshift=2cm, rotate=90, yscale=1.5]
 \tubo
 \end{scope}

 \begin{scope}[xshift=-5cm, yshift=-2cm, rotate=90, yscale=1.5]
 \tubo
 \end{scope}

 \begin{scope}[xshift=5cm, yshift=2cm, rotate=90, yscale=1.5]
 \tubo
 \end{scope}

 \begin{scope}[xshift=5cm, yshift=-2cm, rotate=90, yscale=1.5]
 \tubo
 \end{scope}

 \def\curvaroja{
     \draw[red] (0.0, 1.6)
     to [out=180, in=90] (-0.8, 0.0)
     to [out=270, in=180] (0.0, -1.6)
     to [out=0, in=270] (0.8, 0.0)
     to [out=90, in=0] (0.0, 1.6);
     }
 \begin{scope}[xshift=-1.9cm]
 \draw[red]  (-1.7,-1.90)
 to [out=0, in=210] (0.0, -1.6)
 to [out=30, in=270] (0.8, 0.0)
 to [out=90, in=330] (0.0, 1.6)
 to [out=150, in=0] (-1.7, 1.9);

 \end{scope}

 \begin{scope}[xshift=1.9cm,xscale=-1]
 \draw[red] (-1.7, -1.9)
 to [out=0,  in=210] (0.0, -1.6)
 to [out=30, in=270] (0.8, 0.0)
 to [out=90, in=330] (0.0, 1.6)
 to [out=150,  in=0] (-1.7, 1.9);
 \end{scope}

 \begin{scope}[xshift=9.95cm]
 \draw[red] (-1.7, -1.9)
 to [out=0, in=210] (0.0, -1.6)
 to [out=30, in=270] (0.8, 0.0)
 to [out=90, in=330] (0.0, 1.6)
 to [out=150, in=0] (-1.7, 1.9);
 \end{scope}

 \begin{scope}[xshift=-9.95cm]
 \draw[red] (1.7, 1.9)
 to [out=180, in=30] (0.0, 1.7)
 to [out=210, in=90] (-0.9, 0.0)
 to [out=270, in=150] (0.0, -1.7)
 to [out=330, in=180] (1.7, -1.9);

 \end{scope}

 \begin{scope}[xshift=-9.95cm]
 \draw[blue] (1.7, 2.1)
 to [out=180, in=30] (0.0, 1.9)
 to [out=210, in=90] (-1.0, 0.5)
 to [out=270, in=90] (-1.4, 0.2);
 \draw[blue, dotted] (-1.4, 0.2)
 to [out=270, in=90] (-0.7, -0.2);
 \draw[blue]  (-0.7,-0.2)
 to [out=270, in=90] (-1.0, -0.5)
 to [out=270, in=150] (0.0, -1.9)
 to [out=330, in=180] (1.7, -2.1);

 \end{scope}

 \begin{scope}[xshift=9.95cm]
 \draw[blue] (0.65, 0.2)
 to [out=90, in=270] (1.0, 0.5)
 to [out=90, in=330] (0.0, 1.8)
 to [out=150, in=0] (-1.7, 2.1);

 \draw[blue]  (-1.7,-2.1)
 to [out=0,  in=210] (0.0, -1.8)
 to [out=30, in=270] (1.0, -0.5)
 to [out=90, in=270] (1.4, -0.2);
 \draw[blue, dotted] (1.4, -0.2)
 to [out=90, in=270] (0.65, 0.2);
 \end{scope}

 \begin{scope}[xshift=-1.95cm]
 \draw[blue] (3.25, -0.2)
 to [out=120, in=290] (1., 1.2)
 to [out=110, in=330] (0.0, 1.8)
 to [out=150,  in=0] (-1.7, 2.1);

 \draw[blue]  (-1.7,-2.1)
 to [out=0,  in=210] (0.0, -1.8)
 to [out=30, in=290] (0.8, -1.0)
 to [out=110, in=290] (0.55, -0.8);
 \draw[blue, dotted] (3.25, -0.1)
 to [out=270, in=90] (0.55, -0.9);
 \end{scope}

 \begin{scope}[xshift=1.95cm]
 \draw[blue] (1.7, 2.1)
 to [out=180, in=30] (0.0, 1.8)
 to [out=210, in=90] (-0.9, 0.95)
 to [out=270, in=60] (-0.62, 0.7);
 \draw[blue, dotted] (-0.65, 0.7)
 to [out=240, in=60] (-3.2, -0.0);
 \draw[blue]  (-3.25,-0.0)
 to [out=290, in=120] (-0.8, -1.4)
 to [out=300, in=150] (0.0, -1.8)
 to [out=330,  in=180] (1.7, -2.1);

 \end{scope}

 \draw[red] (0.0, 4.8)
 to [out=180, in=90] (-0.8,4.0)
 to [out=270, in=90] (-0.1, 1.0)
 to [out=270, in=90] (-0.1,-1.0)
 to [out=270, in=90] (-0.8, -4.0)
 to [out=270, in=180] (0.0, -4.8)
 to [out=0, in=270] (0.8, -4.0)
 to [out=90, in=270] (0.4, -1.0)
 to [out=90, in=270] (0.4, 1.0)
 to [out=90, in=270] (0.8, 4.0)
 to [out=90, in=0] (0,4.8);

 \draw[blue] (0.0, 4.9)
 to [out=180, in=90] (-0.9,4.0)
 to [out=270, in=106] (-0.22, 1.6)
 to [out=283, in=120] (1.3,0.20);
 \draw[blue, dotted] (1.25,0.2)
 to [out=240, in=60] (-1.25,-0.4);
 \draw[blue] (-1.3,-0.4)
 to [out=270, in=80] (-0.25,-1.6)
 to [out=260, in=90] (-0.9, -4.0)
 to [out=270, in=180] (0.0, -4.9)
 to [out=0, in=270] (0.9, -4.0)
 to [out=90, in=280] (0.5,-1.6)
 to [out=100, in=290] (-1.3, -0.15);
 \draw[blue, dotted] (1.3,0.55)
 to [out=240, in=60] (-1.25,-0.15);
 \draw[blue] (1.35,0.60)
 to [out=130, in=290] (0.5, 1.3)
 to [out=90, in=270] (0.9, 4.0)
 to [out=90, in=0] (0,4.9);

 \draw[red] (-1.04,3.75)
 to [out=180, in=120] (-1.7,2.6)
 to [out=300, in=90] (-0.4,1.6)
 to [out=270, in=90] (-0.4,-1.6)
 to [out=270, in=60] (-1.7, -2.6)
 to [out=240, in=180] (-1.04, -3.75);
 \draw[line width=1,red, dotted] (-0.95, -3.75)
 -- (-0.55, -3.75);

 \draw[red] (-0.5, -3.75)
 to [out=0, in=270] (0.1,-1.6)
 to [out=90, in=270] (0.1,1.6)
 to [out=90, in=0.0] (-0.5,3.75);
 \draw[line width=1,red, dotted] (-0.95, 3.75)
 -- (-0.55, 3.75);

 \draw[blue] (-1.04,3.85)
 to [out=180, in=120] (-1.8,2.6)
 to [out=300, in=90] (-0.5,1.5)
 to [out=270, in=120] (1.3,0.05);
 \draw[blue, dotted] (1.3, 0.0)
 to [out=240, in=60] (-1.3,-0.6);
 \draw[blue] (-1.35,-0.6)
 to [out=290, in=90] (-0.5,-1.5)
 to [out=270, in=60] (-1.8, -2.6)
 to [out=240, in=180] (-1.04, -3.85);
 \draw[line width=1,blue, dotted] (-0.95, -3.85)
 -- (-0.55, -3.85);

 \draw[blue] (-0.5, -3.85)
 to [out=0, in=270] (0.2,-1.6)
 to [out=90, in=290] (-1.3, -0.3);
 \draw[blue, dotted] (-1.3, -0.3)
 to [out=60, in=240] (1.3, 0.4);
 \draw[blue] (1.3,0.4)
 to [out=120, in=270] (0.2,1.5)
 to [out=90, in=0.0] (-0.5,3.85);
 \draw[line width=1,blue, dotted] (-0.95, 3.85)
 -- (-0.55, 3.85);

 \draw[red] (-8.3,1.9) -- (-3.6,1.9);
 \draw[blue] (-8.3,2.1) -- (-3.6,2.1);
 \draw[red] (-8.3,-1.9) -- (-3.6,-1.9);
 \draw[blue] (-8.3, -2.1)
 to [out=0, in=150] (-6, -2.35);
 \draw[blue, dotted] (-6.0, -2.35)
 to  [out=20, in=200] (-4.8, -1.7);
 \draw[blue] (-4.8, -1.7)
 to [out=330, in =180] (-3.6, -2.1);

 \draw[red] (8.3,1.9) -- (3.6,1.9);
 \draw[blue] (8.3,2.1) -- (3.6,2.1);
 \draw[red] (8.3,-1.9) -- (3.6,-1.9);
 \draw[blue] (8.3, -2.1)
 to [out=180, in=0] (6, -2.35);
 \draw[blue, dotted] (6.0, -2.35)
 to  [out=150, in=330] (4.8, -1.7);
 \draw[blue] (4.8, -1.7)
 to [out=210, in=0] (3.6, -2.1);

 \draw [dotted, red, line width=1] (-0.95, 3.75)
 -- (-0.55, 3.75);

 \draw[blue] (-1.04,3.85)
 to [out=180, in=120] (-1.8,2.6)
 to [out=300, in=90] (-0.5,1.5)
 to [out=270, in=120] (1.3,0.05);
 \draw[blue, dotted] (1.3, 0.0)
 to [out=240, in=60] (-1.3,-0.6);
 \draw[blue] (-1.35,-0.6)
 to [out=290, in=90] (-0.5,-1.5)
 to [out=270, in=60] (-1.8, -2.6)
 to [out=240, in=180] (-1.04, -3.85);
 \draw[line width=1,blue, dotted] (-0.95, -3.85)
 -- (-0.55, -3.85);

 \draw[blue] (-0.5, -3.85)
 to [out=0, in=270] (0.2,-1.6)
 to [out=90, in=290] (-1.3, -0.3);
 \draw[blue, dotted] (-1.3, -0.3)
 to [out=60, in=240] (1.3, 0.4);
 \draw[blue] (1.3,0.4)
 to [out=120, in=270] (0.2,1.5)
 to [out=90, in=0.0] (-0.5,3.85);
 \draw[line width=1,blue, dotted] (-0.95, 3.85)
 -- (-0.55, 3.85);

 \end{tikzpicture}
 \caption{Curves after adding the second edge in step~\ref{g6}}
 \end{center}

\end{figure}

\item\label{g7} If the edge creates a loop, we choose cylinders and substitute them by new
ones as before. The two red lines $\delta_r$ and $\gamma_r$ are substituted by two new ones.
The first one is constructed as in~\ref{g6}. In order to construct the second one
we choose (arbitrarily) one of the old ones, say $\gamma_r$; it can be decomposed as $\gamma_1\cdot\gamma_2$ where $\gamma_2$ is the the path contained
in the tube which is going to disappear. We take two parallel copies of $\gamma_1$ and we  connect them by turning around the new cylinders in such a way that
the resulting curve is disjoint with the first red curve.
%
Let $\delta_b$ and $\gamma_b$ be the two old blue lines.
As before, a new blue line is obtained by
gluing $\delta_b$ and $\gamma_b$ as in~\ref{g6}, going parallel to the corresponding red one in one of
the new cylinders and performing a Dehn twist along the other one. The second
new blue line is created from one of the preexisting ones (say $\delta_b$ in this example) as we did for the red one. That is, we decompose $\delta_b$ as
$\delta_1\cdot\delta_2$ where $\delta_2$ is the the path contained
in the tube which is going to disappear; we take two parallel copies of $\delta_1$ and we  connect them by turning around the new cylinders in such a way that
the resulting curve is disjoint with the first blue curve.
\end{enumerate}

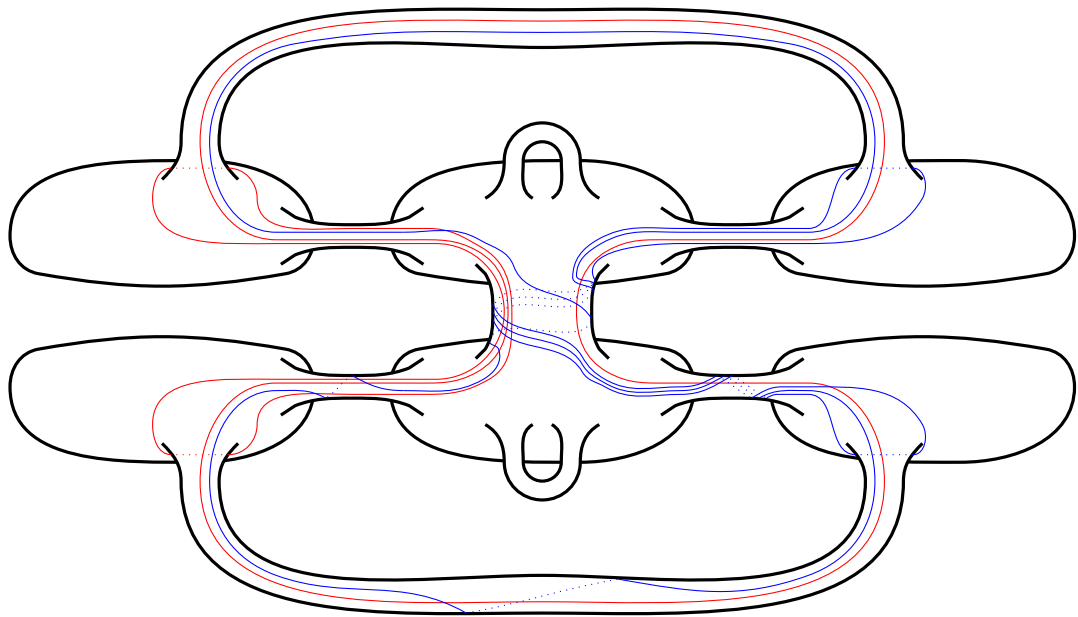
\begin{figure}[ht]
\begin{center}
\begin{tikzpicture}[scale=0.5]

\def\geno{

    \draw[very thick] (2.25+1,1) to [out=10,in=270] (3+1,2.0)
    to [out=90, in=0] (0+1,4.0)
    to [out=180, in=90] (-3-1, 2)
    to [out=270, in=170] (-2.25-1,1) to[out=-10,in=190] (2.25+1,1);}
\def\asa{
    \fill[white]
    (1.5,3.0)
    to [out=150, in=270] (1.0, 4.0)
    to [out=90, in=0.0] (0.0, 5.0)
    to [out=180, in=90] (-1.0,4.0)
    to [out=270, in=30] (-1.5,3.0)--
    (-0.25,3.0)
    to [out=150, in=270] (-0.5,4.0)
    to [out=90, in=180] (0,4.5)
    to [out=0, in=90] (0.5,4.0)
    to [out=270, in=30] (0.25, 3.0);

    \draw[very thick] (1.5,3.0)
    to [out=150, in=270] (1.0, 4.0)
    to [out=90, in=0.0] (0.0, 5.0)
    to [out=180, in=90] (-1.0,4.0)
    to [out=270, in=30] (-1.5,3.0);

    \draw[very thick] (0.25,3.0)
    to [out=30, in=270] (0.5,4.0)
    to [out=90, in=0] (0,4.5)
    to [out=180, in=90] (-0.5,4.0)
    to [out=270, in=150] (-0.25, 3.0);
    }

 \def\tubo{
     \fill[white] (-0.75,1.25) -- (-0.5,1)
     to [out=300, in=90] (-0.3,0)
     to [out=270, in = 60] (-0.5,-1)
     -- (-0.75,-1.25)-- (0.75,-1.25) -- (0.5,-1)
     to [out=-240, in=-90] (0.3,0)
     to [out=-270, in = -120] (0.5,1)
     -- (0.75,1.25)-- (-0.75,1.25);

     \draw[ very thick] (-0.75,1.25) -- (-0.5,1)
     to [out=300, in=90] (-0.3,0)
     to [out=270, in = 60] (-0.5,-1)
     -- (-0.75,-1.25);
     \draw[very thick] (0.75,1.25) -- (0.5,1)
     to [out=240, in=90] (0.3,0)
     to [out=270, in = 120] (0.5,-1)
     -- (0.75,-1.25);
     }

 \def\tubogordo{
     \fill[white] (-1.75,1.25) -- (-1.5,1)
     to [out=300, in=90] (-1.3,0)
     to [out=270, in = 60] (-1.5,-1)
     -- (-1.75,-1.25)-- (1.75,-1.25) -- (1.5,-1)
     to [out=-240, in=-90] (1.3,0)
     to [out=-270, in = -120] (1.5,1)
     -- (1.75,1.25)-- (-1.75,1.25);

     \draw[ very thick] (-1.75,1.25) -- (-1.5,1)
     to [out=300, in=90] (-1.3,0)
     to [out=270, in = 60] (-1.5,-1)
     -- (-1.75,-1.25);
     \draw[very thick] (1.75,1.25) -- (1.5,1)
     to [out=240, in=90] (1.3,0)
     to [out=270, in = 120] (1.5,-1)
     -- (1.75,-1.25);
     }

 \def\tubopegado{
     \fill[white] (10, 3.5)
     to [out=135, in = 270] (9.5, 4.5)
     to [out=90, in=0] (0, 8)
     to [out=180, in=90] (-9.5,4.5)
     to [out=270, in=45]  (-10, 3.5) -- (-8, 3.5)
     to [out=135, in=270] (-8.5, 4.5)
     to [out=90, in=180] (0,6.5)
     to [out=0, in=90] (8.5,4.5)
     to [out=270, in=45] (8, 3.5);

     \draw[very thick] (10, 3.5)
     to [out=135, in = 270] (9.5, 4.5)
     to [out=90, in=0] (0, 8)
     to [out=180, in=90] (-9.5,4.5)
     to [out=270, in=45]  (-10, 3.5);
     \draw[very thick] (-8, 3.5)
     to [out=135, in=270] (-8.5, 4.5)
     to [out=90, in=180] (0,7)
     to [out=0, in=90] (8.5,4.5)
     to [out=270, in=45] (8, 3.5);
     }

 \geno
 \asa

 \begin{scope}[yscale=-1]
 \geno
 \asa
 \end{scope}

 \tubogordo

 \begin{scope}[xshift=-10cm]
 \geno

 \begin{scope}[yscale=-1]
 \geno
 \end{scope}

 \end{scope}

 \begin{scope}[xshift=10cm]
 \geno

 \begin{scope}[yscale=-1]
 \geno
 \end{scope}

 \end{scope}

 \begin{scope}[xshift=-5cm, yshift=2cm, rotate=90, yscale=1.5]
 \tubo
 \end{scope}

 \begin{scope}[xshift=-5cm, yshift=-2cm, rotate=90, yscale=1.5]
 \tubo
 \end{scope}

 \begin{scope}[xshift=5cm, yshift=2cm, rotate=90, yscale=1.5]
 \tubo
 \end{scope}

 \begin{scope}[xshift=5cm, yshift=-2cm, rotate=90, yscale=1.5]
 \tubo
 \end{scope}

 \tubopegado
 \begin{scope}[yscale=-1]
 \tubopegado
 \end{scope}

 \draw[red]  (-3.4,-1.90)
 to [out=0, in=210] (-1.7, -1.6)
 to [out=30, in=270] (-0.9, 0.0)
 to [out=90, in=330] (-1.7, 1.6)
 to [out=150, in=0] (-3.4, 1.9)
 --  (-7,1.9)
 to [out=180, in=270] (-9, 4.5)
 to [out=90, in=180] (0, 7.7)
 to [out=0, in=90] (9,4.5)
 to [out=270, in=0] (7,1.9)
 -- (3.4, 1.9)
 to [out=180, in=30] (1.7,1.6)
 to [out=210, in=90] (0.9,0.0)
 to [out=270, in=150] (1.7, -1.6)
 to [out=330, in=180] (3.4, -1.9)
 --(7, -1.9)
 to [out=0, in=90] (9, -4.5)
 to [out=270, in=0] (0, -7.7)
 to [out=180, in=270] (-9,-4.5)
 to [out= 90, in=180] (-7, -1.9)
 --  (-3.4,-1.9);

 \draw[red]  (-8.25, -3.8)
 to [out=0, in=270] (-7.5, -3)
 to [out=90, in=180] (-5,-2.2)
 -- (-3, -2.2)
 to [out=0, in=210] (-1.4, -1.8)
 to [out=30, in=270] (-0.8, 0.0)
 to [out=90, in=330] (-1.6, 1.7)
 to [out=150, in=0] (-3, 2.2)
 -- (-5, 2.2)
 to [out=180, in=270] (-7.5, 3)
 to [out=90, in=0] (-8.25,3.8);
 \draw[red, dotted] (-8.25, 3.8) --(-9.75,3.8);
 \draw[red] (-9.75, 3.8)
 to [out=180, in=90] (-10.25, 3)
 to [out=270, in=180] (-7,1.8)
 -- (-2.8,1.8)
 to [out=0, in=90] (-1, 0.0)
 to [out=270, in=0] (-2.8,-1.8)
 -- (-7, -1.8)
 to [out=180, in=90] (-10.25,-3)
 to [out=270, in=180] (-9.75,-3.8);
 \draw[red, dotted] (-8.25, -3.8) --(-9.75,-3.8);

 \draw[blue] (1.3, -0.2)
 to [out=120, in=290] (-0.75, 1.2)
 to [out=110, in=330] (-1.65, 1.9)
 to [out=150,  in=0] (-3.65, 2.1);

 \draw[blue]  (-3.65,-2.1)
 to [out=0,  in=210] (-1.55, -1.8)
 to [out=30, in=290] (-1.15, -1.0)
 to [out=110, in=290] (-1.4, -0.8);
 \draw[blue, dotted] (1.3, -0.1)
 to [out=270, in=90] (-1.4, -0.9);

 \draw[blue] (-3.65, -2.1)
to [out=180, in=330] (-5,-1.7);
\draw[blue, dotted] (-5,-1.7)
to [out=210, in=30] (-5.7, -2.3);
\draw[blue] (-5.7,-2.3)
 to [out=150, in=0] (-7,-2.1)
 to [out=180, in=90] (-8.75, -4.5)
 to [out=270, in=160] (-7, -7)
 to [out=340, in=150] (-2, -8);
 \draw[blue, dotted] (-2,-8)
 to [out=10, in=190] (2,-7.1);
 \draw[blue] (2,-7.1)
 to [out=350, in=200] (7,-7)
 to [out=20, in=270] (8.75, -4.5)
 to [out=90, in=0] (7, -2.1)
 to [out=180, in=30] (5.7, -2.3);

 \draw[blue] (-3.65, 2.1)
 to [out=180, in=0] (-7,2.1)
 to [out=180, in=270] (-8.75, 4.5)
 to [out=90, in = 190] (-6.5,7.1)
 to [out=10, in=180] (0,7.4)
 to [out=0, in=170] (6.5,7.1)
 to [out=-10, in=90] (8.75, 4.5)
 to [out=270, in=0] (7,2.1)
 to [out=180, in=0] (3.65, 2.1);

 \draw[blue] (3.65, 2.1)
 to [out=180, in=30] (1.55, 1.8)
 to [out=210, in=90] (0.9, 0.95)
 to [out=270, in=60] (1.33, 0.7);
 \draw[blue, dotted] (1.3, 0.7)
 to [out=240, in=60] (-1.25, -0.0);
 \draw[blue]  (-1.3,-0.0)
 to [out=290, in=120] (1.05, -1.4)
 to [out=300, in=150] (1.65, -1.9)
 to [out=330,  in=180] (3.65, -2.1)
 to [out=0, in=210] (4.8, -1.7);
\draw[blue, dotted] (4.8, -1.7)
to  [out=330, in=150] (5.7, -2.3);

 \draw[blue] (8.25, 3.8)
 to [out=180, in=0]
 (7, 2.2) -- (3.65, 2.2)
 to [out=180, in=30] (1.45, 1.9)
 to [out=210, in=90] (0.8, 0.9)
 to [out=270, in=60] (1.33, 0.5);
 \draw[blue] (9.75, 3.8)
 to [out=0, in= 60] (10, 3)
 to [out=240, in=0] (7,1.8)
 -- (3.65, 1.8)
 to [out=180, in=120] (1.4, 0.8);

 \draw[blue, dotted] (1.4, 0.9)
 to [out=240, in=60] (-1.25, 0.2);
 \draw[blue, dotted] (1.3, 0.5)
 to [out=240, in=60] (-1.25, -0.2);

 \draw[blue]  (-1.3,-0.2)
 to [out=290, in=120] (1.05, -1.6)
 to [out=300, in=150] (1.65, -2)
 to [out=330,  in=180] (3.65, -2.2)
to [out=0, in=210] (5.0, -1.7);
\draw[blue, dotted] (5,-1.7)
to [out=330, in=150] (5.9,-2.3);
\draw[blue] (5.9, -2.3)
to[out=30, in=180] (6.2,-2.2)
 -- (7, -2.2)
 to [out=0, in=180] (8.25,-3.8);

 \draw[blue]  (-1.3,0.2)
 to [out=290, in=120] (1.05, -1.2)
 to [out=300, in=150] (1.65, -1.8)
 to [out=330,  in=180] (3.65, -2)
to [out=0, in=210] (4.6, -1.7);
\draw[blue, dotted] (4.6,-1.7)
to [out=330, in=150] (5.5,-2.3);
\draw[blue] (5.5, -2.3)
to [out=30, in=180] (6.3,-2)
 -- (7, -2)
 to [out=0, in=120] (10, -3)
 to [out=300, in=0] (9.75, -3.8);

\draw[blue, dotted] (9.75, -3.8) -- (8.25, -3.8);
\draw[blue, dotted] (9.75, 3.8) -- (8.25, 3.8);

 \end{tikzpicture}
 \caption{Curves after step~\ref{g7}. The Handle curves have been omited for clarity.}
\end{center}

\end{figure}

\section{Plumbing graph of a graph manifold}\label{sec:graph-manifold}

We recall the needed facts of Neumann's plumbing construction~\cite{neu:81} of 
Waldhausen graph manifolds~\cite{wal:67,wal:67a}.
Everything in this section is known but we recall
it in order to fix notations. The \emph{atoms} of these constructions
are $\mathbb{S}^1$-fiber bundles. Since the actual family we are interested in satisfies
strong orientation properties we will restrict our attention to oriented
graph manifolds built up using oriented fibrations.

Let $\pi:M\to S$ be an oriented $\bs^1$-bundle 
over a closed oriented surface $S$ of genus~$g$.
The oriented $\bs^1$-bundles over a manifold~$N$ are classified by its Euler class in $H^2(N;\bz)$.
If $S$ is an oriented closed surface there is a natural identification $\bz\equiv H^2(S;\bz)$
and the Euler class is interpreted as an Euler number
$e\in\bz$. Let us recall for further use how to compute this number. Because of the Euler class classification,
any  oriented $\bs^1$-bundle over an oriented surface with boundary is homeomorphic to a product.

Let us consider a \emph{small} closed disk $D\subseteq S$ 
and consider the surface with boundary $\check{S}:=\overline{S\setminus D}$.
The restrictions of $\pi$ over $D$ and $\check{S}$  are product bundles. Let $\mu_1$
be the boundary of a meridian disk of the solid torus $\pi^{-1}(D)$ (oriented accordingly
as $\partial D$) and let $s_1$ be the boundary of a section defined over $\check{S}$
(oriented as $\partial\check{S}$). These two simple closed curves define
elements in $H_1(\pi^{-1}(\partial D);\bz)$ as an oriented fiber $\phi_1$ does.
Let us use multiplicative notation for $H_1(\pi^{-1}(\partial D);\bz)$. The fact
that $\mu_1$ and $s_1$ project onto opposite generators of $H_1(\partial D;\bz)$
implies that these elements satisfy a relation
\begin{equation}\label{eq:euler}
s_1\cdot\mu_1\cdot\phi^{e}=1 
\end{equation}
for some $e\in\mathbb{Z}$, which happens to be the Euler number of the fibration.
There are several variations of this construction. The first one is very simple, we can
replace $D,\check{S}$ by two surfaces $S_1,S_2$ with common connected boundary
such that $S=S_1\cup S_2$ and the formula~\eqref{eq:euler} is still true.
Moreover, there is  no need to assume that the their boundaries are connected. 
Assume that $\partial S_1=\partial S_2=S_1\cap S_2$ has $r$~connected components
$C_1,\dots,C_r$; let us fix sections $s_i:S_i\to M$, $i=1,2$, and let
us denote by $s_i^j$ the boundary of such section in $C_j$ (oriented as $\partial S_i$).
Then in $H_1(C_j;\bz)$ we have inequalities
\begin{equation}
s_1^j\cdot s_2^j\cdot\phi^{e_j}=1,\quad e_j\in\mathbb{Z}, 
\end{equation}
and $e=e_1+\dots+e_r$.

Moreover, any decomposition of $e$ as above, can be realized in this way for a given oriented $\mathbb{S}^1$-bundle with Euler number~$e$.

A plumbing graph $(\Gamma,g,e,o)$ is given by a (connected) graph $\Gamma$ (without loops), a \emph{genus} function
$g:V(\Gamma)\to\bz_{\geq 0}$ (where $V(\Gamma)$ is the set of vertices of~$\Gamma$), an \emph{Euler} function
$e:V(\Gamma)\to\bz$ and an \emph{orientation} class $o\in H^1(\Gamma;\bz/2)$. We usually represent
this graph by decorating each vertex $v$ with $[g(v)]$ and $e(v)$, and by decorating each edge~$e$ with a sign~$\sigma_e=\pm$ representing the coefficients of a cocycle (cochain)
representing~$o$. If the decoration $[g(v)]$ is not written it means that
$g(v)=0$, and empty decoration of an edge~$e$ means $+$-decoration.

\begin{obs}\label{obs:coborde}
If we change a  cocycle by reversing the signs of all the edges adjacent to a fixed vertex, we obtain another representative
of~$o$; moreover, we can pass from one representative to another by a sequence of these moves. Of course, if $\Gamma$ is a
tree the $o$-decoration can be chosen as void. 
\end{obs}
The plumbing manifold associated to $(\Gamma,g,e,o)$ is constructed as follows. First, we collect for each
$v\in V(\Gamma)$ an oriented $\mathbb{S}^1$-bundle $\pi_v:M_v\to S_v$ with Euler number $e(v)$ and
such that $S_v$ is a closed oriented surface of genus~$g(v)$.
For each edge $\eta$ with end points $v,w$ we collect two closed disks $D_v^\eta\subset S_v$ and $D_w^\eta\subset S_w$.
We choose these disks such that they are pairwise disjoint for any fixed~$v$. 
Let us define $\check{M}_v$ to be the closure of $M_v\setminus\bigcup_{v\in\eta}\pi_v^{-1}(D_v^\eta)$, which is an oriented
manifold whose boundary is composed by tori, as many as the valency of $v$ in~$\Gamma$. 
We define then $T_v^\eta:=\pi_v^{-1}(\partial D_v^\eta)$. In each one of these tori we have a pair of curves $(\phi_v^\eta,\mu_v^\eta)$,
where $\mu_v^\eta$ is a meridian of the solid torus $\pi_v^{-1}(D_v^\eta)$ (oriented as $\partial D_v^\eta$)
and $\phi_v^\eta$ is an oriented fiber. Note that these curves induce a basis of $H_1(T_v^\eta;\bz)$
which represents the orientation of $T_v^\eta$ as part of the boundary of $\check{M}_v$.

Let us consider a homeomorphism $\Phi_{v,w}^\eta:T_v^\eta\to T_w^\eta$ such that 
$\Phi_{v,w}^\eta(\phi_v^\eta)=(\mu_w^\eta)^{\sigma_\eta}$ and $\Phi_{v,w}^\eta(\mu_v^\eta)=(\phi_w^\eta)^{\sigma_\eta}$. Basically, we are exchanging
sections and fibers (twisted by the sign). This map
is determined up to isotopy by the matrix $\sigma_\eta\left(\begin{smallmatrix}
0&1\\
1&0  
\end{smallmatrix}\right)$ of determinant~$-1$. These maps are well-defined
only up to isotopy and we can choose representatives such that  $\Phi_{w,v}^\eta=(\Phi_{v,w}^\eta)^{-1}$.
Then the plumbing manifold associated to $(\Gamma,g,e,o)$ is defined as:
$$
\left(\coprod_{v\in V(\Gamma)} \check{M}_v\right) \Big/\{\Phi_{v,w}^\eta\}_{\eta}
$$
We will drop any reference to $o$ if it is trivial.

\begin{obs}
Note that the above construction depends on a fixed choice of a cocycle.
Let us fix a vertex~$v$ and consider the cocycle $\tilde{\sigma}$
given by
\[
\tilde{\sigma}_\eta=
\begin{cases}
\sigma_\eta&\text{ if }v\notin\eta\\
-\sigma_\eta&\text{ if }v\in\eta
\end{cases}
\]
For the construction associated to~$\tilde{\sigma}$ we keep the fibrations
for $w\neq v$ and we consider the fibration $\tilde{\pi}_v:M_v\to(-S_v)$
which is the opposite fibration to~$\pi_v$ but the orientation of $M_v$
remains unchanged. As a consequence
$\tilde\phi_v^\eta=(\phi_v^\eta)^{-1}$ and $\tilde\mu_v^\eta=(\mu_v^\eta)^{-1}$, when $v\in\eta$. Note that
$\tilde{\Phi}_{v,w}^\eta=\Phi_{v,w}^\eta$ and the resulting
manifold is the same as above. Hence, by Remark~\ref{obs:coborde},
the manifold depends only on~$o$ and not on the particular choice of
a representative cocycle.
\end{obs}

\begin{ejm}
Let $X$ be a complex surface and let $D=\bigcup_{j=1}^r D_j$ be a normal crossing compact divisor in $X$.
Let $\Gamma$ be the dual graph of~$D$ and define the functions $g,e$ as the genus and self-intersection.
Then the boundary of a regular neighbourhood of~$D$ is homeomorphic to the graph manifold of $(\Gamma,g,e)$.
If the intersection matrix of~$D$ is negative definite then $D$ can be obtained as the exceptional divisor of
a resolution of an isolated surface singularity. That is, the link of an isolated surface singularity is always
a plumbing manifold, whose graph is the dual graph of the resolution.
This example is the main motivation for this work.
\end{ejm}

The rest of the paper is devoted to proof that the construction
of \S\ref{sec:introduction} provides a Heegaard splitting
of the corresponding graph manifold described in this section.

\section{Topological constructions}\label{sec:TopConst}

In this section we introduce different constructions which will be used in the sequel.

\subsection{Drilled bodies}
\mbox{}

\begin{dfn}
A \emph{$(g,n)$-drilled body} is a product $H_{g,n}:=\Sigma_{g,n}\times I$, where $I:=[0,1]$ and
$\Sigma_{g,n}$ is an oriented compact surface of genus~$g$ and $n$~boundary components, with $n>0$.
\end{dfn}

\begin{figure}[ht]
\begin{subfigure}[b]{.3\textwidth}
\begin{tikzpicture}[scale=.31,y=-1cm,x=1cm]
\tikzstyle{flecha}=[-angle 45 new, arrow head=3mm]
\node at (4.2,3) {$a$};
\node at (11,3) {$a$};
\node at (4,17.1) {$a$};
\node at (12.3,17.3) {$a$};
\node at (8,1.4) {$b$};
\node at (7,6) {$b$};
\node at (8,15.2) {$b$};
\node at (6,20) {$b$};
\node at (2.5,12) {$f$};
\node at (9,12) {$f$};
\node at (5,11) {$f$};
\node at (13.2,11) {$f$};
\draw[flecha,black] (9.8425,10.668) -- (9.8425,10.4775);
\draw[flecha,black] (12.38038,10.43093) -- (12.38038,10.24043);
\draw[flecha,black] (5.71288,11.00243) -- (5.71288,10.81193);
\draw[flecha,black] (3.1623,13.01327) -- (3.1623,12.82277);
\draw[flecha,black] (7.07813,5.06518) -- (7.71313,5.06518);
\draw[flecha,black] (8.79263,2.21827) -- (9.42763,2.21827);
\draw[flecha,black] (8.22113,16.18827) -- (8.85613,16.18827);
\draw[flecha,black] (6.31613,19.04577) -- (6.95113,19.04577);
\draw[flecha,black] (11.2903,17.41593) -- (11.59298,17.07303);
\draw[flecha,black] (4.59105,17.4371) -- (4.89373,17.0942);
\draw[flecha,black] (4.37938,3.71052) -- (4.68207,3.36762);
\draw[flecha,black] (11.08922,3.67877) -- (11.3919,3.33587);

\draw[thick,black] (12.3825,2.2225) -- (12.3825,16.1925);
\draw[thick,dashed,black] (5.715,2.2225) -- (5.715,16.1925);
\draw[thick,black] (3.175,5.08) -- (3.175,19.05);
\draw[thick,black] (9.84038,5.12233) -- (9.84038,19.09233);
\draw[thick,black] (3.175,5.08) -- (9.8425,5.08) -- (12.3825,2.2225) -- (5.715,2.2225) -- cycle;
\draw[thick,dashed,black] (3.175,19.05) -- (5.715,16.1925) -- (12.3825,16.1925);
\draw[thick,black] (3.175,19.05) -- (9.8425,19.05) -- (12.3825,16.1925);
\end{tikzpicture}%
\caption{($\bs^1)^2\times I=\Sigma_{1,0}\times I$}
\label{fig1:subfig1}
\end{subfigure}
\begin{subfigure}[b]{.3\textwidth}
\begin{tikzpicture}[scale=.31,y=-1cm]
\tikzstyle{flecha}=[-latex new, arrow head=2.1mm]
\draw[very thick,black] (9.88757,4.92473) +(-44:0.70107) arc (-44:-191:0.70107);
\draw[very thick,black] (9.9045,18.89473) +(-44:0.70107) arc (-44:-191:0.70107);
\draw[thick,black] (5.32286,1.75024) +(99:1.12874) arc (99:23:1.12874);
\draw[thick,dashed,black] (5.32286,15.72024) +(99:1.12874) arc (99:23:1.12874);
\draw[thick,black] (3.83534,4.99116) +(6:0.60476) arc (6:-96:0.60476);
\draw[thick,black] (3.83534,18.96116) +(6:0.60476) arc (6:-96:0.60476);
\draw[thick,black] (3.83534,12.35716) +(6:0.60476) arc (6:-96:0.60476);
\draw[thick,dashed,black] (5.32286,10.98949) +(99:1.12874) arc (99:23:1.12874);
\draw[very thick,black] (9.90238,13.00617) +(-44:0.70107) arc (-44:-191:0.70107);
\draw[thick,black] (11.79552,2.14075) +(92:0.70871) arc (92:175:0.70871);
\draw[thick,dashed,black] (11.77858,16.12768) +(92:0.70871) arc (92:175:0.70871);
\draw[thick,dashed,black] (11.77858,10.39575) +(92:0.70871) arc (92:175:0.70871);
\draw[thick,black] (3.80153,4.39843) -- (3.80153,18.36843);
\draw[thick,black] (10.42247,4.49157) -- (10.42247,18.46157);
\draw[thick,dashed,black] (5.13503,2.88925) -- (5.13503,16.85925);
\draw[thick,dashed,black] (6.34153,2.25425) -- (6.34153,16.22425);
\draw[thick,black] (4.13597,4.47887) -- (4.13597,18.44887);
\draw[thick,black] (9.72397,4.2545) -- (9.72397,18.2245);
\draw[thick,black] (9.19903,5.08) -- (9.19903,19.05);
\draw[thick,dashed,black] (11.08922,2.23943) -- (11.08922,16.20943);
\draw[thick,dashed,black] (11.40672,2.76225) -- (11.40672,16.73225);
\draw[thick,black] (4.43653,5.05037) -- (4.43653,19.02037);

\draw[flecha,black] (9.18422,10.24678) -- (9.18422,10.05628);
\draw[flecha,black] (10.40553,9.88272) -- (10.40553,9.69222);
\draw[flecha,black] (4.42172,10.53253) -- (4.42172,10.34203);
\draw[flecha,black] (3.76978,9.96103) -- (3.76978,9.77053);
\draw[flecha,black] (6.34153,7.77028) -- (6.34153,7.57978);
\draw[flecha,black] (5.10328,8.67622) -- (5.10328,8.48572);
\draw[flecha,black] (11.07228,7.64328) -- (11.07228,7.45278);
\draw[flecha,black] (11.75597,10.51772) -- (11.75597,10.32722);
\draw[flecha,black] (8.08778,2.21403) -- (8.72278,2.21403);
\draw[flecha,black] (6.56378,19.02672) -- (7.19878,19.02672);
\draw[flecha,black] (7.61153,16.20097) -- (8.24653,16.20097);
\draw[flecha,black] (4.7625,3.33375) -- (5.06518,2.99085);
\draw[flecha,black] (11.30088,3.42265) -- (11.60357,3.07975);
\draw[flecha,black] (11.33263,17.3609) -- (11.63532,17.018);
\draw[flecha,black] (4.74557,17.2339) -- (5.04825,16.891);
\draw[flecha,black] (6.84953,4.73922) -- (7.48453,4.73922);

\draw[thick,black] (11.75597,2.90618) -- (11.75597,16.87618);
\draw[thick,black] (6.34153,2.21403) -- (11.05747,2.21403);
\draw[thick,black] (4.43653,19.02672) -- (9.15247,19.02672);
\draw[thick,dashed,black] (6.3246,16.20097) -- (11.04053,16.20097);
\draw[thick,black] (3.81,4.42595) -- (5.16043,2.85327);
\draw[thick,black] (10.40553,4.45347) -- (11.75597,2.88078);
\draw[thick,black] (10.43728,18.41923) -- (11.78772,16.84655);
\draw[thick,dashed,black] (3.80153,18.34092) -- (5.15197,16.76823);
\draw[thick,dashed,black] (5.72558,2.78342) -- (5.72558,16.75342);
\draw[thick,black] (4.43653,4.75403) -- (9.15247,4.75403);
\node at (4.1,3) {$a$};
\node at (10.5,3.2) {$a$};
\node at (5.5,17.5) {$a$};
\node at (12,17.7) {$a$};
\node at (8,1.4) {$b$};
\node at (7,6) {$b$};
\node at (8,15.2) {$b$};
\node at (6,20) {$b$};
\end{tikzpicture}%
\caption{$H_{1,1}$}
\label{fig1:subfig2}
\end{subfigure}
\begin{subfigure}[b]{.3\textwidth}
\begin{tikzpicture}[scale=.31,y=-1cm]
 \tikzstyle{flecha}=[-latex new, arrow head=2.1mm]
\draw[very thick,black] (9.88757,4.92473) +(-44:0.70107) arc (-44:-191:0.70107);
\draw[very thick,black] (9.9045,18.89473) +(-44:0.70107) arc (-44:-191:0.70107);
\draw[thick,black] (5.32286,1.75024) +(99:1.12874) arc (99:23:1.12874);
\draw[thick,dashed,black] (5.32286,15.72024) +(99:1.12874) arc (99:23:1.12874);
\draw[thick,black] (3.83534,4.99116) +(6:0.60476) arc (6:-96:0.60476);
\draw[thick,black] (3.83534,18.96116) +(6:0.60476) arc (6:-96:0.60476);
\draw[thick,black] (3.83534,12.35716) +(6:0.60476) arc (6:-96:0.60476);
\draw[thick,dashed,black] (5.32286,10.98949) +(99:1.12874) arc (99:23:1.12874);
\draw[very thick,black] (9.90238,13.00617) +(-44:0.70107) arc (-44:-191:0.70107);
\draw[thick,black] (11.79552,2.14075) +(92:0.70871) arc (92:175:0.70871);
\draw[thick,dashed,black] (11.77858,16.12768) +(92:0.70871) arc (92:175:0.70871);
\draw[thick,dashed,black] (11.77858,10.39575) +(92:0.70871) arc (92:175:0.70871);
\draw[thick,black] (3.80153,4.39843) -- (3.80153,18.36843);
\draw[thick,black] (10.42247,4.49157) -- (10.42247,18.46157);
\draw[thick,dashed,black] (5.13503,2.88925) -- (5.13503,16.85925);
\draw[thick,dashed,black] (6.34153,2.25425) -- (6.34153,16.22425);
\draw[thick,black] (4.13597,4.47887) -- (4.13597,18.44887);
\draw[thick,black] (9.72397,4.2545) -- (9.72397,18.2245);
\draw[thick,black] (9.19903,5.08) -- (9.19903,19.05);
\draw[thick,dashed,black] (11.08922,2.23943) -- (11.08922,16.20943);
\draw[thick,dashed,black] (11.40672,2.76225) -- (11.40672,16.73225);
\draw[thick,black] (4.43653,5.05037) -- (4.43653,19.02037);
\draw[thick,dashed,black] (6.99558,3.55812) -- (6.99558,17.52812);
\draw[thick,dashed,black] (8.68892,3.49462) -- (8.68892,17.46462);

\draw[flecha,black] (9.18422,10.24678) -- (9.18422,10.05628);
\draw[flecha,black] (10.40553,9.88272) -- (10.40553,9.69222);
\draw[flecha,black] (4.42172,10.53253) -- (4.42172,10.34203);
\draw[flecha,black] (3.76978,9.96103) -- (3.76978,9.77053);
\draw[flecha,black] (6.34153,7.77028) -- (6.34153,7.57978);
\draw[flecha,black] (5.10328,8.67622) -- (5.10328,8.48572);
\draw[flecha,black] (11.07228,7.64328) -- (11.07228,7.45278);
\draw[flecha,black] (11.75597,10.51772) -- (11.75597,10.32722);
\draw[flecha,black] (8.08778,2.21403) -- (8.72278,2.21403);
\draw[flecha,black] (6.56378,19.02672) -- (7.19878,19.02672);
\draw[flecha,black] (7.61153,16.20097) -- (8.24653,16.20097);
\draw[flecha,black] (4.7625,3.33375) -- (5.06518,2.99085);
\draw[flecha,black] (11.30088,3.42265) -- (11.60357,3.07975);
\draw[flecha,black] (11.33263,17.3609) -- (11.63532,17.018);
\draw[flecha,black] (4.74557,17.2339) -- (5.04825,16.891);
\draw[flecha,black] (6.84953,4.73922) -- (7.48453,4.73922);

\draw[thick,black] (7.84648,3.57717) ellipse (0.84032cm and 0.508cm);
\draw[thick,dashed,black] (7.84013,17.45403) ellipse (0.84032cm and 0.508cm);
\draw[thick,dashed,black] (7.84225,10.9855) ellipse (0.84032cm and 0.508cm);
\draw[thick,black] (11.75597,2.90618) -- (11.75597,16.87618);
\draw[thick,black] (6.34153,2.21403) -- (11.05747,2.21403);
\draw[thick,black] (4.43653,19.02672) -- (9.15247,19.02672);
\draw[thick,dashed,black] (6.3246,16.20097) -- (11.04053,16.20097);
\draw[thick,black] (3.81,4.42595) -- (5.16043,2.85327);
\draw[thick,black] (10.40553,4.45347) -- (11.75597,2.88078);
\draw[thick,black] (10.43728,18.41923) -- (11.78772,16.84655);
\draw[thick,dashed,black] (3.80153,18.34092) -- (5.15197,16.76823);
\draw[thick,dashed,black] (5.72558,2.78342) -- (5.72558,16.75342);
\draw[thick,black] (4.43653,4.75403) -- (9.15247,4.75403);
\node at (4.1,3) {$a$};
\node at (10.5,3.2) {$a$};
\node at (5.5,17.5) {$a$};
\node at (12,17.7) {$a$};
\node at (8,1.4) {$b$};
\node at (7.5,5.5) {$b$};
\node at (8,15.2) {$b$};
\node at (6,20) {$b$};
\end{tikzpicture}
\caption{$H_{1,2}$}
\label{fig1:subfig3}
\end{subfigure}
\label{fig1:subfigureExample}
\caption{Products}
\end{figure}

\begin{lema}\label{lema-borde}
The boundary of a $(g,n)$-drilled body is
an oriented surface of genus $2g+n-1$ which is decomposed as a union
of two copies of $\Sigma_{g,n}$ and $n$ cylinders, called the \emph{drill holes}.
\end{lema}
\begin{proof}
It is clear that $\partial H_{g,n}$ is an oriented surface for being the boundary of an oriented $3$-manifold.
It can be decomposed as follows:
\[
 \partial H_{g,n}=\Sigma_{g,n}\times\{0,1\} \cup (\partial\Sigma_{g,n} \times I).
\]
Since $n>0$, the surface is connected. Its Euler characteristic is:
\begin{equation*}
\chi( \partial H_{g,n})=2\chi(\Sigma_{g,n})=2(2-2g-n)=2-2(2g+n-1).
\qedhere
\end{equation*}
\end{proof}

\begin{thm}\label{thm-cuerpo-asas}
A $(g,n)$-drilled body is a $(2g+n-1)$-handle body.
\end{thm}
\begin{proof}
We consider $\Sigma_{g,n}$ as the closure of the complement of $n$ pairwise disjoint disks
in a closed surface $\Sigma_g$ of genus~$g$. This surface is represented as
a $4 g$-polygon $P_{4 g}$ with the usual identifications; recall all the vertices are identified
as a point~$P$. The first disk to be removed
can be chosen with center at~$P$. The other $n-1$ disks are in the interior of $P_{4 g}$.

Hence the surface  $\Sigma_{g,n}$ can be seen as an $8g$-gone  $Q_{8g}^{n-1}$
(with $n-1$ removed disks $D_2,\dots,D_n$ in its interior),
with identifications in $4g$ of its edges. Recall that $\partial D_1$ is obtained
by gluing the non-identified edges of $Q_{8g}^{n-1}$.

We can choose $n-1$ disjoint (topological) segments~$\alpha_j$
joining $\partial D_{j}$ and $\partial D_{1}$, $j=2,\dots,n$. Note that if we
cut along
these segments and the identified edges, we obtain a topological disk.

The 3-manifold $H_{g,n}$ can be seen as a drilled prism with basis $Q_{8g}^{n-1}$,
where the vertical faces are identified as the corresponding edges on $Q_{8g}^{n-1}$.

Let us cut $Q_{8g}^{n-1}\times I$ along the $2 g$ identified faces and
the $n-1$ topological disks $\alpha_j\times I$. We obtain the product of a disk and an interval
which is a topological $3$-ball.
\end{proof}

\subsection{Float gluings}
\mbox{}

We are going to define another construction. Let $M$ be an oriented $3$-manifold with boundary
and let $\eta$ be an oriented simple closed curve in $\partial M$; then a regular neighbourhood $C$ in $\partial M$ of $\eta$ is an annulus. Consider an oriented solid torus $V$ with oriented
core $\gamma$ and let $\tilde{\gamma}$ be a longitude in $\partial V$.
Let $A$ be tubular neighbourhood of $\tilde{\gamma}$ in $\partial V$; note that $A$ is an annulus.
Let $\psi:C\to A$ be an orientation-reversing homeomorphism.

\begin{prop}\label{prop-MV}
The manifold $M\cup_\psi V$ is homeomorphic to $M$.
\end{prop}

\begin{proof}
The solid torus $V$ can be retracted to $A$ and this rectraction induces an isotopy between
$M\cup_\psi V$ and $M$.
\end{proof}

\begin{obs}
 Note that in the previous construction there are two possible choices for the gluing morphism $\psi$. One of them identifies $\gamma$ with $\eta$ and the other one, $\gamma$ with $\eta^{-1}$. Moreover, the gluings of the boundary components of $C$ and $A$ are interchanged.
\end{obs}

\begin{dfn}
 The above operation is called a \emph{float gluing} of $M$ along $C$.
\end{dfn}



\begin{dfn}Given a handle-body~$M$ of genus~$g$ we say that a simple closed curve
$\gamma\subset\partial M$ is a \emph{float curve}
if there is a cutting system of curves in $\partial M$ such that $\gamma$ intersects this system in only one point, and this
intersection is transverse.
\end{dfn}

\begin{ejm}
 Let us consider a solid torus $V_1$ and let $\gamma$ be a
simple closed curve in $\partial V_1$ isotopic to the core of $V_1$.
Let $V_{g-1}$ be a handle-body of genus~$g-1$.
Let $V_g$ be the handle-body obtained by gluing two disks in the boundaries of $V_1$ and $V_{g-1}$;
for further use, we will refer to this operation as the \emph{handle sum} of $V_1$ and $V_{g-1}$;
we can assume that $\gamma$ is disjoint with the disk in $\partial V_1$
used for the handle sum. Then $\gamma\subset V_g$ is a \emph{float curve} of $V_g$
since it cuts only the meridian of $V_1$.
\end{ejm}

\begin{dfn}
The pair $(V_g,\gamma)$ is called a \emph{standard float-curve system} of genus~$g$.
\end{dfn}

\begin{lema}
Let $M$ be a handle-body of genus~$g$ and let $\gamma\subset M$
be a float curve. Then, the pair $(M,\gamma)$ is homeomorphic to
a standard float-curve system of genus~$g$.
\end{lema}

\begin{proof}
A handle-body can be seen as a closed ball $\mathbb{B}^3$ with $2 g$ pairwise disjoint disks in
the boundary glued in pairs. In this model a curve $\gamma$ is a segment joining a pair of glued disks
and avoiding the other ones. Two such models can be connected by a homeomorphism.
\end{proof}


\begin{prop}
\label{float-gluing-handlebodies}
Let $M_1,M_2$ be two handle-bodies of genus $g_1,g_2\geq 1$. Fix float curves $\gamma_1,\gamma_2$ in each one and
consider regular neighbourhoods $A_1,A_2$ of these curves in $\partial M_1,\partial M_2$, respectively.
Let $\psi:A_1\to A_2$ be an orientation-reversing homeomorphism. Then,
$M_1\cup_\psi M_2$ is a handle-body of genus $g_1+g_2-1$.
\end{prop}

\begin{obs}
If $M_2$ is of genus~$1$ the above operation is a particular case of float gluing since we only need
the curve in the solid torus to be a float curve. In fact, the above proposition remains true if we only ask
$\gamma_2$ to be a float curve, but we do not use this more general fact.
\end{obs}

\begin{proof}
Note that $M_2$ can be constructed as a handle sum of a solid torus $V_1$ and a handle body of genus~$g_2-1$.
This operation can be performed in order to have $A_2\subset V_1$ and $\gamma_2$ homotopic to the core of $V_1$.

Then, the gluing of $M_1$ and $M_2$ can be performed as a float gluing followed by a handle sum.
\end{proof}

\begin{obs}\label{obs:handle-1}
In fact, we can be more specific with the handlebody structure of $M:=M_1\cup_\psi M_2$. Consider a system
of cutting curves $\alpha_1,\dots,\alpha_{g_1}$ for $M_1$ and $\beta_1,\dots,\beta_{g_2}$ for $M_2$. We first assume
that only $\alpha_1$ (resp. $\beta_1$) intersects $\gamma_1$ (resp. $\gamma_2$), at only one point and transversally  (which is
possible since $\gamma_1$ and $\gamma_2$ are float curves). Let $\check{\alpha}_1$ be the piece of $\alpha_1$ outside
the small neighbourhood of $\gamma_1$ used for the gluing; define $\check{\beta}_1$ accordingly. We can isotopically
move  $\beta_1$ such that $\delta:=\check{\alpha_1}\cdot\check{\beta}_1$ is a cycle in the boundary of~$M$.
Then, $\delta,\alpha_2,\dots,\alpha_{g_1},\beta_2,\dots,\beta_{g_2}$ is a cutting system for~$M$.
\end{obs}

\begin{obs}\label{obs:handle-vv}
This process can be generalized when $\alpha_i$, $1\leq i\leq h_1$, interesects $\gamma_1$ transversally at one point
and $\alpha_i\cap\gamma_1=\emptyset$ if $h_1<i\leq g_1$ and a similar fact arises for the other system for some $h_2$.
In this case we can choose suitable curves $\alpha'_i$, $2\leq i\leq h_2$
(resp.  $\beta'_i$, $2\leq i\leq h_1$) parallel to $\alpha_1$ (resp. $\beta_1$) such that
$\tilde{\beta}_i:=\check{\alpha}'_i\cdot\check{\beta}_i$, $2\leq i\leq h_2$
(resp. $\tilde{\alpha}_i:=\check{\alpha}_i\cdot\check{\beta}'_i$, $2\leq i\leq h_1$) are cycles.
Then
\begin{equation}
\label{eq:system-vv}
\delta,\tilde{\alpha}_2,\dots,\tilde{\alpha}_{h_1},\alpha_{h_1+1},\dots,\alpha_{g_1},
\tilde{\beta}_2,\dots,\tilde{\beta}_{h_2},\beta_{h_2+1},\dots,\beta_{g_2}
\end{equation}
is a cutting system of~$M$. We can prove it using handle-slide moves
of $\alpha_i$ (resp. $\beta_i$), $1< i\leq h_1$ (resp. $h_2$),
along $\alpha_1$ (resp. $\beta_1$) in order to pass to the situation of Remark~\ref{obs:handle-1}; after the construction
of the cutting system of~$M$ we perform \emph{inverse} handle-slide moves along~$\delta$ and we recover
the system~\eqref{eq:system-vv}.
\end{obs}

The same idea can be used if we identify two different annuli in a single
handle body.

\begin{prop}
\label{float-gluing-una}
Let $M$ be a handle body of genus $g\geq 2$. Fix a cutting system of curves and
two disjoint float curves $\gamma_1,\gamma_2$ such that they intersect different curves
of the cutting system $\alpha_1,\alpha_2$. Consider regular neighbourhoods
$A_1,A_2$ of $\gamma_1$
and $\gamma_2$ respectively. Let $\psi:A_1\to A_2$ be an orientation-reversing
homeomorphism. Then the quotient $M_\psi$ of $M$ by $\psi$ is a handle body of genus $g$.
\end{prop}
\begin{proof}
 Up to homeomorphism, we may assume that the float systems are standard
ones. In that case, $M$ is a handle sum of solid tori, being $\gamma_1$ and
$\gamma_2$ the longitudes of two of them. The identification then produces a
float gluing bewteen these two solid tori, so we obtain again a handle sum
of solid tori, but introducing a loop in the chain of handle sums. This loop
introduces a new handle, that compensates the one lost by the identification.
\end{proof}

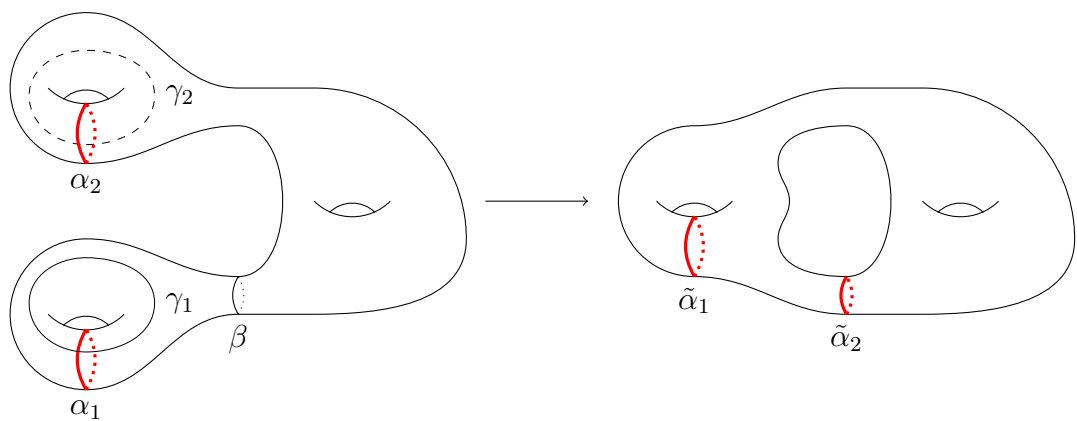
\begin{figure}[ht]

\begin{tikzpicture}[scale=1]

\draw (0,0)
to [out=0, in=180] (2,1)
to [out =0, in=180] (3,1)
to [out=0, in=270] (5,2)
to [out=90, in=0] (3,4)
to [out=180, in=0] (2,4)
to [out=180, in=0] (0,5)
to [out=180, in=90] (-1,4)
to [out=270, in=180] (0,3)
to [out=0, in=180] (2,3.5)
to [out=0, in=0] (2, 1.5)
to [out=180, in=0] (0,2)
to [out=180, in=90] (-1,1)
to [out=270, in=180] cycle;

\def\hole{
    \draw (-0.5, 1)
    to [out=315, in=225] (0.5, 1);

    \draw (-0.3, 0.85)
    to [out=45, in=135] (0.3,0.85);
    }

\hole

\begin{scope}[yshift=3cm]
\hole
\end{scope}
\begin{scope}[xshift=3.5cm, yshift=1.5cm]
\hole
\end{scope}

\draw[red,line width=1.2](0,0) to
[out=120, in=240] (0,0.8);
\draw[red, line width=1.2,dotted](0,0) to
[out=60, in=300]  (0,0.8);
\node[below] at (0,0) {$\alpha_1$};

\draw[line width=.25] (0,0.5) to[out=0, in=-90](.9,1.15) node[right] {$\gamma_1$} to[out=90,in=0](0,1.75)  to[out=180,in=90](-.75,1.15)to[out=-90,in=180] cycle;

\begin{scope}[yshift=2.75cm]
\draw[line width=.25,dashed] (0,0.5) to[out=0, in=-90](.9,1.15) node[right] {$\gamma_2$} to[out=90,in=0](0,1.75)  to[out=180,in=90](-.75,1.15)to[out=-90,in=180] cycle;
\end{scope}

\draw[red, line width=1.2](0,3) to
[out=120, in=240] (0,3.8);
\draw[red, line width=1.2, dotted](0,3) to
[out=60, in=300]  (0,3.8);
\node[below] at (0,3) {$\alpha_2$};
\draw[] (2,1) to
[out=120, in=240] (2,1.5);
\draw[dotted](2,1) to
[out=60, in=300]  (2,1.5);
\node[below] at (2,1) {$\beta$};

\draw[->] (5.25,2.5)--(6.6,2.5) ;

\begin{scope}[xshift=8cm]
\draw (0,1.5)
to [out=0, in=180] (2,1)
to [out =0, in=180] (3,1)
to [out=0, in=270] (5,2)
to [out=90, in=0] (3,4)
to [out=180, in=0] (2,4)
to [out=180, in=0] (0,3.5)
to [out=180, in=90] (-1,2.5)
to [out=270, in=180] cycle;

\draw(1.25,2.5)
to [out=90, in=270] (1.1,3)
to [out=90, in=180] (2,3.5)
to [out=0, in=0] (2, 1.5)
to [out=180, in=270] (1.1, 2)
to [out=90, in=270]
cycle;
\begin{scope}[yshift=1.5cm]
\hole
\end{scope}
\begin{scope}[xshift=3.5cm, yshift=1.5cm]
\hole
\end{scope}

\draw[red, line width=1.2](0,1.5) to
[out=120, in=240] (0,2.3);
\draw[red, line width=1.2, dotted](0,1.5) to
[out=60, in=300]  (0,2.3);
\node[below] at (0,1.5) {$\tilde{\alpha}_1$};
\draw[red, line width=1.2](2,1) to
[out=120, in=240] (2,1.5);
\draw[red, line width=1.2, dotted](2,1) to
[out=60, in=300]  (2,1.5);
\node[below] at (2,1) {$\tilde{\alpha}_2$};
\end{scope}

\end{tikzpicture}

\caption{Gluing disjoint float curves in a handle body\label{fig:flotador-ciclo}}
\end{figure}

\begin{obs}\label{obs:cutting-ciclo}
In Figure~\ref{fig:flotador-ciclo} it can be seen how a new cutting curve is
obtained by joining the two identified ones, and another one appears for the
handle corresponding to the cycle. We are going
to check that the latter corresponds to the commutator of
$\alpha_1$ and $\gamma_1$.

Let $F=\partial{M}$ and consider regular neighbourhoods $N(\gamma_1)$ and $N(\gamma_2)$
of $\gamma_1$ and $\gamma_2$ bounded by four curves $\gamma_i^\pm$. The surface
$F_\psi:=\partial{M_\psi}$ is obtained as follows. Consider
the quotient of $\overline{F\setminus(N(\gamma_1)\cup N(\gamma_2))}$ obtained by gluing
$\gamma_1^+$ with $\gamma_2^-$ and $\gamma_1^-$ with $\gamma_2^+$
in order to obtain an oriented 3-manifold.

Note that $F=\partial{M}$ and $F_\psi=\partial{M_\psi}$ are equal outside regular neighbourhoods
of $\gamma_1$ and $\gamma_2$. A cutting system for $M_\psi$ can be constructed as follows.
We keep the curves $\alpha_3,\dots,\alpha_g$ of the cutting system of $M$ and we add two new
curves $\tilde{\alpha}_1$ and~$\bar{\alpha}_2$. The curve $\bar{\alpha}_1$ is the connected sum
$\alpha_1\#\alpha_2$ obtained as the union of two pieces
$\check{\alpha}_1,\check{\alpha}_2$
as $\delta$ in Remark~\ref{obs:handle-1}.
The curve $\bar{\alpha}_2$ is the image by the gluing of the curve 
$\beta$ in $M$,which is the commutator
of $\alpha_1$ and $\gamma_1$ (see Figure~\ref{fig:flotador-ciclo}). Note that
the commutator of $\alpha_2$ and $\gamma_2$ could also be chosen instead
of~$\beta$.
\end{obs}

\section{\texorpdfstring{Heegaard splittings of $\bs^1$-bundles over  surfaces with unimodular Euler number}%
{Heegaard decomposition of S1-bundles over  surfaces with unimodular Euler number}}
\label{sec-euler1}
Let $\pi:M\to S$ be an oriented $\bs^1$-bundle 
over a closed oriented  surface $S$ of genus~$g$, 
with Euler number $e\in\bz\equiv H^2(S;\bz)$. 
Consider a small closed disk $D\subseteq S$ 
and consider the surface with boundary $\check{S}:=\overline{S\setminus D}$. 
Since $H^2(\check{S};\bz)$ is trivial, there exists a section $s_1:\check{S}\to M$ of $\pi$. 
We take another \emph{parallel} section $s_2$. 
These two sections divide $\check{M}=\pi^{-1}(\check{S})$ in two pieces $M_1$ and $M_2$;
which are oriented compact $3$-manifolds with boundary,
and satisfy that $M_1\cap M_2=\partial M_1\cap \partial M_2=\gs\coprod\gn$, where $\gs:=s_1(\check{S})$ and $\gn:=s_2(\check{S})$.
We will now show how to use these two pieces to construct a Heegaard splitting of $M$
when the Euler number of the fibration is $e=\pm 1$ ( i.e., the plumbing manifold
associated with a graph  with only one vertex~$v$, $g_v=g$, $e_v=\pm 1$).

\begin{cvt}
Once the two sections $s_1,s_2$ have been chosen, we choose $M_1$ and $M_2$
in such a way that the orientations on $\gn$ induced by $M_1$  and $s_2$ coincide.
This means that a positive half-fiber inside $M_1$ goes from $\gs$ to $\gn$.
\end{cvt}

The boundary of $M_1$ is obtained by gluing $\gs$ and $\gn$ with an
annulus $C$ which fibers over $\partial D=\partial\check S$ (whose fibers
are positive half-fibers inside $M_1$ homeomorphic to $[0,1]$). In the same
way $\partial M_2=\gs\cup C'\cup\gn$, where $C'$ is the other annulus in $M_2$.
Note that $C\cup C'$ is the torus $\pi^{-1}(\partial(D))=\partial\pi^{-1}(D)$ ($C$ and $C'$
have common boundaries).


\begin{prop}\label{prop-cuerpo-asas-2g}
The $3$-manifolds $M_1,M_2$ are $2g$-handle bodies.
\end{prop}
\begin{proof}
This manifold is, by construction, the drilled body $H_{g,1}$,
see Theorem~\ref{thm-cuerpo-asas}.
Since $M_2$ is homeomorphic to $M_1$, it is also a $2g$-handle body.
%
%
%
%
\end{proof}

%
%
%

\begin{thm}\label{thm-pegado-toro1}
Let $\tilde{M}_2:=M_2\cup\pi^{-1}(D)$. The manifold $\tilde{M_2}$ is homeomorphic to $M_2$ and, hence, it is a $2g$-handle body.
\end{thm}

\begin{proof}
Note that $C'$ is the annulus along which $M_2$ and $\pi^{-1}(D)$ are glued. Let $K$ be the core
of this annulus. Since $e=\pm 1$, $K$ is homologous to the core of $\pi^{-1}(D)$ and the statement follows
from Proposition~\ref{prop-MV}. 
\end{proof}

\begin{cor}
The submanifolds $M_1$ and $\tilde{M_2}$ form a Heegaard splitting of $M$ of genus~$2g$.
\end{cor}

\section{\texorpdfstring{Heegaard diagram of a unimodular 
$\bs^1$-bundle.}{Heegaard diagram of a unimodular S1-bundle.}}
\label{sec:diagramae1}
\mbox{}

Let us denote $\Sigma_1:=\partial M_1=\partial\tilde{M}_2$ (oriented as
boundary of~$M_1$),
which is the gluing of $\gs$,
$\gn$ and the cylinder~$C\cong\partial D\times I$. 
Note that $\gn$ inherits
the orientation of $S$ while $\gs$ inherits the opposite one.

\begin{figure}[ht]
\centering
\definecolor{cff0000}{RGB}{255,0,0}
\definecolor{c0000ff}{RGB}{0,0,255}
\begin{subfigure}[b]{.55\textwidth}
\begin{tikzpicture}[y=0.80pt, x=0.8pt,yscale=-1,xscale=1.2, inner sep=0pt, outer 
sep=0pt,
vertice/.style={draw,circle,fill,minimum size=0.15cm,inner sep=0}]
\tikzset{flecha1/.style={decoration={
  markings,
  mark=at position #1 with  {\arrowreversed{angle 90 
new}}},postaction={decorate}}}
\tikzset{flecha2/.style={decoration={
  markings,
  mark=at position #1 with  {\arrow{angle 90 new}}},postaction={decorate}}}
\path[draw=black,line join=miter,line cap=butt,miter limit=4.00,line 
width=0.960pt] (147.6849,800.9034) .. controls (165.4973,790.4617) and 
(156.8982,738.2531) .. (185.7665,740.0957) .. controls (214.6348,741.9384) and 
(210.3353,792.9186) .. (221.3912,797.8323);
\path[draw=black,line join=miter,line cap=butt,miter limit=4.00,line 
width=0.960pt] (185.1523,799.0607) .. controls (168.5684,793.5328) and 
(176.5532,760.3649) .. (186.9950,759.7507) .. controls (197.4367,759.1365) and 
(208.4926,788.0048) .. (194.9798,797.2181);
\path[draw=black,line join=miter,line cap=butt,miter limit=4.00,line 
width=0.960pt] (159.9693,776.9489) .. controls (120.0451,783.7053) and 
(91.1768,824.8579) .. (153.2129,862.9395);
\path[draw=black,line join=miter,line cap=butt,miter limit=4.00,line 
width=0.960pt] (176.5532,775.7204) .. controls (183.3096,775.7204) and 
(193.1372,775.1062) .. (200.5078,776.9489);
\path[draw=black,line join=miter,line cap=butt,miter limit=4.00,line 
width=0.960pt] (212.1780,776.3346) .. controls (271.1430,792.9186) and 
(268.6861,833.4570) .. (216.4775,864.1680);
\path[draw=black,line join=miter,line cap=butt,miter limit=4.00,line 
width=0.960pt] (147.6849,976.7127) .. controls (165.4973,987.1544) and 
(156.8982,1039.3630) .. (185.7665,1037.5203) .. controls (214.6348,1035.6777) 
and (210.3353,984.6976) .. (221.3912,979.7838);
\path[draw=black,line join=miter,line cap=butt,miter limit=4.00,line 
width=0.960pt] (185.1523,978.5553) .. controls (168.5684,984.0833) and 
(176.5532,1017.2511) .. (186.9950,1017.8653) .. controls (197.4367,1018.4795) 
and (208.4926,989.6112) .. (194.9798,980.3980);
\path[draw=black,line join=miter,line cap=butt,miter limit=4.00,line 
width=0.960pt] (159.9693,1000.6672) .. controls (120.0451,993.9108) and 
(91.1768,952.7581) .. (153.2129,914.6766);
\path[draw=black,line join=miter,line cap=butt,miter limit=4.00,line 
width=0.960pt] (176.5532,1001.8956) .. controls (183.3096,1001.8956) and 
(193.1372,1002.5098) .. (200.5078,1000.6672);
\path[draw=black,line join=miter,line cap=butt,miter limit=4.00,line 
width=0.960pt] (212.1780,1001.2814) .. controls (271.1430,984.6975) and 
(268.6861,944.1591) .. (216.4775,913.4481);
\path[draw=black,line join=miter,line cap=butt,miter limit=4.00,line 
width=0.960pt] (144.6138,846.9698) .. controls (163.6546,862.9395) and 
(155.0556,912.0771) .. (150.1418,924.9757);
\path[draw=black,line join=miter,line cap=butt,miter limit=4.00,line 
width=0.960pt] (224.7967,850.0409) .. controls (205.7559,866.0106) and 
(214.3550,915.1482) .. (219.2687,928.0468);
\path[flecha1=.15,flecha1=.5,flecha1=.75,draw=cff0000,line join=miter,line 
cap=butt,miter limit=4.00,line width=0.432pt] (156.2840,784.3195) .. controls 
(134.7863,797.2181) and (142.7712,810.1167) .. (146.4565,824.8579) .. controls 
(150.1418,839.5992) and (173.2161,853.7381) .. (171.6395,894.8789) .. controls 
(169.7320,944.6537) and (145.8423,941.9110) .. (143.3854,964.2857) .. controls 
(141.1730,984.4338) and (153.2129,990.0829) .. (157.5124,993.1540);
\node[vertice] (P1) at (168,865) {};
\node[above=7pt] at (P1.north) {$p_1^1$};

\path[flecha2=.5,draw=cff0000,line join=miter,line cap=butt,miter 
limit=4.00,line width=0.432pt] (174.7106,781.2484) .. controls 
(189.4518,781.8626) and (186.8746,796.8672) .. (185.7665,816.2589) .. controls 
(184.5381,837.7566) and (190.6803,836.5281) .. (191.2945,894.2647) .. controls 
(191.9087,952.0013) and (188.8376,967.9710) .. (188.2234,977.7985) .. controls 
(187.6092,987.6260) and (183.9239,990.6971) .. (175.9390,994.9966);

\node[vertice] (P2) at (190.5,866) {};
\node[above=8pt] at (P2.east) {$p_1^2$};
\node[vertice] (P3) at (168,920) {};
\node[above=7pt] at (P3.west) {$q_1^1$};
\node[vertice] (P4) at (191,920) {};
\node[below right=2pt] at (P4.east) {$q_1^2$};

\path[flecha1=.5,draw=cff0000,line join=miter,line cap=butt,miter 
limit=4.00,line width=0.432pt] (177.1675,1021.4080) .. controls 
(191.2945,1033.0782) and (206.0358,1010.9663) .. (207.8784,994.9967) .. controls 
(209.7211,979.0270) and (206.0358,958.1435) .. (204.1931,889.3510) .. controls 
(202.3504,820.5584) and (210.3353,797.8323) .. (209.7211,791.6901) .. controls 
(209.1069,785.5479) and (205.4215,750.5374) .. (186.3807,749.9232) .. controls 
(167.3400,749.3090) and (164.8831,780.6342) .. (164.2689,795.9897) .. controls 
(163.6546,811.3451) and (181.4670,839.5992) .. (180.2386,890.5794) .. controls 
(179.0101,941.5596) and (163.0404,975.3417) .. (166.1115,990.6972) .. controls 
(169.1826,1006.0526) and (177.1675,1021.4080) .. (177.1675,1021.4080) -- cycle;
\node at (155,761.11987) {$b_1$};
\node at (140,823.56549) {$\nu_1$};
\node at (164,880) {$\lambda_1^1$};
\node at (140,950) {$\nu'_1$};
\node at (198,875) {$\lambda_1^2$};
\node at (229,875) {$a_1=$};
\node at (250,885) {$\nu_1\cdot\lambda_1^2\cdot\nu'_1\cdot\lambda_1^1$};
\draw[thick,flecha1=.5] (153,916) arc (180:0:31.5 and 4);
\draw[thick,dashed] (153,916) arc (-180:0:31.5 and 4);
\node at (150,930) {$\gamma'_1$};
\node at (250,930) {$\gs$};
\node at (250,845) {$\gn$};
\node at (145,895) {$C$};
\draw[thick,flecha2=.5] (153,862) arc (180:0:31.5 and 4);
\draw[thick,dashed] (153,862) arc (-180:0:31.5 and 4);
\node at (150,873) {$\gamma_1$};

\end{tikzpicture}\caption{Cutting curves for $M_1$.}
\label{fig:cutting1}
\end{subfigure}
\hfil
\begin{subfigure}[b]{.43\textwidth}
\begin{tikzpicture}[y=0.80pt, x=0.8pt,yscale=-1,xscale=1.2, inner sep=0pt, outer 
sep=0pt,vertice/.style={draw,circle,fill,minimum size=0.15cm,inner sep=0}]
\tikzset{flecha1/.style={decoration={
  markings,
  mark=at position #1 with  {\arrowreversed{angle 90 
new}}},postaction={decorate}}}
\tikzset{flecha2/.style={decoration={
  markings,
  mark=at position #1 with  {\arrow{angle 90 new}}},postaction={decorate}}}
\path[draw=black,line join=miter,line cap=butt,miter limit=4.00,line 
width=0.960pt] (147.6849,800.9034) .. controls (165.4973,790.4617) and 
(156.8982,738.2531) .. (185.7665,740.0957) .. controls (214.6348,741.9384) and 
(210.3353,792.9186) .. (221.3912,797.8323);
\path[draw=black,line join=miter,line cap=butt,miter limit=4.00,line 
width=0.960pt] (185.1523,799.0607) .. controls (168.5684,793.5328) and 
(176.5532,760.3649) .. (186.9950,759.7507) .. controls (197.4367,759.1365) and 
(208.4926,788.0048) .. (194.9798,797.2181);
\path[draw=black,line join=miter,line cap=butt,miter limit=4.00,line 
width=0.960pt] (159.9693,776.9489) .. controls (120.0451,783.7053) and 
(91.1768,824.8579) .. (153.2129,862.9395);
\path[draw=black,line join=miter,line cap=butt,miter limit=4.00,line 
width=0.960pt] (176.5532,775.7204) .. controls (183.3096,775.7204) and 
(193.1372,775.1062) .. (200.5078,776.9489);
\path[draw=black,line join=miter,line cap=butt,miter limit=4.00,line 
width=0.960pt] (212.1780,776.3346) .. controls (271.1430,792.9186) and 
(268.6861,833.4570) .. (216.4775,864.1680);
\path[draw=black,line join=miter,line cap=butt,miter limit=4.00,line 
width=0.960pt] (147.6849,976.7127) .. controls (165.4973,987.1544) and 
(156.8982,1039.3630) .. (185.7665,1037.5203) .. controls (214.6348,1035.6777) 
and (210.3353,984.6976) .. (221.3912,979.7838);
\path[draw=black,line join=miter,line cap=butt,miter limit=4.00,line 
width=0.960pt] (185.1523,978.5553) .. controls (168.5684,984.0833) and 
(176.5532,1017.2511) .. (186.9950,1017.8653) .. controls (197.4367,1018.4795) 
and (208.4926,989.6112) .. (194.9798,980.3980);
\path[draw=black,line join=miter,line cap=butt,miter limit=4.00,line 
width=0.960pt] (159.9693,1000.6672) .. controls (120.0451,993.9108) and 
(91.1768,952.7581) .. (153.2129,914.6766);
\path[draw=black,line join=miter,line cap=butt,miter limit=4.00,line 
width=0.960pt] (176.5532,1001.8956) .. controls (183.3096,1001.8956) and 
(193.1372,1002.5098) .. (200.5078,1000.6672);
\path[draw=black,line join=miter,line cap=butt,miter limit=4.00,line 
width=0.960pt] (212.1780,1001.2814) .. controls (271.1430,984.6975) and 
(268.6861,944.1591) .. (216.4775,913.4481);
\path[draw=black,line join=miter,line cap=butt,miter limit=4.00,line 
width=0.960pt] (144.6138,846.9698) .. controls (163.6546,862.9395) and 
(155.0556,912.0771) .. (150.1418,924.9757);
\path[draw=black,line join=miter,line cap=butt,miter limit=4.00,line 
width=0.960pt] (224.7967,850.0409) .. controls (205.7559,866.0106) and 
(214.3550,915.1482) .. (219.2687,928.0468);

\path[draw=c0000ff,line join=miter,line cap=butt,miter limit=4.00,line 
width=0.480pt] (158.7409,779.4057) .. controls (139.0859,788.6190) and 
(134.5225,795.3641) .. (134.7863,808.8883) .. controls (135.1216,826.0801) and 
(163.6546,856.7974) .. (159.9693,864.7822) .. controls (158.3212,868.3531) and 
(155.0556,875.8382) .. (155.0556,875.8382);
\path[draw=c0000ff,dash pattern=on 0.48pt off 0.96pt,line join=miter,line 
cap=butt,miter limit=4.00,line width=0.480pt] (155.0556,873.9955) .. controls 
(155.0556,880.1377) and (212.7922,884.4372) .. (213.4064,893.6505);
\path[draw=c0000ff,line join=miter,line cap=butt,miter limit=4.00,line 
width=0.480pt] (213.4064,893.0363) .. controls (214.0206,905.3207) and 
(159.9693,902.8638) .. (157.5124,913.3055) .. controls (155.0556,923.7472) and 
(131.2325,962.9460) .. (131.7152,970.4280) .. controls (132.9437,989.4687) and 
(148.9134,996.2251) .. (159.3551,998.0678);
\path[draw=c0000ff,line join=miter,line cap=butt,miter limit=4.00,line 
width=0.480pt] (175.9390,998.6820) .. controls (184.5381,996.8394) and 
(192.5229,988.0248) .. (192.5229,980.2554) .. controls (192.5229,969.1994) and 
(191.2945,922.5188) .. (198.0509,916.3766) .. controls (204.8073,910.2344) and 
(213.4064,909.0060) .. (214.0206,904.0922);
\path[draw=c0000ff,dash pattern=on 0.48pt off 0.96pt,line join=miter,line 
cap=butt,miter limit=4.00,line width=0.480pt] (214.0206,904.7065) .. controls 
(214.0206,897.3358) and (155.6698,895.4932) .. (156.2840,890.5794);
\path[draw=c0000ff,line join=miter,line cap=butt,miter limit=4.00,line 
width=0.480pt] (155.6698,891.1936) .. controls (156.8982,884.4372) and 
(195.5940,882.5946) .. (196.8225,873.9955) .. controls (198.0509,865.3964) and 
(190.0661,802.7461) .. (190.6803,797.2181) .. controls (191.2945,791.6901) and 
(193.1372,778.7915) .. (175.9390,778.7915);
\path[draw=c0000ff,line join=miter,line cap=butt,miter limit=4.00,line 
width=0.480pt] (185.1523,1030.6213) .. controls (193.7514,1030.6213) and 
(206.8157,1017.5654) .. (210.3353,997.4536) .. controls (214.6348,972.8848) and 
(206.0358,932.9605) .. (209.7211,922.5188) .. controls (210.9131,919.1415) and 
(214.0206,918.8335) .. (214.6348,912.0771);
\path[draw=c0000ff,dash pattern=on 0.48pt off 0.96pt,line join=miter,line 
cap=butt,miter limit=4.00,line width=0.480pt] (215.2491,912.0771) .. controls 
(216.4775,907.1633) and (154.4413,907.1633) .. (155.0556,899.7927);
\path[draw=c0000ff,line join=miter,line cap=butt,miter limit=4.00,line 
width=0.480pt] (155.0556,900.4069) .. controls (156.2840,894.8789) and 
(206.6500,888.1225) .. (207.2642,883.8230) .. controls (207.8784,879.5235) and 
(214.0206,795.9897) .. (212.7922,791.0759) .. controls (211.5637,786.1621) and 
(207.8549,742.8898) .. (184.5381,745.0095) .. controls (164.2689,746.8521) and 
(160.6568,786.1559) .. (159.9693,795.9896) .. controls (158.8951,811.3567) and 
(179.6243,854.9547) .. (174.7106,869.0817) .. controls (172.2724,876.0915) and 
(155.6698,879.5235) .. (155.6698,885.0514);
\path[draw=c0000ff,dash pattern=on 0.48pt off 0.96pt,line join=miter,line 
cap=butt,miter limit=4.00,line width=0.480pt] (155.6698,884.4372) .. controls 
(155.0556,889.9652) and (215.2491,896.7216) .. (214.0206,899.7927);
\path[draw=c0000ff,line join=miter,line cap=butt,miter limit=4.00,line 
width=0.480pt] (214.0206,899.7927) .. controls (214.6348,905.3207) and 
(173.4821,909.6202) .. (172.2537,916.3766) .. controls (171.0253,923.1330) and 
(160.5835,976.5701) .. (160.5835,983.9407) .. controls (160.5835,991.3113) and 
(168.5684,1030.0071) .. (185.1523,1030.6213);

\node at (155,761.11987) {$b'_1$};
\node at (147,823.56549) {$a'_1$};

\end{tikzpicture}
\caption{Cutting curves for $\tilde{M}_2$, $e=1$.}
\label{fig:dibujo-cilindro-unfibrado-azul}
\end{subfigure}
\caption{} 
\end{figure}

In this situation, the system of cutting curves for $M_1$ is formed by two
families of curves:
\begin{itemize}
 \item Curves $a_1,\ldots, a_g$ coming from half of the identified faces
in the prysm, see Figure~\ref{fig1:subfig2}. 
They are decomposed into four pieces as follows, see 
Figure~\ref{fig:cutting1}.
%
Consider points $p_i^1,p_i^2$ in $C\cap\gn$ and points $q_i^1,q_i^2$ in $C\cap\gs$ such that
there are half-fibers $\lambda_i^1$ (from $q_i^1$ to $p_i^1$) $\lambda_i^2$ 
(from $p_i^2$ to $q_i^2$).
Pick up a path $\nu_i$ in $\gs$ from $p_i^1$ to $p_i^2$ which turns around the 
$i$'th handle like its meridian.
We construct a path $\nu'_i$ in $\gn$ in a similar way with reversed 
orientation. 
Then,
$a_i:=\nu_i\cdot\lambda_i^2\cdot\nu'_i\cdot\lambda_i^1$. It is possible to 
choose these cycles
to be pairwise disjoint.

\item Curves $b_1,\ldots, b_g$ coming from the other half of the identified
faces. They are constructed in the same way as the $a_i$, but instead of taking
$\nu_i$ and ${\nu'}_i$, we take paths that turn around the handles like their
longitudes. These paths are chosen in such a way that they don't intersect
each other and they are also disjoint to the paths~$a_i$'s.
\end{itemize}

The prysm of Figure~\ref{fig1:subfig2} shows how to prove
that this is a system of cutting curves.

In order to obtain a system of cutting curves for $\tilde{M}_2$ we recall its 
construction. We
start with $M_2$ (homeomorphic copy to $M_1$) which is constructed in the 
same way as $M_1$ but using the other cylinder~$C'$. Recall that 
the union of the two cylinders $C$ and $C'$ along their common boundary yields
the torus $\mathbb{T}:=\pi^{-1}(\partial D)$, the boundary of the solid torus
$\pi^{-1}(D)$.
So the construction of the system of cutting curves for $M_2$ will
mimic the one for $M_1$ replacing the cylinder $C$ by~$C'$.

Since $\tilde{M}_2=M_2\cup\pi^{-1}(D)$, let us consider the situation
 at $\pi^{-1}(D)$.
In order to fix the orientations, we assume that $e=1$, leaving the case $e=-1$
for later. The solid torus $\pi^{-1}(D)$ is represented as a cylinder whose bottom and top are 
glued by a vertical translation in Figure~\ref{dibujo-cilindro-unfibrado}. Note
 that $\pi^{-1}(D)$ is the solid torus used in the float gluing in order to
obtain $\tilde{M}_2$ from $M_2$.

In the torus~$\mathbb{T}$, we fix the product structure with
oriented section $\mu_1$ (the boundary of a 
disk in the solid torus) and with 
oriented fibre $\phi_1$. 
Let us fix one cutting curve ($a_i$ or $b_i$) of $M_1$; it
intersects the cylinder~$C$ in two half-fibers.
Let
$\lambda_1$ be the one from $\gs$ to $\gn$; 
let ${\lambda'}_1$ be the other half of the fiber in $C'$ (which is part of a cutting curve in $M_2$)
but with opposite orientation, in order to go again from $\gs$ to $\gn$; i.e.,
$\lambda_1\cdot{\lambda'_1}^{-1}$ 
is homologous to~$\phi_1$ in $\mathbb{T}$. 

The cylinders $C$ and $C'$ have as common boundaries
two cycles $\gamma_1\subset\gn$ 
and ${\gamma'}_1\subset\gs$, oriented as boundaries of these surfaces; in Figure~\ref{dibujo-cilindro-unfibrado},
the front part of $C$ is coloured. The homology class of $\gamma_1$ in $\mathbb{T}$ is 
(with multiplicative notation) $\mu_1^{-1}\cdot\phi_1^{-e}$ 
(recall $e=1$ in~Figure~\ref{dibujo-cilindro-unfibrado}), since the definition of Euler number implies
that $\gamma_1\cdot\mu_1\cdot\phi_1^{e}$ is trivial.

The cycle $(\gamma'_1)^{-1}\cdot(\lambda_1\cdot{\lambda'_1}^{-1})^e\sim \gamma_1\cdot\phi_1^e\sim \mu_1^{-1}$ 
bounds a disk in $\pi^{-1}(D)$. The union of this disk with the cutting disk of $M_2$
containing ${\lambda'_1}^{-1}$ in its boundary provides a new disk where ${\lambda'_1}^{-1}$
is no more in its boundary. 
If we repeat this process
with the other half-fiber in the cutting curve, we obtain the corresponding 
cutting curve in~$\tilde{M}_2$ where
the half-fibers have been replaced by curves in~$C\subset\partial\tilde{M}_2=\partial M_1$. It can be checked that the retraction seen in Proposition~\ref{prop-MV} sends $\gamma_1'^{-e}\lambda_1$ to  ${\lambda'_1}$ and hence this
construction provides the cutting curve for~$\tilde{M}_2$.

Figure~\ref{fig:dibujo-cilindro-unfibrado-azul} shows the cutting curves of $\tilde{M}_2$
for $g=1$, $e=1$. Note that the blue curves in~$C$ turn around as $\gamma_1$
when going from $\gs$ to $\gn$. The closed curve $\gamma_1'$ is oriented as
boundary of $N$ and $\gamma_1$ is parallel to $\gamma_1'$.

It is clear that in the case of $e=-1$, the same thing will happen but
instead of turning as $\gamma_1$, the curves will turn as~$\gamma_1^{-1}$, see~Figure~\ref{dibujo-diagrama-g1e-1}.

\begin{figure}[ht]
\definecolor{cffffff}{RGB}{255,255,255}
\definecolor{cacacac}{RGB}{172,172,172}
\begin{subfigure}[b]{0.45\textwidth}
\begin{tikzpicture}[y=0.80pt, x=0.8pt,yscale=-.75, xscale=.75,inner sep=0pt, 
outer sep=0pt
,vertice/.style={draw,circle,fill,minimum size=0.15cm,inner sep=0}]
\tikzset{flecha1/.style={decoration={
  markings,
  mark=at position #1 with  {\arrowreversed{angle 90 
new}}},postaction={decorate}}}
\tikzset{flecha2/.style={decoration={
  markings,
  mark=at position #1 with  {\arrow{angle 90 new}}},postaction={decorate}}}

\node[above right] at (185,252.32137)  {$\lambda_1$};
\node[above right] at (249.50768,152.31625)  {$\gamma_1$};
\node[above right] at  (185,140)  {${\lambda'}_1$};

\draw[thick] (310,82) ellipse (91 and 35);

\path[draw=black,line join=miter,line cap=butt,line width=1pt] (219,82) -- 
(219,300);
\path[draw=black,line join=miter,line cap=butt,line width=1pt] (401,82) -- 
(401,300);
\draw[thick,flecha2=.5] (219,300) arc (180:0:91 and 40);

\draw[thick,dash pattern=on 2.40pt off 0.80pt] (219,300) arc (-180:0:91 and 40);

\path[draw=black,line join=miter,line cap=butt,line width=0.800pt,flecha2=.7] 
(219,300) .. controls (212.8561,223.8464) and (401,260) .. (401,192);

\path[draw=black,dash pattern=on 0.80pt off 2.40pt,line join=miter,line 
cap=butt,miter limit=4.00,line width=0.800pt] (219,82) .. controls 
(223.7457,134.6386) and (402.5428,150.8010) .. (401,192);

\path[draw=black,dash pattern=on 0.80pt off 2.40pt,line join=miter,line 
cap=butt,miter limit=4.00,line width=0.800pt] (401,294) .. controls 
(401.5387,234.6437) and (217.6909,232.6234) .. (219,192);

\path[draw=black,line join=miter,line cap=butt,line width=0.800pt,flecha2=.7] 
(401,82) .. controls (395.6039,157.0971) and (219,145) .. (219,192);

\path[fill=cacacac,fill opacity=0.415]  (219,187) --(219,300).. controls 
(212.8561,223.8464) and (401,260) .. (401,192)--(401,82).. controls 
(395.6039,157.0971) and (219,145) .. (219,192)--cycle;

\path[draw=black,line join=miter,line cap=butt,miter limit=4.00,line 
width=1.212pt, flecha1=0.2, flecha1=.4,flecha1=.9] (162,181) -- (212,181) -- 
(212,87) -- (162,87);

\path[draw=blue,line join=miter,line cap=butt,miter limit=4.00,line 
width=1.440pt] (212,181) -- (219,184) .. controls (227.3065,183.6310) and 
(217.3269,249.8127) .. (228.2945,251.3112) .. controls (241.1233,253.0640) and 
(272.3312,235.2952) .. (286.8833,230.0980) .. controls (301.0255,225.0473) and 
(404.0610,208.8848) .. (401,182);
\path[draw=blue,dash pattern=on 8.96pt off 8.96pt,line join=miter,line 
cap=butt,miter limit=4.00,line width=1.120pt] (401,182) .. controls 
(400.0204,149.2858) and (297.9950,126.0523) .. (283.8529,121.0015) .. controls 
(270.5346,116.2450) and (219.2031,101.8086) .. (219,88);

\path[draw=blue,line join=miter,line cap=butt,miter limit=4.00,line 
width=1.440pt](219,88)--(212,87);

 \path[draw=black,line join=miter,line cap=butt,line width=0.800pt, flecha2=.5] 
(411,300) -- (411,87);
\node[above right] at (417.19299,194.74268) {$\phi_1$};
\node[below right] at (365,232) {$\gamma'_1$};
\node[above right] at (297.995,332) {$\mu_1$};
\path[draw=black,line join=miter,line cap=butt,miter limit=4.00,line 
width=1.360pt, flecha1=0.1, flecha1=.4,flecha1=.9] (162,190) -- (212,190) -- 
(212,300) -- (162,300);

\node at (300,343) {};

\end{tikzpicture}
\caption{From $M_2$ to $\tilde{M}_2$, $e=1$.}
\label{dibujo-cilindro-unfibrado}
\end{subfigure}
\begin{subfigure}[b]{.45\textwidth}
\definecolor{cff0000}{RGB}{255,0,0}
\definecolor{c0000ff}{RGB}{0,0,255}

\begin{tikzpicture}[y=-0.80pt, x=0.8pt,yscale=-.75,, inner sep=0pt, outer 
sep=0pt]
\begin{scope}[shift={(0,-702.36215)}]
\path[draw=black,line join=miter,line cap=butt,miter limit=4.00,line 
width=0.960pt] (147.6849,800.9034) .. controls (165.4973,790.4617) and 
(156.8982,738.2531) .. (185.7665,740.0957) .. controls (214.6348,741.9384) and 
(210.3353,792.9186) .. (221.3912,797.8323);
\path[draw=black,line join=miter,line cap=butt,miter limit=4.00,line 
width=0.960pt] (185.1523,799.0607) .. controls (168.5684,793.5328) and 
(176.5532,760.3649) .. (186.9950,759.7507) .. controls (197.4367,759.1365) and 
(208.4926,788.0048) .. (194.9798,797.2181);
\path[draw=black,line join=miter,line cap=butt,miter limit=4.00,line 
width=0.960pt] (159.9693,776.9489) .. controls (120.0451,783.7053) and 
(91.1768,824.8579) .. (153.2129,862.9395);
\path[draw=black,line join=miter,line cap=butt,miter limit=4.00,line 
width=0.960pt] (176.5532,775.7204) .. controls (183.3096,775.7204) and 
(193.1372,775.1062) .. (200.5078,776.9489);
\path[draw=black,line join=miter,line cap=butt,miter limit=4.00,line 
width=0.960pt] (212.1780,776.3346) .. controls (271.1430,792.9186) and 
(268.6861,833.4570) .. (216.4775,864.1680);
\path[draw=black,line join=miter,line cap=butt,miter limit=4.00,line 
width=0.960pt] (147.6849,976.7127) .. controls (165.4973,987.1544) and 
(156.8982,1039.3630) .. (185.7665,1037.5203) .. controls (214.6348,1035.6777) 
and (210.3353,984.6976) .. (221.3912,979.7838);
\path[draw=black,line join=miter,line cap=butt,miter limit=4.00,line 
width=0.960pt] (185.1523,978.5553) .. controls (168.5684,984.0833) and 
(176.5532,1017.2511) .. (186.9950,1017.8653) .. controls (197.4367,1018.4795) 
and (208.4926,989.6112) .. (194.9798,980.3980);
\path[draw=black,line join=miter,line cap=butt,miter limit=4.00,line 
width=0.960pt] (159.9693,1000.6672) .. controls (120.0451,993.9108) and 
(91.1768,952.7581) .. (153.2129,914.6766);
\path[draw=black,line join=miter,line cap=butt,miter limit=4.00,line 
width=0.960pt] (176.5532,1001.8956) .. controls (183.3096,1001.8956) and 
(193.1372,1002.5098) .. (200.5078,1000.6672);
\path[draw=black,line join=miter,line cap=butt,miter limit=4.00,line 
width=0.960pt] (212.1780,1001.2814) .. controls (271.1430,984.6975) and 
(268.6861,944.1591) .. (216.4775,913.4481);
\path[draw=black,line join=miter,line cap=butt,miter limit=4.00,line 
width=0.960pt] (144.6138,846.9698) .. controls (163.6546,862.9395) and 
(155.0556,912.0771) .. (150.1418,924.9757);
\path[draw=black,line join=miter,line cap=butt,miter limit=4.00,line 
width=0.960pt] (224.7967,850.0409) .. controls (205.7559,866.0106) and 
(214.3550,915.1482) .. (219.2687,928.0468);
\path[draw=cff0000,line join=miter,line cap=butt,miter limit=4.00,line 
width=0.432pt] (156.2840,784.3195) .. controls (134.7863,797.2181) and 
(142.7712,810.1167) .. (146.4565,824.8579) .. controls (150.1418,839.5992) and 
(173.2161,853.7381) .. (171.6395,894.8789) .. controls (169.7320,944.6537) and 
(145.8423,941.9110) .. (143.3854,964.2857) .. controls (141.1730,984.4338) and 
(153.2129,990.0829) .. (157.5124,993.1540);
\path[draw=cff0000,line join=miter,line cap=butt,miter limit=4.00,line 
width=0.432pt] (174.7106,781.2484) .. controls (189.4518,781.8626) and 
(186.8746,796.8672) .. (185.7665,816.2589) .. controls (184.5381,837.7566) and 
(190.6803,836.5281) .. (191.2945,894.2647) .. controls (191.9087,952.0013) and 
(188.8376,967.9710) .. (188.2234,977.7985) .. controls (187.6092,987.6260) and 
(183.9239,990.6971) .. (175.9390,994.9966);
\path[draw=cff0000,line join=miter,line cap=butt,miter limit=4.00,line 
width=0.432pt] (177.1675,1021.4080) .. controls (191.2945,1033.0782) and 
(206.0358,1010.9663) .. (207.8784,994.9967) .. controls (209.7211,979.0270) and 
(206.0358,958.1435) .. (204.1931,889.3510) .. controls (202.3504,820.5584) and 
(210.3353,797.8323) .. (209.7211,791.6901) .. controls (209.1069,785.5479) and 
(205.4215,750.5374) .. (186.3807,749.9232) .. controls (167.3400,749.3090) and 
(164.8831,780.6342) .. (164.2689,795.9897) .. controls (163.6546,811.3451) and 
(181.4670,839.5992) .. (180.2386,890.5794) .. controls (179.0101,941.5596) and 
(163.0404,975.3417) .. (166.1115,990.6972) .. controls (169.1826,1006.0526) and 
(177.1675,1021.4080) .. (177.1675,1021.4080) -- cycle;
\path[draw=c0000ff,line join=miter,line cap=butt,miter limit=4.00,line 
width=0.480pt] (158.7409,779.4057) .. controls (139.0859,788.6190) and 
(134.5225,795.3641) .. (134.7863,808.8883) .. controls (135.1216,826.0801) and 
(163.6546,856.7974) .. (159.9693,864.7822) .. controls (158.3212,868.3531) and 
(155.0556,875.8382) .. (155.0556,875.8382);
\path[draw=c0000ff,dash pattern=on 0.48pt off 0.96pt,line join=miter,line 
cap=butt,miter limit=4.00,line width=0.480pt] (155.0556,873.9955) .. controls 
(155.0556,880.1377) and (212.7922,884.4372) .. (213.4064,893.6505);
\path[draw=c0000ff,line join=miter,line cap=butt,miter limit=4.00,line 
width=0.480pt] (213.4064,893.0363) .. controls (214.0206,905.3207) and 
(159.9693,902.8638) .. (157.5124,913.3055) .. controls (155.0556,923.7472) and 
(131.2325,962.9460) .. (131.7152,970.4280) .. controls (132.9437,989.4687) and 
(148.9134,996.2251) .. (159.3551,998.0678);
\path[draw=c0000ff,line join=miter,line cap=butt,miter limit=4.00,line 
width=0.480pt] (175.9390,998.6820) .. controls (184.5381,996.8394) and 
(192.5229,988.0248) .. (192.5229,980.2554) .. controls (192.5229,969.1994) and 
(191.2945,922.5188) .. (198.0509,916.3766) .. controls (204.8073,910.2344) and 
(213.4064,909.0060) .. (214.0206,904.0922);
\path[draw=c0000ff,dash pattern=on 0.48pt off 0.96pt,line join=miter,line 
cap=butt,miter limit=4.00,line width=0.480pt] (214.0206,904.7065) .. controls 
(214.0206,897.3358) and (155.6698,895.4932) .. (156.2840,890.5794);
\path[draw=c0000ff,line join=miter,line cap=butt,miter limit=4.00,line 
width=0.480pt] (155.6698,891.1936) .. controls (156.8982,884.4372) and 
(195.5940,882.5946) .. (196.8225,873.9955) .. controls (198.0509,865.3964) and 
(190.0661,802.7461) .. (190.6803,797.2181) .. controls (191.2945,791.6901) and 
(193.1372,778.7915) .. (175.9390,778.7915);
\path[draw=c0000ff,line join=miter,line cap=butt,miter limit=4.00,line 
width=0.480pt] (185.1523,1030.6213) .. controls (193.7514,1030.6213) and 
(206.8157,1017.5654) .. (210.3353,997.4536) .. controls (214.6348,972.8848) and 
(206.0358,932.9605) .. (209.7211,922.5188) .. controls (210.9131,919.1415) and 
(214.0206,918.8335) .. (214.6348,912.0771);
\path[draw=c0000ff,dash pattern=on 0.48pt off 0.96pt,line join=miter,line 
cap=butt,miter limit=4.00,line width=0.480pt] (215.2491,912.0771) .. controls 
(216.4775,907.1633) and (154.4413,907.1633) .. (155.0556,899.7927);
\path[draw=c0000ff,line join=miter,line cap=butt,miter limit=4.00,line 
width=0.480pt] (155.0556,900.4069) .. controls (156.2840,894.8789) and 
(206.6500,888.1225) .. (207.2642,883.8230) .. controls (207.8784,879.5235) and 
(214.0206,795.9897) .. (212.7922,791.0759) .. controls (211.5637,786.1621) and 
(207.8549,742.8898) .. (184.5381,745.0095) .. controls (164.2689,746.8521) and 
(160.6568,786.1559) .. (159.9693,795.9896) .. controls (158.8951,811.3567) and 
(179.6243,854.9547) .. (174.7106,869.0817) .. controls (172.2724,876.0915) and 
(155.6698,879.5235) .. (155.6698,885.0514);
\path[draw=c0000ff,dash pattern=on 0.48pt off 0.96pt,line join=miter,line 
cap=butt,miter limit=4.00,line width=0.480pt] (155.6698,884.4372) .. controls 
(155.0556,889.9652) and (215.2491,896.7216) .. (214.0206,899.7927);
\path[draw=c0000ff,line join=miter,line cap=butt,miter limit=4.00,line 
width=0.480pt] (214.0206,899.7927) .. controls (214.6348,905.3207) and 
(173.4821,909.6202) .. (172.2537,916.3766) .. controls (171.0253,923.1330) and 
(160.5835,976.5701) .. (160.5835,983.9407) .. controls (160.5835,991.3113) and 
(168.5684,1030.0071) .. (185.1523,1030.6213);
\end{scope}

\end{tikzpicture}
\caption{Example for the case $g=1$, $e=-1$.}
\label{dibujo-diagrama-g1e-1}
\end{subfigure}
\caption{}
\end{figure}

\section{\texorpdfstring{Heegaard splittings of arbitrary $\bs^1$-bundles over 
 surfaces}%
{Heegaard decomposition of arbitrary S1-bundles over  surfaces}}
\label{sec-eulern}

In order to construct a Heegaard splitting for arbitrary Euler number~$e$ we proceed as 
follows.
Let now $\check{S}:=\overline{S\setminus\bigcup_{j=1}^n D_i}$, where $D_1,\dots,D_n$ are 
pairwise disjoint
closed disks in $S$. As before, let $s_1,s_2:\check{S}\to M$ be arbitrary 
parallel sections of~$\pi$.
For each $j=1,\dots,n$, let $\gamma_j:=s_1(\partial D_j)$ (oriented as part of 
$\partial\check{S}$)
and let $\mu_j$ be the boundary of a meridian disk of the solid torus 
$\pi^{-1}(D_j)$. 
By the choice of orientations
the cycle $\gamma_j\cdot\mu_j\cdot \phi^{e_j}$ 
is trivial in $H_1(\pi^{-1}(\partial D_j);\bz)$, 
for some $e_j\in\bz$, where
$\phi$ is an oriented fiber of~$\pi$.
The following is a classical result.

\begin{lema}
With the above notations, $e=\sum_{j=1}^n e_j$. Moreover, for every choice of 
the $e_j$'s satisfying this equality, there exists a choice of sections that 
realizes it.
\end{lema}

As we did in \S\ref{sec-euler1}, we may decompose 
$\check{M}:=\pi^{-1}(\check{S})$ in two pieces $M_1$ and $M_2$;
$M_1$ and $M_2$ are oriented compact $3$-manifolds with boundary
and $M_1\cap M_2=\partial M_1\cap \partial M_2=s_1(\check{S})\coprod 
s_2(\check{S})$ with the same orientation convention.
From Theorem~\ref{thm-cuerpo-asas}, the manifolds $M_1$ and
$M_2$ are $(2g+n-1)$-handle bodies.

Let us assume that $e_j=\pm 1$, $j=1,\dots,n$. 
Note that $M_2$ is homeomorphic to $M_1$ and hence, it is also a 
$(2g+n-1)$-handle body.
Let $\tilde{M}_2:=M_2\cup\bigcup_{j=1}^n \pi^{-1}(D_j)$. Following the arguments
in the proof of Theorem~\ref{thm-pegado-toro1}, we can see that 
$M_2\cong\tilde{M}_2$ and
$M_1$ and $\tilde{M}_2$ have the same boundary. We have proven the following 
result.

\begin{thm}
The submanifolds $M_1$ and $\tilde{M}_2$ form a Heegaard splitting of $M$.
If $e=0$, a decomposition of this kind of genus $2 g+1$ can be always obtained; and if $e\neq 0$,
one of genus $2 g+|e|-1$.
\end{thm}

\begin{obs}
In this process, we have glued all the solid tori $ \pi^{-1}(D_j)$ to $M_2$. 
This is not essential
for the proof:  we could have glued some of them to $M_1$ and the result would 
be equally valid.
\end{obs}

Let us describe
the cutting curves of $M_1$. First, we consider the cutting
curves of 
\S\ref{sec:diagramae1}. Second, we add curves $c_j$, $j=2,\dots,n$
as follows. Consider the paths $\alpha_j$ (as in the proof of Theorem~\ref{thm-cuerpo-asas}) 
joining $p_j\in\partial D_1$ and 
$q_j\in\partial D_j$; recall that by cutting along them $\check{S}$ becomes a disk.
%
The boundaries
$$
c_j=s_1(\alpha_j)\cdot(\{q_j\}\times I)\cdot s_1(\alpha_j)^{-1}\cdot(\{p_j\}\times I)^{-1}.
$$
of $\alpha_j\times I$, together with the curves of \S\ref{sec:diagramae1},
form a system of cutting curves for~$M_1$.

Following the arguments in \S\ref{sec:diagramae1}, the curves of $M_2$ mimic the ones of $M_1$ except for the 
modification in the cylinders $\partial D_i\times I$, $1\leq i\leq n$, due to the float gluing
of the solid tori $\pi^{-1}(D_i)$. By the same reasoning as before,
these modifications consist on a Dehn twist along each cylinder. The orientation of each
Dehn twist depends on the sign of each $e_i$.
Note that the cylinder $\partial D_1\times I$ plays a special role;
it will be called \emph{main cylinder}.

\begin{figure}[ht]
\centering
\begin{subfigure}[b]{.45\textwidth}
\definecolor{cff0000}{RGB}{255,0,0}
\definecolor{c0000ff}{RGB}{0,0,255}
\begin{tikzpicture}[y=0.80pt, x=0.8pt,yscale=-1.3, inner sep=0pt, outer sep=0pt]
\def\arriba{\path[draw=black,line join=miter,line cap=butt,line width=1.200pt]
    (48,922) .. controls (37,919) and (21,910) ..
    (21,900) .. controls (20,900) and (25,862) ..
    (91,865) .. controls (157,863) and (187,875) ..
    (180,900) .. controls (179,910) and (177,910) ..
    (162,922)}

\def\abajo{\path[draw=black,line join=miter,line cap=butt,line width=1.200pt]
    (71,924) .. controls (74,924) and (88,924.6) ..
    (92,923.6);
  \path[draw=black,line join=miter,line cap=butt,line width=1.200pt]
    (113,924.6) .. controls (123,924.7) and (131,924.7) ..
    (136,924.8);
}

\arriba;
\abajo
\begin{scope}[yscale=-1,yshift=-1520]
\arriba;
\abajo
\end{scope}

 \path[draw=black,line join=miter,line cap=butt,line width=1.200pt]
    (45.1413,910.5685) .. controls (51.1557,928.9792) and (51.8239,955.2802) ..
    (50.8215,963.6088) .. controls (49.8191,971.9375) and (47.2262,985.5335) ..
    (44.8873,990.7937);

  \path[draw=black,line join=miter,line cap=butt,line width=1.200pt]
    (73.0131,910.9500) .. controls (70.4797,925.8778) and (68.1996,946.9525) ..
    (68.4530,955.2945) .. controls (68.7063,963.6366) and (71.7464,983.8331) ..
    (73.5198,988.2237);

  \path[draw=black,line join=miter,line cap=butt,line width=1.200pt]
    (88.7865,908.7827) .. controls (93.8740,927.2085) and (94.4393,953.5312) ..
    (93.5914,961.8667) .. controls (92.7435,970.2022) and (91.2576,982.9379) ..
    (89.2792,988.2024);

  \path[draw=black,line join=miter,line cap=butt,line width=1.200pt]
    (115.3074,910.4968) .. controls (112.9766,925.4306) and (110.8789,946.5137) ..
    (111.1120,954.8591) .. controls (111.3450,963.2045) and (114.1420,983.4091) ..
    (115.7736,987.8014);

  \path[draw=black,line join=miter,line cap=butt,line width=1.200pt]
    (133.3172,908.3327) .. controls (138.0968,926.7639) and (138.6278,953.0941) ..
    (137.8312,961.4320) .. controls (137.0346,969.7699) and (135.7712,982.0732) ..
    (133.9125,987.3393);

  \path[draw=black,line join=miter,line cap=butt,line width=1.200pt]
    (164.2186,910.0895) .. controls (161.4820,925.0116) and (159.0191,946.0781) ..
    (159.2927,954.4169) .. controls (159.5664,962.7557) and (162.8503,982.9445) ..
    (164.7659,987.3333);


  \path[draw=cff0000,line join=miter,line cap=butt,miter limit=4.00,very thick] (64.6467,985.9161) .. controls (62.9918,971.5604) and
    (60.7074,942.5836) .. (63.3336,910.1936) .. controls (63.7581,904.9583) and
    (67.7106,892.6855) .. (77.3401,892.6855) .. controls (86.9695,892.6855) and
    (91.9239,901.8385) .. (94.8482,910.1936) .. controls (97.9121,918.9476) and
    (97.4744,978.4752) .. (94.4105,986.3538) .. controls (92.9055,990.2239) and
    (87.4072,1000.7980) .. (80.8417,1001.2357) .. controls (74.2762,1001.6734) and
    (66.6597,1003.3783) .. (64.6467,985.9161) -- cycle;

  \path[draw=cff0000,line join=miter,line cap=butt,miter limit=4.00,very thick] (108.9442,987.9770) .. controls (107.2893,973.6212) and
    (105.0049,933.2641) .. (107.6311,912.2544) .. controls (108.2826,907.0426) and
    (112.0081,894.7463) .. (121.6376,894.7463) .. controls (131.2670,894.7463) and
    (136.2214,903.8994) .. (139.1457,912.2544) .. controls (142.2096,921.0085) and
    (143.9604,980.5360) .. (140.8965,988.4147) .. controls (139.3914,992.2847) and
    (131.7047,1002.8588) .. (125.1392,1003.2965) .. controls (118.5736,1003.7342)
    and (110.9572,1005.4391) .. (108.9442,987.9770) -- cycle;

  \path[draw=c0000ff,line join=miter,line cap=butt,miter limit=4.00,very thick] (136.8676,936.0180) .. controls (144.3086,930.3279) and
    (143.8709,920.6984) .. (142.1201,914.1329) .. controls (140.3692,907.5673) and
    (139.9315,891.8099) .. (121.5480,892.2478) .. controls (110.6076,892.5084) and
    (104.9153,901.0018) .. (104.4776,919.3853) .. controls (104.1651,932.5127) and
    (95.7236,934.2672) .. (93.0974,937.7688);

  \path[draw=c0000ff,dash pattern=on 1pt off 1pt,line join=miter,line
    cap=butt,miter limit=4.00,very thick] (93.0974,937.3311) .. controls
    (100.9760,939.0819) and (107.9793,961.4048) .. (111.4809,961.8425);

  \path[draw=c0000ff,line join=miter,line cap=butt,miter limit=4.00,very thick] (111.4809,961.8425) .. controls (104.9153,966.2195) and
    (104.9153,970.1588) .. (105.3530,976.7244) .. controls (105.7908,983.2899) and
    (104.9153,997.2964) .. (112.7940,1003.8619) .. controls (117.5730,1007.8444)
    and (125.5074,1006.1091) .. (129.4267,1004.7373) .. controls
    (138.1808,1001.6734) and (143.6153,992.9806) .. (143.8709,987.2292) ..
    controls (144.7463,967.5326) and (153.5003,960.5294) .. (159.6282,958.7786);

  \path[draw=c0000ff,dash pattern=on 1pt off 1pt,line join=miter,line
    cap=butt,miter limit=4.00,very thick] (159.1905,959.2163) .. controls
    (153.0626,958.3409) and (140.3693,936.8934) .. (137.3053,936.4557);

  \path[draw=c0000ff,line join=miter,line cap=butt,miter limit=4.00,very thick] (90.9089,924.2573) .. controls (98.3498,918.5672) and
    (98.5226,907.2126) .. (94.8482,901.4968) .. controls (90.9089,895.3690) and
    (85.6588,889.6286) .. (76.9024,889.2411) .. controls (69.4741,888.9124) and
    (62.9078,895.4362) .. (60.7074,907.1869) .. controls (59.3583,914.3913) and
    (51.5156,929.0721) .. (48.8894,932.5737);

  \path[draw=c0000ff,dash pattern=on 1pt off 1pt,line join=miter,line
    cap=butt,miter limit=4.00,very thick] (49.3271,930.8229) .. controls
    (57.2058,932.5737) and (64.6467,953.5834) .. (68.1483,954.0211);

  \path[draw=c0000ff,line join=miter,line cap=butt,miter limit=4.00,very thick] (68.1483,954.0211) .. controls (61.5828,958.3981) and
    (59.8320,967.5899) .. (60.2697,974.1554) .. controls (60.7074,980.7210) and
    (60.4988,1000.5844) .. (70.3368,1003.4815) .. controls (76.4290,1005.2755) and
    (84.2843,1006.9310) .. (93,1000) .. controls (99,988.5728) and
    (97.6908,977.2820) .. (102,971.5292) .. controls (105,960.1490) and
    (105.7908,948.7687) .. (111.9186,947.0179);

  \path[draw=c0000ff,dash pattern=on 1pt off 1pt,line join=miter,line
    cap=butt,miter limit=4.00,very thick] (111.0432,947.0179) .. controls
    (104.9153,946.1425) and (94.4105,925.1327) .. (91.3466,924.6950);
\end{tikzpicture}
 \caption{Example for the case $g=0$, $e=3$.}
\label{fig:g0e3}
\end{subfigure}
\begin{subfigure}[b]{.45\textwidth}
\definecolor{cff0000}{RGB}{255,0,0}
\definecolor{c0000ff}{RGB}{0,0,255}

\begin{tikzpicture}[y=0.80pt, x=0.8pt,yscale=-.6, inner sep=0pt, outer sep=0pt]
\begin{scope}[shift={(-277.84067,-130.38785)}]
\def\arriba{\path[draw=black,line join=miter,line cap=butt,miter limit=4.00,line
    width=2.000pt] (362,181) .. controls (341,189) and
    (321,132) .. (288,147) .. controls (256,161) and
    (323,215) .. (317,225);
  \path[draw=black,line join=miter,line cap=butt,miter limit=4.00,line
    width=2.000pt] (344,199) .. controls (344,187) and
    (310,156) .. (299,162) .. controls (286,169) and
    (323,208) .. (337,207);
  \path[draw=black,line join=miter,line cap=butt,miter limit=4.00,line
    width=2.000pt] (337,166) .. controls (345,160) and
    (363,154) .. (375,155) .. controls (388,155) and
    (404,160) .. (418,168) -- (428,176) .. controls
    (433,181) and (439,187) .. (442,192) --
    (448,205) .. controls (461,235) and (447,283) ..
    (410,300);
  \path[draw=black,line join=miter,line cap=butt,miter limit=4.00,line
    width=2.000pt] (389,307) .. controls (382,309) and
    (379,308) .. (373,307);
  \path[draw=black,line join=miter,line cap=butt,miter limit=4.00,line
    width=2.000pt] (324,290) .. controls (303,264) and
    (291,246) .. (305,205);
  \path[draw=black,line join=miter,line cap=butt,miter limit=4.00,line
    width=2.000pt] (325,174) .. controls (316,184) and
    (317,183) .. (312,190);
}

\arriba
\begin{scope}[yscale=-1,yshift=-526]
\arriba
\end{scope}

  \path[draw=black,line join=miter,line cap=butt,miter limit=4.00,line
    width=2.000pt] (319.7430,265.5190) .. controls (328.3293,275.1154) and
    (323.5744,395.4733) .. (317.5135,402.5444);

  \path[draw=black,line join=miter,line cap=butt,miter limit=4.00,line
    width=2.000pt] (378.5125,270.6669) .. controls (369.9262,280.2634) and
    (371.8240,391.6927) .. (377.8849,398.7638);

  \path[draw=black,line join=miter,line cap=butt,miter limit=4.00,line
    width=2.000pt] (384.7734,281.8683) .. controls (393.3597,291.4648) and
    (389.3191,376.8227) .. (383.2582,383.8937);

  \path[draw=black,line join=miter,line cap=butt,miter limit=4.00,line
    width=2.000pt] (414.2572,281.8683) .. controls (405.6709,291.4648) and
    (409.7116,376.8227) .. (415.7725,383.8937);






  \path[draw=cff0000,line join=miter,line cap=butt,line width=0.800pt]
    (348.4667,286.3812) .. controls (348.9718,264.1578) and (368.2006,201.3087) ..
    (360.6245,194.7427) .. controls (353.0483,188.1767) and (299.4743,132.7662) ..
    (288.8677,154.9896) .. controls (278.2611,177.2129) and (317.9635,211.8540) ..
    (320.7234,218.4813) .. controls (323.4832,225.1080) and (327.6509,254.1071) ..
    (332.0233,282.3913) .. controls (336.3505,310.3837) and (334.6353,346.3881) ..
    (332.5897,379.7231) .. controls (330.5289,413.3035) and (325.7742,430.1082) ..
    (317.1879,441.2199) .. controls (308.6016,452.3316) and (274.7615,491.2224) ..
    (284.3579,502.3341) .. controls (293.9544,513.4458) and (309.4055,514.1467) ..
    (343.4519,475.0600) .. controls (358.0853,458.2602) and (343.4999,428.3216) ..
    (347.4461,373.0600) .. controls (347.6144,353.5304) and (348.2984,317.4670) ..
    (348.4667,286.3812) -- cycle;

  \path[draw=c0000ff,line join=miter,line cap=butt,miter limit=4.00,line
    width=0.400pt] (372.4783,327.7935) .. controls (367.6367,339.3700) and
    (345.8467,391.3402) .. (319.7814,447.4563) .. controls (313.8658,460.1920) and
    (283.9210,485.3370) .. (293.5174,496.4487) .. controls (303.1139,507.5603) and
    (307.2287,509.6338) .. (342.0048,471.1948) .. controls (353.1612,458.8633) and
    (339.5591,435.8114) .. (350.4719,411.0891) .. controls (362.3809,384.1101) and
    (369.7259,362.7328) .. (372.2691,343.1654)(325.5762,325.0440) .. controls
    (327.0383,317.8517) and (333.8120,321.4037) .. (340.5666,280.1453) .. controls
    (349.3306,226.6128) and (345.3838,241.5763) .. (348.3257,196.8518) .. controls
    (348.9837,186.8479) and (301.1802,142.8340) .. (294.1092,160.5117) .. controls
    (284.9638,183.3752) and (321.8017,210.5755) .. (326.3474,216.1313) .. controls
    (330.8931,221.6872) and (332.6368,283.3678) .. (324.9486,308.3257);

  \path[draw=c0000ff,dash pattern=on 0.40pt off 0.80pt,line join=miter,line
    cap=butt,miter limit=4.00,line width=0.400pt] (370.7143,327.3622) .. controls
    (371.0714,321.2908) and (325.0000,315.2193) .. (325.0000,307.7193);

  \path[draw=c0000ff,dash pattern=on 0.40pt off 0.80pt,line join=miter,line
    cap=butt,miter limit=4.00,line width=0.400pt] (370.7133,341.8265) .. controls
    (371.0704,335.7550) and (324.9990,329.6836) .. (324.9990,322.1836);

  \path[draw=cff0000,line join=miter,line cap=butt,line width=0.800pt]
    (343.9286,173.0765) .. controls (348.2143,168.4336) and (357.8571,164.5050) ..
    (363.9286,170.5765) .. controls (370.0000,176.6479) and (377.8571,193.4336) ..
    (376.0714,203.0765) .. controls (374.2857,212.7193) and (358.5714,253.0765) ..
    (358.2143,271.2908) .. controls (357.8571,289.5050) and (354.2857,415.9336) ..
    (357.1429,428.0765) .. controls (360.0000,440.2193) and (372.8571,450.5765) ..
    (371.7857,468.0765) .. controls (370.7143,485.5765) and (352.1429,495.2193) ..
    (342.5000,483.4336);

  \path[draw=cff0000,line join=miter,line cap=butt,line width=0.800pt]
    (336.7857,187.3622) .. controls (332.8571,194.8622) and (343.5714,203.7908) ..
    (343.5714,213.7908) .. controls (343.5714,223.7908) and (340.0000,436.2908) ..
    (340.7143,438.0765) .. controls (341.4286,439.8622) and (343.2143,452.0050) ..
    (339.6429,455.2193) .. controls (336.0714,458.4336) and (332.1429,465.9336) ..
    (335.0000,471.2908);

  \path[draw=cff0000,dash pattern=on 0.80pt off 1.60pt,line join=miter,line
    cap=butt,miter limit=4.00,line width=0.800pt] (342.8571,174.8622) --
    (337.5000,184.8622);

  \path[draw=cff0000,dash pattern=on 0.80pt off 1.60pt,line join=miter,line
    cap=butt,miter limit=4.00,line width=0.800pt] (336.0714,472.7193) .. controls
    (341.0714,480.9336) and (341.4286,480.9336) .. (341.4286,480.9336);

  \path[draw=c0000ff,line join=miter,line cap=butt,miter limit=4.00,line
    width=0.400pt] (341.6504,171.8155) .. controls (346.8274,167.1019) and
    (357.0471,160.9706) .. (364.3812,167.1345) .. controls (371.7153,173.2985) and
    (382.6351,192.4828) .. (380.4780,202.2725) .. controls (378.3209,212.0623) and
    (346.8385,349.1057) .. (323.1928,354.0260)(374.0414,379.2717) .. controls
    (346.7276,404.3688) and (374.8768,450.5749) .. (375.3010,471.3097) .. controls
    (375.6653,489.1197) and (354.7536,502.4038) .. (338.4962,485.8294);

  \path[draw=c0000ff,line join=miter,line cap=butt,miter limit=4.00,line
    width=0.400pt] (333.5220,182.5408) .. controls (329.5934,190.0408) and
    (340.3884,203.9789) .. (339.9506,213.9693) .. controls (339.7474,218.6044) and
    (334.1015,297.8757) .. (325.5760,316.3600)(371.8034,335.5067) .. controls
    (368.7986,354.5069) and (337.6067,430.0947) .. (337.4506,430.7550) .. controls
    (334.2710,444.2062) and (340.3077,450.3979) .. (336.7363,453.6122) .. controls
    (333.1649,456.8265) and (328.1649,470.3979) .. (331.0220,475.7550);

  \path[draw=c0000ff,dash pattern=on 0.40pt off 0.80pt,line join=miter,line
    cap=butt,miter limit=4.00,line width=0.400pt] (370.7133,336.4693) .. controls
    (371.0704,330.3979) and (324.9990,324.3265) .. (324.9990,316.8265);

  \path[draw=c0000ff,dash pattern=on 0.40pt off 0.80pt,line join=miter,line
    cap=butt,miter limit=4.00,line width=0.400pt] (338.2143,484.8622) .. controls
    (338.2143,484.8622) and (333.2143,479.1479) .. (331.0714,476.2908);

  \path[draw=c0000ff,dash pattern=on 0.40pt off 1.20pt,line join=miter,line
    cap=butt,miter limit=4.00,line width=0.400pt] (340.0000,173.0765) --
    (335.0000,181.2908);

  \path[draw=c0000ff,dash pattern=on 0.40pt off 0.80pt,line join=miter,line
    cap=butt,miter limit=4.00,line width=0.400pt] (373.5704,379.3265) .. controls
    (373.9275,373.2550) and (323.9275,360.7550) .. (323.9275,353.2550);

  \path[draw=cff0000,line join=miter,line cap=butt,line width=0.800pt]
    (385.3571,260.5765) .. controls (365.0000,260.3979) and (363.5714,296.6479) ..
    (364.2857,317.7193) .. controls (365.0000,338.7908) and (360.0000,395.9336) ..
    (367.5000,405.2193) .. controls (375.0000,414.5050) and (398.2630,394.4326) ..
    (402.1429,377.7193) .. controls (406.7857,357.7193) and (405.7143,260.7550) ..
    (385.3571,260.5765) -- cycle;

  \path[draw=c0000ff,line join=miter,line cap=butt,miter limit=4.00,line
    width=0.400pt] (322.5000,363.4336) .. controls (345.3571,348.7908) and
    (359.0910,342.0495) .. (369.6429,284.5050) .. controls (374.4764,258.1456) and
    (388.1673,266.2906) .. (390.7143,278.4336) .. controls (393.2613,290.5766) and
    (397.5000,297.7193) .. (390.7143,312.0050);

  \path[draw=c0000ff,line join=miter,line cap=butt,miter limit=4.00,line
    width=0.400pt] (409.6429,338.7908) .. controls (403.9286,345.2193) and
    (395.8178,357.9638) .. (394.6429,370.5765) .. controls (393.0480,387.6962) and
    (384.1819,400.2732) .. (378.2143,402.3622) .. controls (366.1114,406.5989) and
    (369.4741,395.5440) .. (373.9286,386.2907);

  \path[draw=c0000ff,dash pattern=on 0.40pt off 0.80pt,line join=miter,line
    cap=butt,miter limit=4.00,line width=0.400pt] (374.6418,386.8265) .. controls
    (374.9990,380.7550) and (323.5704,370.3979) .. (323.5704,362.8979);

  \path[draw=c0000ff,dash pattern=on 0.40pt off 0.80pt,line join=miter,line
    cap=butt,miter limit=4.00,line width=0.400pt] (408.5667,339.3265) .. controls
    (408.9238,333.2550) and (389.6381,320.0408) .. (389.6381,312.5408);

  \node at (370,200)  {$a_1$};

 \node at (270,160) {$b_1$};

 \node at (386,250) {$c_2$};

\end{scope}

\end{tikzpicture}

\caption{\label{dibujo-diagrama-g1e-2}Example for the case $g=1$, $e=2$.}
\end{subfigure}
\caption{}
\end{figure}

\begin{ejm}
Figure~\ref{fig:g0e3} shows this construction for the case of 
genus zero and Euler number equal to $3$. We choose three solid tori and sections with
$e_i=1$. The resulting Heegaard decomposition has genus~$2$ and therefore is not minimal,
since the manifold in question is a lens space, and as such admits a genus one decomposition.
Figure~\ref{dibujo-diagrama-g1e-2} 
shows an example of this construction for the case of $g=1,e=2$.
\end{ejm}

\section{Heegaard splitting of a plumbed graph manifold with an edge}\label{sec:2vertices}

Let $M$ be a plumbed graph manifold with an edge and two vertices. This manifold is obtained as follows.
We start with two manifolds $W_1$ and $W_2$, which are oriented $\bs^1$-bundles~$\pi_i$ over closed
 surfaces~$S_i$ of genus~$g_i$ and Euler numbers $e_i$, $i=1,2$. We take closed disks
$D_{i,0}\subset S_i$ and choose a system of curves $\mu_i,\phi_i$ on $\pi_i^{-1}(\partial D_{i,0})$ as follows:
the curve $\phi_i$ is an oriented fiber of $\pi_i$, and $\mu_i$ is the oriented boundary of a meridian disk of
$\pi_i^{-1}(\partial D_{i,0})$.

Then, $M$ is obtained by gluing $\pi_1^{-1}(\overline{S_1\setminus D_{1,0}})$ and
$\pi_2^{-1}(\overline{S_2\setminus D_{2,0}})$ along their boundaries. These boundaries are tori
$\pi_i^{-1}(\partial D_{i,0})$, $i=1,2$,
and the gluing is described by a matrix in $\GL(2;\bz)$ once ordered integral bases in
$H_1(\pi_i^{-1}(\partial E_i);\bz)$ are chosen. For the choice of $(\mu_{1,0},\phi_1)$
and  $(\mu_{2,0},\phi_2)$ the matrix is
$\pm\left(\begin{smallmatrix}
0&1\\
1&0
\end{smallmatrix}\right)$, depending
on the sign of the edge as described in~\S\ref{sec:graph-manifold}.
Since the edge is contractible, the cohomology class $o$ of \S\ref{sec:graph-manifold} vanishes and can be represented by any sign, yielding
to homeomorphic constructions.

Let us consider pairwise disjoint closed disks $D_{j,1},\dots,D_{j,n_j}\subset S_j\setminus D_{j,0}$, $j=1,2$.
Let $\check{S}_j:=\overline{S\setminus\bigcup_{i=0}^{n_j} D_{j,i}}$.
We consider two parallel sections $s_{j,1},s_{j,2}: \check{S}_j\to M_j$ of $\pi_j$
as in the previous section.

As in \S\ref{sec-eulern}, we denote  $\gamma_{j,i}:=s_1(\partial D_{j,i})$ (oriented as part of $\partial\check{S}_j$);
let $\mu_{j,i}$ be the boundary of a meridian disk of $\pi^{-1}(D_{j,i})$.
As in that section, we collect the integers $e_{j,i}$
appearing in the equalities (in homology of the boundary tori)
$\gamma_{j,i}\cdot\mu_{j,i}\cdot \phi_j ^{e_{j,i}}=1$, where
$\phi_j$ is a fiber of~$\pi_j$, and they must satisfy
$$
\sum_{i=0}^{n_j} e_{j,i}=e_j.
$$
We impose the following conditions:
\begin{itemize}
\item $\min\{n_1,n_2\}\geq 2$;
\item $|e_{j,i}|=1$;
\item $\varepsilon:=e_{1,0}=e_{2,0}$, determining the sign of the edge.
\item $\partial D_{i,0}\times I$ is not a main cylinder.
\end{itemize}

In this case, we can construct Heegaard splittings $M_1^i,\bar{M}_2^i$ of
$W_i$ as in Section~\ref{sec-eulern} using the systems of disks
$\{D_{j,0},\ldots,D_{j,n_j}\}$. To do the plumbing, we have to remove
$\pi_i^{-1}(\mathring{D}_{i,0})$ from $\bar{M}_2^i$, but as we saw before, this
operation doesn't change the topology (since it is the inverse of a float
gluing). Let's denote by $\bar{M}_2^{'i}$ the result of the removal of
$\pi_i^{-1}(\mathring{D}_{i,0})$ from $\bar{M}_2^i$.

Note that after the plumbing, $\mu_{1,0}$ is identified with $\phi_2^{\varepsilon}$, and $\mu_{2,0}^{\varepsilon}$
is identified with~$\phi_1$.
This implies that $\gamma_{1,0}$ and $\gamma_{2,0}$ are
homologous after the plumbing (because of the choice
of the edge sign). In particular, it means that we can choose the
sections $s_{j,i}$ in such a way that $s_{1,i}(\partial D_{1,0})$ is identified
with $s_{2,i}(\partial D_{2,0})$. This way, the two Heegaard splittings are
compatible, and we can extend them to a decomposition of $M$.

Sumarizing, we have now the following decomposition:
\begin{equation}
\label{descomposicion-dos-nodos}
 M=\left(M_1^1\cup M_1^2\right)\bigcup\left(\bar{M}_2^{'1}
\cup \bar{M}_2^{'2}\right).
\end{equation}

\begin{prop}
The manifolds $M_1^1\cup M_1^2$ and $\bar{M}_2^{'1}
\cup \bar{M}_2^{'2}$ are handle bodies, i.e.,
the decomposition{\rm~\ref{descomposicion-dos-nodos}} is a Heegaard splitting of
$M$.
\end{prop}
\begin{proof}
It is enough to prove it for $M_1^1\cup M_1^2$.
We have already seen that both $M_1^1$ and $M_1^2$ are handle-bodies. We will show now that they are glued as in Proposition~\ref{float-gluing-handlebodies}. In
order to do so, we have to see that they are glued along annuli that are
neighborhoods of a float curve.

Let us consider the torus $\pi_i^{-1}(\partial D_{1,0})$ as the product of
$\mu_{1,0}$ and $\phi_1$. The curves
$s_{1,i}(\partial D_{1,0})$ are parallel curves that meet $\phi_1$ transversally at
only one point. Let
$$
A_1^1=M_1^1\cap \pi_i^{-1}(\partial
D_{1,0})
$$
be the annulus along which the gluing is made. This annulus is a neighborhood of a
curve parallel to $s_{1,i}(\partial D_{1,0})$.

From the construction in Section~\ref{sec-eulern}, we see that
$\phi_1\cap M_1^1$ is part of a cutting curve of $M_1^1$. And moreover, its the only
intersection of a cutting curve with the torus $\pi_i^{-1}(\partial D_{1,0})$.

So the annulus $A_1^1$ is a regular neighborhood of a float curve in $M_1^1$.
Analogously, $A_1^2$ is also a float curve in $M_1^2$. By
Proposition~\ref{float-gluing-handlebodies}, we get the result.
\end{proof}


It is time now to describe a Heegaard diagram, i.e., to understand
what happens with the cutting curves during the plumbing.
Let us consider
the cylinders $A_1^1\subset{M}_1^1$ and $A_1^2\subset {M}_1^2$ which
are identified by the plumbing.

Let us fix a cutting curve $\lambda_1$ of $M_1^1$ which intersects once the
core of $A_1^1$ (a float curve). In the neighborhood of $A_1^1$, this curve is decomposed in
three connected components $\lambda_1^b,\lambda_1^c,\lambda_1^e$ where
$\lambda_1^c$ is the part of $\lambda_1$ that lies in $A_1^1$.
As in \S\ref{sec:diagramae1}, the path $\lambda_1^c$ is a half of the fiber $\phi_1$.
Analogously, the cutting curve
$\lambda_2$ in ${M}_1^2$ in a neighbourhood of $A_1^2$ can be divided in three connected components
$\lambda_2^b,\lambda_2^c,\lambda_2^e$. The path $\lambda_2^c$ is equivalent to
a half of the fiber $\phi_2$ and recall that $\phi_2$ is identified with a section $\mu_1$.

\begin{figure}[ht]

\definecolor{cffffff}{RGB}{255,255,255}
\centering
\begin{subfigure}[b]{.45\textwidth}
\begin{tikzpicture}[y=0.80pt, x=0.8pt,yscale=-.8,xscale= .8,inner sep=0pt, outer sep=0pt]

\tikzset{flecha1/.style={decoration={
  markings,
  mark=at position #1 with  {\arrowreversed{angle 90
new}}},postaction={decorate}}}
\tikzset{flecha2/.style={decoration={
  markings,
  mark=at position #1 with  {\arrow{angle 90 new}}},postaction={decorate}}}

\node at (330,132) {$\gamma_2$};

\draw[thick] (310,82) ellipse (91 and 35);
\draw[thick,flecha2=.5] (219,300) arc (180:0:91 and 40);
\draw[thick,dash pattern=on 2.40pt off 0.80pt] (219,300) arc (-180:0:91 and 40);

  \path[draw=black,line join=miter,line cap=butt,line width=0.800pt]
    (219,82) -- (219,300);

  \path[draw=black,line join=miter,line cap=butt,line width=0.800pt]
    (401,82) -- (401,300);


  \path[draw=black,line join=miter,line cap=butt,line width=0.800pt,flecha1=.7]
    (219,300) .. controls (213,224) and (400,237) ..
    (401,193);

  \path[draw=black,dash pattern=on 1pt off 2pt,line join=miter,line
    cap=butt,miter limit=4.00,line width=0.800pt] (219,82) .. controls
    (224,135) and (402,151) .. (400,193);

  \path[draw=black,dash pattern=on 1pt off 2pt,line join=miter,line
    cap=butt,miter limit=4.00,line width=0.800pt] (401,300) .. controls
    (401.5,235) and (218,232) .. (219,193);

  \path[draw=black,line join=miter,line cap=butt,line width=0.800pt,flecha2=.7]
    (401,82) .. controls (396,157) and (220,143) ..
    (219,193);

  \path[draw=black,line join=miter,line cap=butt,line width=0.800pt,flecha2=.8]
    (411,300) -- (411,82);

  \node at (420,106.17125) {$\phi_1$};

  \node at (364.66507,230.09802)
    {$\gamma_1$};

  \node at (300,350)  {$\mu_1$};

\path[flecha1=.2,draw=black,line join=miter,line cap=butt,miter limit=4.00,line
    width=1.360pt](435,164)-- (401,193);

\path[flecha1=.5,draw=black,line join=miter,line cap=butt,miter limit=4.00,line
    width=1.360pt](219,193) -- (172,161);
  \path[draw=black,line join=miter,line cap=butt,miter limit=4.00,line
    width=1.360pt,flecha1=.7](401,193) .. controls (399,204) and
    (356,213) .. (316,214) .. controls (268,213) and
    (218,204) .. (219,193);


\node at  (300,200)    {$\lambda_2^c$};

\node at (183,150)    {$\lambda_2^b$};

  \path[draw=black,dash pattern=on 3pt off 5.5pt,line join=miter,line
    cap=butt,miter limit=4.00,line width=1.360pt,red] (219,192) .. controls
    (222,195) and (220,260) .. (226,250) .. controls
    (231,260) and (252,235) .. (273,236) .. controls
    (294,220) and (402,212) .. (401,192);

\node at (423,190)    {$\lambda_2^e$};

\node at (199,245)     {$\lambda_1^c$};

\path[flecha2=.1,flecha2=.5,flecha2=.95,draw=black,line join=miter,line cap=butt,miter limit=4.00,line
    width=1.360pt] (151,275) -- (210,300) -- (210,198)
    -- (161,220);

\node at (183,305)     {$\lambda_1^b$};

\node at  (183,195) {$\lambda_1^e$};

\end{tikzpicture}
\caption{Gluing of $M_1^i$, $e_{i,0}=1$.}
\label{fig:cilindroplumbing}
\end{subfigure}
\begin{subfigure}[b]{.45\textwidth}
\begin{tikzpicture}[y=0.80pt, x=0.8pt,yscale=-.8,xscale= .8, inner sep=0pt, outer sep=0pt]

\tikzset{flecha1/.style={decoration={
  markings,
  mark=at position #1 with  {\arrowreversed{angle 90
new}}},postaction={decorate}}}
\tikzset{flecha2/.style={decoration={
  markings,
  mark=at position #1 with  {\arrow{angle 90 new}}},postaction={decorate}}}

\node at (330,132) {$\gamma_2$};

\draw[thick] (310,82) ellipse (91 and 35);
\draw[thick,flecha2=.5] (219,300) arc (180:0:91 and 40);
\draw[thick,dash pattern=on 2.40pt off 0.80pt] (219,300) arc (-180:0:91 and 40);

  \path[draw=black,line join=miter,line cap=butt,line width=0.800pt]
    (219,82) -- (219,300);

  \path[draw=black,line join=miter,line cap=butt,line width=0.800pt]
    (401,82) -- (401,300);


  \path[draw=black,line join=miter,line cap=butt,line width=0.800pt,flecha1=.7]
    (219,300) .. controls (213,224) and (400,237) ..
    (401,193);

  \path[draw=black,dash pattern=on 1pt off 2pt,line join=miter,line
    cap=butt,miter limit=4.00,line width=0.800pt] (219,82) .. controls
    (224,135) and (402,151) .. (400,193);

  \path[draw=black,dash pattern=on 1pt off 2pt,line join=miter,line
    cap=butt,miter limit=4.00,line width=0.800pt] (401,300) .. controls
    (401.5,235) and (218,232) .. (219,193);

  \path[draw=black,line join=miter,line cap=butt,line width=0.800pt,flecha2=.7]
    (401,82) .. controls (396,157) and (220,143) ..
    (219,193);

  \path[draw=black,line join=miter,line cap=butt,line width=0.800pt,flecha2=.8]
    (411,300) -- (411,82);

  \node at (422,106.17125) {$\phi_1$};

  \node at (364.66507,230.09802)
    {$\gamma_1$};

  \node at (300,350)  {$\mu_1$};

  \path[flecha1=.45,draw=black,dash pattern=on 2.40pt off 0.80pt,line join=miter,line cap=butt,miter limit=4.00,line
    width=1.360pt] (401,193) .. controls (399,172) and
    (356,163) .. (316,164) .. controls (268,163) and
    (218,174) .. (219,193);
\path[flecha1=.5,draw=black,line join=miter,line cap=butt,miter limit=4.00,line
    width=1.360pt](435,164)-- (401,193);
\path[flecha1=.5,draw=black,line join=miter,line cap=butt,miter limit=4.00,line
    width=1.360pt](219,193) -- (172,225);

%

\node at  (300,177)    {$\lambda_2^c$};

\node at (183,230)    {$\lambda_2^b$};

  \path[draw=blue,dash pattern=on 3pt off 5.5pt,line join=miter,line
    cap=butt,miter limit=4.00,line width=1.360pt] (219,192) .. controls
    (232,156) and (220,111) .. (226,112) .. controls
    (231,112) and (252,137) .. (273,143) .. controls
    (294,150) and (402,167) .. (401,188);

\node at (423,190)    {$\lambda_2^e$};

\node at (199,125)     {$\lambda_1^c$};

  \path[flecha2=.05,flecha2=.55,flecha2=.85,draw=black,line join=miter,line cap=butt,miter limit=4.00,line
    width=1.360pt] (161,56) -- (210,79) -- (210,187)
    -- (161,164);

\node at (183,82)     {$\lambda_1^b$};

\node at  (183,160) {$\lambda_1^e$};

\end{tikzpicture}
\caption{Gluing of $M_2^{'i}$, $e_{i,0}=1$.}
\label{fig:cilindroplumbing2}
\end{subfigure}
\caption{}
\end{figure}

Let us decompose $\gamma_1=\lambda_1^\gamma\cdot\lambda_1^{'\gamma}$ in two halves
where $\lambda_1^\gamma$ is the bottom part in  Figure~\ref{fig:cilindroplumbing}. If $e_{i,0}=1$,
we can check that
$\lambda_2^c$ can be isotoped inside $A_1^1$ to $(\lambda_1^c)^{-1}$ followed by
$(\lambda_1^\gamma)^{-1}$,
see Figure~\ref{fig:cilindroplumbing}. That means that the new cutting curve~$\bar{\lambda}_1$ has
two connected components near $A_1^1\equiv A_1^2$; one is $\lambda_2^b\cdot\lambda_1^e$, and
the other one is $\lambda_1^b\cdot(\lambda_1^\gamma)^{-1}\cdot\lambda_2^e$.

We perform a similar argument for the gluing of $M_1^{'2}$ and $M_2^{'2}$.
In this case we consider the other annuli $A_2^1\subset M_1^{'2}$
and $A_2^2\subset M_2^{'2}$ which become identified; they are the other
parts of the plumbing tori.
Let us choose cutting curves $\lambda'_1,\lambda'_2$ which go parallel
near the annuli to $\lambda_1,\lambda_2$; in order to emphasize it, we keep the above notation
for their decomposition in the neighborhood of the annuli,
see Figure~\ref{fig:cilindroplumbing2}. Assuming again $e_{i,0}=1$, we see that
$\lambda_1^c$ can be isotoped inside $A_2^1$ to $\lambda_2^c$ followed by
$\lambda_1^{'\gamma}$; note that the isotopy is done in the back part of $A_2^1$
if Figure~\ref{fig:cilindroplumbing2}.
The new cutting curve~$\bar{\lambda}_2$ has
two connected components near $A_2^1\equiv A_2^2$; one is $\lambda_2^b\cdot\lambda_1^e$, as before, and
the other one is $\lambda_1^b\cdot\lambda_1^{'\gamma}\cdot\lambda_2^e$.


As we see in Figures~\ref{fig:cilindroplumbing} and~\ref{fig:cilindroplumbing2},
some of the ends do not fit; in order for them to fit we have to do a half-turn around
$\gamma_1$ in the suitable direction. Since we have freedom
to choose the product structure in the annulus, this is equivalent
to keep the intersection of the red curves as fibers, while
the intersection of the blue curves perform  a
full loop. To be precise, since
the curve $\bar{\lambda}_2\cdot(\bar{\lambda}_1)^{-1}$ equals $\gamma_1$
near the plumbing (in homology), for $e_{i,0}=1$ the curve $\lambda_2$ turns as $\gamma_1$
(when going from the first vertex to the second one), see
Figure~\ref{fig:plumbingtwist}. It is easily seen that it turns as $\gamma_1^{-1}$ for
$e_{i,0}=-1$.

\begin{figure}[ht]
\centering
\begin{tikzpicture}[ inner sep=0pt, outer sep=0pt]

\tikzset{flecha1/.style={decoration={
  markings,
  mark=at position #1 with  {\arrowreversed{angle 90
new}}},postaction={decorate}}}
\tikzset{flecha2/.style={decoration={
  markings,
  mark=at position #1 with  {\arrow{angle 90 new}}},postaction={decorate}}}
\def\nextAngle{0}
\tikzset{
    next angle/.style={
        in=#1+180,
        out=\nextAngle,
        prefix after command= {\pgfextra{\def\nextAngle{#1}}}
    },
    start angle/.style={
        out=#1,
        nangle=#1,
    },
    nangle/.code={
        \def\nextAngle{#1}
    }
}
\def\recta{\draw[dashed] (-3,-1) -- (-2,-1);
\draw[very thick] (-2,-1) -- (2,-1);
\draw[dashed] (2,-1) -- (3,-1)}
\recta;
\begin{scope}[yshift=2cm]
\recta;
\end{scope}
\draw[dashed] (0,1) arc (90:-90:.5 and 1);
\draw[very thick,flecha1=.5] (0,1) arc (90:270:.5 and 1);
\begin{scope}[yshift=1.5cm]
\draw[dashed,red] (-3,-1) -- (-2,-1);
\draw[red,flecha2=.2] (-2,-1) -- (2,-1);
\draw[dashed,red] (2,-1) -- (3,-1);
\end{scope}
\begin{scope}[yshift=.5cm]
\draw[dashed,blue] (-3,-1) -- (-2,-1);
\draw[blue,flecha2=.7] (-2,-1) -- (-.8,-1);
\draw[blue] (.8,-1) -- (2,-1);
\draw[dashed,blue] (2,-1) -- (3,-1);
\end{scope}
\draw [blue] (-.8,-.5) to[start angle=0,next angle=80] (-.7,0)
to[next angle=70] (-.6,.5) to [next angle=0]  (-.2,1);
\draw [dashed,blue] (-.2,1) to[start angle=0,next angle=-80]  (0,.5) to [next angle=-90] (.1,0);
\draw [dashed,blue] (.1,0) to[start angle=-90,next angle=-80]  (.2,-.5) to [next angle=0] (.5,-1);
\draw [blue] (.4,-1) to[start angle=0,next angle=0] (.8,-.5);
\node at (-3.3,.5) {$\bar{\lambda}_1$};
\node at (-3.3,-.5) {$\bar{\lambda}_2$};
\node at (-.2,0) {$\gamma_1$};
\end{tikzpicture}
\caption{Cutting curves for $e_{i,0}=1$.}
\label{fig:plumbingtwist}
\end{figure}
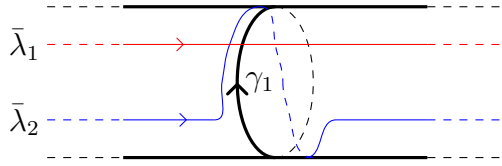

\begin{ejm}\label{ejm:e22}
Figure~\ref{fig:diagramaplumbing} illustrates the case of two vertices with genus zero and both
with Euler number~$-2$. Note that we take $n_1=n_2=1$ and $e_{i,j}=-1$.
\end{ejm}

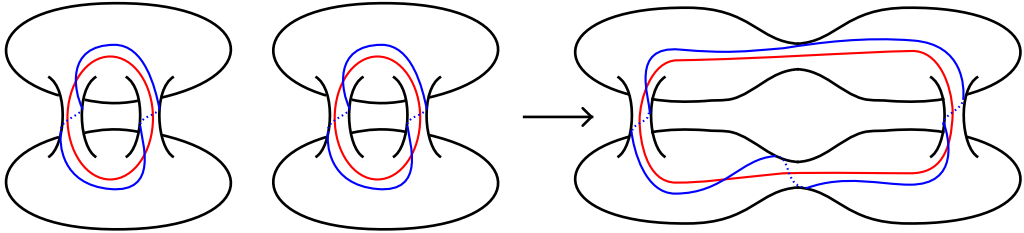
\begin{figure}[ht]

\definecolor{cff0000}{RGB}{255,0,0}
\definecolor{c0000ff}{RGB}{0,0,255}

\begin{tikzpicture}[y=0.80pt, x=0.8pt,yscale=0.5, inner sep=0pt, outer sep=0pt]
\tikzset{flecha1/.style={decoration={
  markings,
  mark=at position #1 with  {\arrowreversed{angle 90
new}}},postaction={decorate}}}
\tikzset{flecha2/.style={decoration={
  markings,
  mark=at position #1 with  {\arrow{angle 90 new}}},postaction={decorate}}}
\def\vertical{\path[draw=black,line join=miter,line cap=butt,line width=1pt] (220,117) .. controls (228,126) and (230,181) .. (220,193)}
\def\tubo{\vertical;\begin{scope}[xscale=-1,xshift=-370pt]
\vertical;
\end{scope}}
\def\bola{\path[draw=black,line join=miter,line cap=butt,line width=1pt] (262,141) .. controls (252,145) and (244,143) .. (237,140);
\path[draw=black,line join=miter,line cap=butt,line width=1pt] (225,136) .. controls (185,112) and (193,48) .. (253,49) .. controls (313,49) and (324,114) .. (273,138);
}
\def\gii{\tubo
\begin{scope}[xshift=29pt]
\tubo;
\end{scope}
\bola
\begin{scope}[yscale=-1,yshift=-249pt]
\bola;
\end{scope}
}
\def\corte{\draw[line width=.8pt,color=cff0000] (249,154) ellipse (20 and 58);
\path[draw=c0000ff,line join=miter,line cap=butt,line width=0.8pt] (263,147) .. controls (264,131) and (272,83) .. (249,87) .. controls (226,91) and (225,130) .. (226,142);
\path[draw=c0000ff,dash pattern=on 0.60pt off 1.20pt,line join=miter,line cap=butt,miter limit=4.00,line width=0.8pt] (272,163) .. controls (271,156) and (263,156) .. (263,147);
\path[draw=c0000ff,dash pattern=on 0.60pt off 1.20pt,line join=miter,line cap=butt,miter limit=4.00,line width=0.8pt] (236,163) .. controls (236,157) and (225,152) .. (226,142);
\path[draw=c0000ff,line join=miter,line cap=butt,line width=0.8pt] (236,163) .. controls (232,179) and (228,223) .. (251,223) .. controls (265,223) and (270,183) .. (272,163);
}
\gii
\corte
\begin{scope}[xshift=100pt]
\gii
\corte
\end{scope}

\path[draw=black,line join=miter,line cap=butt,line width=1pt,flecha2=1] (442,155) -- (475,155);

\begin{scope}[xshift=213pt]
\tubo;
\end{scope}
\begin{scope}[xshift=330pt]
\tubo;
\end{scope}

\def\trozo1{\path[draw=black,line join=miter,line cap=butt,line width=1pt] (492,176) .. controls (451,198) and (459,259) .. (519,258) .. controls (549,258) and (552,226) .. (571,224)(502,172) .. controls(510,169) and (519,169) .. (527,170).. controls (531,171) and (536,170).. (540,172).. controls (551,178) and (558,199) .. (571,200) ;}

\trozo1
\begin{scope}[xscale=-1,xshift=-913]
\trozo1
\end{scope}

\begin{scope}[yscale=-1,yshift=-250]
\trozo1
\begin{scope}[xscale=-1,xshift=-913]
\trozo1
\end{scope}
\end{scope}

\begin{scope}[xscale=-1,xshift=-912,yscale=-1,yshift=-252]
\path[draw=cff0000,line join=miter,line cap=butt,miter limit=4.00,line width=0.8pt] (516.0734,97.7900) .. controls (503.5573,97.2606) and (497.4582,129.0761) .. (496.9107,155.4737) .. controls (497.4582,193.6036) and (502.9333,213.1574) .. (516.6209,213.1574) .. controls (530.3086,213.1574) and (554.3989,212.1798) .. (576.1623,213.1574) .. controls (605.5907,219.0236) and (612.7083,221.9567) .. (626.3960,221.9567) .. controls (640.0836,221.9567) and (643.3687,182.8490) .. (643.3687,163.2952) .. controls (643.3687,143.7414) and (637.9359,106.9730) .. (625.8485,106.5892) .. controls (585.9721,105.3230) and (543.9208,98.9679) .. (516.0734,97.7900) -- cycle;
\path[draw=c0000ff,line join=miter,line cap=butt,line width=0.8pt] (501.2908,164.2729) .. controls (498.5532,179.9160) and (493.6257,223.9120) .. (516.8947,223.9120) .. controls (530.8561,223.9120) and (547.2813,214.1351) .. (565,227);
\path[draw=c0000ff,dash pattern=on 0.60pt off 1.20pt,line join=miter,line cap=butt,miter limit=4.00,line width=0.8pt] (501.2908,167.2060) .. controls (502.3858,158.4068) and (491.9831,153.5183) .. (491.9831,142.7637);
\path[draw=c0000ff,line join=miter,line cap=butt,line width=0.8pt] (638.4976,157.0533) .. controls (639.0451,140.4325) and (647.8053,92.5257) .. (624.8100,96.4365) .. controls (601.8147,100.3473) and (595.2446,98.3919) .. (575.9168,95.1579) .. controls (559.8740,89.9685) and (537.9737,84.1023) .. (514.9784,88.0131) .. controls (491.9832,91.9239) and (491.4356,131.0315) .. (491.9832,143.7414);
\path[draw=c0000ff,line join=miter,line cap=butt,line width=0.8pt] (580,197) .. controls (596.3396,199.0940) and (603.4572,232.3354) .. (626.7262,232.3354) .. controls (640.6877,232.3354) and (646.1627,192.2501) .. (647.2577,172.6963);
\path[draw=c0000ff,dash pattern=on 0.60pt off 1.20pt,line join=miter,line cap=butt,miter limit=4.00,line width=0.8pt] (646.7102,172.6963) .. controls (646.7102,165.8525) and (639.0451,165.8525) .. (639.0451,156.0756);
\path[draw=c0000ff,dash pattern=on 0.60pt off 1.20pt,line join=miter,line cap=butt,miter limit=4.00,line width=0.8pt] (565,227) .. controls (576.5664,225.5479) and (571.6388,197.1949) .. (580,197);
\end{scope}

\end{tikzpicture}
\caption{Heegaard diagram of the plumbing of two
manifolds with $g=0,e=-2$}
\label{fig:diagramaplumbing}
\end{figure}

\section{Heegaard splittings of arbitrary plumbed graphs}\label{sec:arbitrary}

In this section, we consider an arbitrary plumbing graph
$(\Gamma,g,e,o)$; for the plumbing construction we fix
an explicit cocycle representing~$o$, consisting
on assigning a sign $e_\eta$ to each edge~$\eta$.

Fix a vertex~$v$ with valency $d_v$; this vertex is associated
with a fibration $\pi_v:M_v\to S_v$; we choose $d_v+n_v$ pairwise
disjoint closed disks in $S_v$, determining solid tori in $M_v$.
The first $d_v$ disks are assigned to a fixed edge $\eta$ having
$v$ as an endpoint. As in \S\ref{sec:2vertices}, the first $d_v$ disks
will have associated numbers $e_{\eta}$, and the remaining disks
numbers $e_{v,j}$, $j=1,\dots,n_v$, such
that their absolute value equals~$1$, and
\[
\sum_{v\in\partial\eta} e_{\eta}+\sum_{j=1}^{n_v} e_{v,j}=e_v.
\]
In general one of the extra disks will correspond to the main cylinder,
hence $n_v\geq 1$. The only exception to this rule is the case
$g_v=0$, $d_v=2$, since in this case the main cylinder plays no special role.

If $\Gamma$ is a tree it is enough to iterate the construction
of \S\ref{sec:2vertices}. Note also that there is no restriction
for the choice of the cocyle.



Let us consider now the general case where the graph may have cycles.
We start by the choice of a cocycle and  a spanning tree, for which
we proceed as above. Let us now explain the effect of plumbing
along the remaining edges.


As we saw in Proposition~\ref{float-gluing-una}, the process is different when the plumbing
closes a cycle in the graph, since in that case the gluing process is done between two float
curves of the same handlebody; specially, the way of constructing cutting curve systems changes.
Proposition~\ref{float-gluing-una} proves that this process produces also
a Heegaard splitting (where the genus remains unchanged).

\begin{figure}[ht]

\definecolor{cff0000}{RGB}{255,0,0}

\begin{tikzpicture}[scale=0.5]

\def\geno{

    \draw[very thick] (2.25+1,1) to [out=10,in=270] (3+1,2.0)
    to [out=90, in=0] (0+1,4.0)
    to [out=180, in=90] (-3-1, 2)
    to [out=270, in=170] (-2.25-1,1) to[out=-10,in=190] (2.25+1,1);}
\def\asa{
    \fill[white]
    (1.5,3.0)
    to [out=150, in=270] (1.0, 4.0)
    to [out=90, in=0.0] (0.0, 5.0)
    to [out=180, in=90] (-1.0,4.0)
    to [out=270, in=30] (-1.5,3.0)--
    (-0.25,3.0)
    to [out=150, in=270] (-0.5,4.0)
    to [out=90, in=180] (0,4.5)
    to [out=0, in=90] (0.5,4.0)
    to [out=270, in=30] (0.25, 3.0);

    \draw[very thick] (1.5,3.0)
    to [out=150, in=270] (1.0, 4.0)
    to [out=90, in=0.0] (0.0, 5.0)
    to [out=180, in=90] (-1.0,4.0)
    to [out=270, in=30] (-1.5,3.0);

    \draw[very thick] (0.25,3.0)
    to [out=30, in=270] (0.5,4.0)
    to [out=90, in=0] (0,4.5)
    to [out=180, in=90] (-0.5,4.0)
    to [out=270, in=150] (-0.25, 3.0);
    }

\def\tubo{
    \fill[white] (-0.75,1.25) -- (-0.5,1)
    to [out=300, in=90] (-0.3,0)
    to [out=270, in = 60] (-0.5,-1)
    -- (-0.75,-1.25)-- (0.75,-1.25) -- (0.5,-1)
    to [out=-240, in=-90] (0.3,0)
    to [out=-270, in = -120] (0.5,1)
    -- (0.75,1.25)-- (-0.75,1.25);

    \draw[ very thick] (-0.75,1.25) -- (-0.5,1)
    to [out=300, in=90] (-0.3,0)
    to [out=270, in = 60] (-0.5,-1)
    -- (-0.75,-1.25);
    \draw[very thick] (0.75,1.25) -- (0.5,1)
    to [out=240, in=90] (0.3,0)
    to [out=270, in = 120] (0.5,-1)
    -- (0.75,-1.25);
    }

\def\tubogordo{
    \fill[white] (-1.75,1.25) -- (-1.5,1)
    to [out=300, in=90] (-1.3,0)
    to [out=270, in = 60] (-1.5,-1)
    -- (-1.75,-1.25)-- (1.75,-1.25) -- (1.5,-1)
    to [out=-240, in=-90] (1.3,0)
    to [out=-270, in = -120] (1.5,1)
    -- (1.75,1.25)-- (-1.75,1.25);

    \draw[ very thick] (-1.75,1.25) -- (-1.5,1)
    to [out=300, in=90] (-1.3,0)
    to [out=270, in = 60] (-1.5,-1)
    -- (-1.75,-1.25);
    \draw[very thick] (1.75,1.25) -- (1.5,1)
    to [out=240, in=90] (1.3,0)
    to [out=270, in = 120] (1.5,-1)
    -- (1.75,-1.25);
    }

\begin{scope}[xshift=-5cm]
	\geno

	\begin{scope}[yscale=-1]
	\geno
	\end{scope}

	\begin{scope}[xshift=-1cm,xscale=1.25]
	\tubo
	\end{scope}

	\begin{scope}[xshift=1cm]
	\tubo
	\end{scope}

	 \begin{scope}[rotate=90,yshift=4.5cm,xshift=2cm]
	\tubo
	\end{scope}
	 \begin{scope}[rotate=90,yshift=4.5cm,xshift=-2cm]
	\tubo
	\end{scope}

\end{scope}

\begin{scope}[xshift=5cm]
	\geno

	\begin{scope}[yscale=-1]
	\geno
	\end{scope}

	\begin{scope}[xshift=-1cm]
	\tubo
	\end{scope}

	\begin{scope}[xshift=1cm,xscale=1.25]
	\tubo
	\end{scope}

	 \begin{scope}[rotate=90,yshift=-4.5cm,xshift=2cm]
	\tubo
	\end{scope}
	 \begin{scope}[rotate=90,yshift=-4.5cm,xshift=-2cm]
	\tubo
	\end{scope}
\end{scope}

\def\curvaroja{
\draw[red] (0.0, 1.6)
to [out=180, in=90] (-0.8, 0.0)
to [out=270, in=180] (0.0, -1.6)
to [out=0, in=270] (0.8, 0.0)
to [out=90, in=0] (0.0, 1.6);
}

\begin{scope}[xshift=-5cm]
\curvaroja
\end{scope}

\begin{scope}[xshift=5cm]
\curvaroja
\end{scope}

\begin{scope}[xshift=-5cm]
\draw[blue] (0.65, 0.2)
to [out=90, in=270] (1.0, 0.5)
to [out=90, in=0] (0.0, 1.8)
to [out=180, in=90] (-1.0, 0.5)
to [out=270, in=90] (-1.4, 0.2);
\draw[blue, dotted] (-1.4, 0.2)
to [out=270, in=90] (-0.7, -0.2);
\draw[blue]  (-0.7,-0.2)
to [out=270, in=90] (-1.0, -0.5)
to [out=270, in=180] (0.0, -1.8)
to [out=0, in=270] (1.0, -0.5)
to [out=90, in=270] (1.25, -0.2);
\draw[blue, dotted] (1.25, -0.2)
to [out=90, in=270] (0.65, 0.2);
\end{scope}

\begin{scope}[xshift=5cm]
\draw[blue] (0.65, 0.2)
to [out=90, in=270] (1.0, 0.5)
to [out=90, in=0] (0.0, 1.8)
to [out=180, in=90] (-1.0, 0.5)
to [out=270, in=90] (-1.3, 0.2);
\draw[blue, dotted] (-1.3, 0.2)
to [out=270, in=90] (-0.7, -0.2);
\draw[blue]  (-0.7,-0.2)
to [out=270, in=90] (-1.0, -0.5)
to [out=270, in=180] (0.0, -1.8)
to [out=0, in=270] (1.0, -0.5)
to [out=90, in=270] (1.35, -0.2);
\draw[blue, dotted] (1.35, -0.2)
to [out=90, in=270] (0.65, 0.2);
\end{scope}

\node[] at (-12,2) {$\cdots$};
\node[] at (-12,-2) {$\cdots$};
\node[] at (12,2) {$\cdots$};
\node[] at (12,-2) {$\cdots$};

\draw[->] (0,-4) -- (0,-6);

\begin{scope}[yshift=-10cm,xshift=-5cm]
	\geno

	\begin{scope}[yscale=-1]
	\geno
	\end{scope}

	\begin{scope}[xshift=-1cm,xscale=1.25]
	\tubo
	\end{scope}

	 \begin{scope}[rotate=90,yshift=4.5cm,xshift=2cm]
	\tubo
	\end{scope}
	 \begin{scope}[rotate=90,yshift=4.5cm,xshift=-2cm]
	\tubo
	\end{scope}

\end{scope}

\begin{scope}[yshift=-10cm,xshift=5cm]
	\geno

	\begin{scope}[yscale=-1]
	\geno
	\end{scope}

	\begin{scope}[xshift=1cm,xscale=1.25]
	\tubo
	\end{scope}

	 \begin{scope}[rotate=90,yshift=-4.5cm,xshift=2cm]
	\tubo
	\end{scope}
	 \begin{scope}[rotate=90,yshift=-4.5cm,xshift=-2cm]
	\tubo
	\end{scope}

	\begin{scope}[rotate=90,xshift=2cm,yshift=5cm]
	\tubo
	\end{scope}
	\begin{scope}[rotate=90,xshift=-2cm,yshift=5cm]
	\tubo
	\end{scope}

\end{scope}

\node[] at (-12,-8) {$\cdots$};
\node[] at (-12,-12) {$\cdots$};
\node[] at (12,-8) {$\cdots$};
\node[] at (12,-12) {$\cdots$};

\draw[red] (-5.9,-10) to[out=90,in=180] (-5,-8.1)
--(4.8,-8.1) to[out=0,in=90] (5.8,-10)
to[out=270,in=0] (4.8,-11.9)
--(-5,-11.9) to [out=180,in=270] (-5.9,-10);
\draw[red] (-1,-12.5)
to[out=170,in=0] (-5,-12)
to[out=180,in=270] (-6,-10)
to[out=90,in=180] (-5,-8)
to[out=0,in=190] (-1,-7.5);
\draw[red,dotted] (-1,-7.5)
to[out=300,in=30] (-1,-8.5);
\draw[red] (-1,-8.5)
to[out=170,in=0] (-5,-8.2)
to[out=180,in=90] (-5.8,-10)
to[out=270,in=180] (-5,-11.8)
to[out=0,in=190] (-1,-11.5);
\draw[red,dotted] (-1,-11.5)
to[out=300,in=30] (-1,-12.5);

\draw[blue] (-6.4,-10)
to[out=90, in=270] (-6.2,-9.6)
to[out=90,in=180] (-5,-7.9)
to[out=0,in=180] (5,-7.9)
to[out=0,in=90] (6.1,-9.5)
to[out=270,in=90](5.6,-10);
\draw[blue,dotted] (5.6,-10)
to[out=270,in=90] (6.4,-10.5);
\draw[blue] (6.4,-10.5)
to[out=270,in=90] (6.1,-11)
to[out=270,in=0] (5,-12.1)
to [out=180,in=0] (1,-12.1)
to[out=180,in=-30] (0,-11.7);
\draw[blue,dotted] (0,-11.7)
to[out=210,in=30] (-0.5,-12.3);
\draw[blue] (-0.55,-12.3)
to[out=120,in=0] (-2,-12.1)
to[out=180,in=0] (-5,-12.1)
to[out=180,in=270] (-6.1,-11)
to[out=90,in=270] (-5.6,-10.5);
\draw[blue,dotted] (-5.6,-10.5)
to[out=90,in=270] (-6.4,-10);

\draw[blue] (5.6,-10.1)
to[out=90,in=270] (6.2,-9.5)
to[out=90,in=0] (5,-7.8)
to[out=180,in=350] (1,-7.5);
\draw[blue,dotted] (1,-7.5)
to[out=240,in=120] (1,-8.5);
\draw[blue] (1,-8.5)
to[out=10,in=180] (5,-8)
to[out=0,in=90] (6,-9.5)
to[out=270,in=90] (5.6,-9.9);
\draw[blue,dotted] (5.6,-10.1)
to[out=270,in=90] (6.4,-10.6);
\draw[blue,dotted] (5.6,-9.9)
to[out=270,in=90] (6.4,-10.4);
\draw[blue] (6.4,-10.6)
to[out=270,in=90] (6.2,-11)
to[out=270,in=0] (5,-12.2)
to[out=180,in=10] (1,-12.5);
\draw[blue,dotted] (1,-12.5)
to[out=120,in=240] (1,-11.5);
\draw[blue] (1,-11.5)
to[out=-10,in=180] (5,-12)
to[out=0,in=270] (6,-11)
to[out=90,in=270] (6.4,-10.4);

\end{tikzpicture}

\caption{Float gluing that closes a cycle}
\label{fig:curvasciclo}
\end{figure}
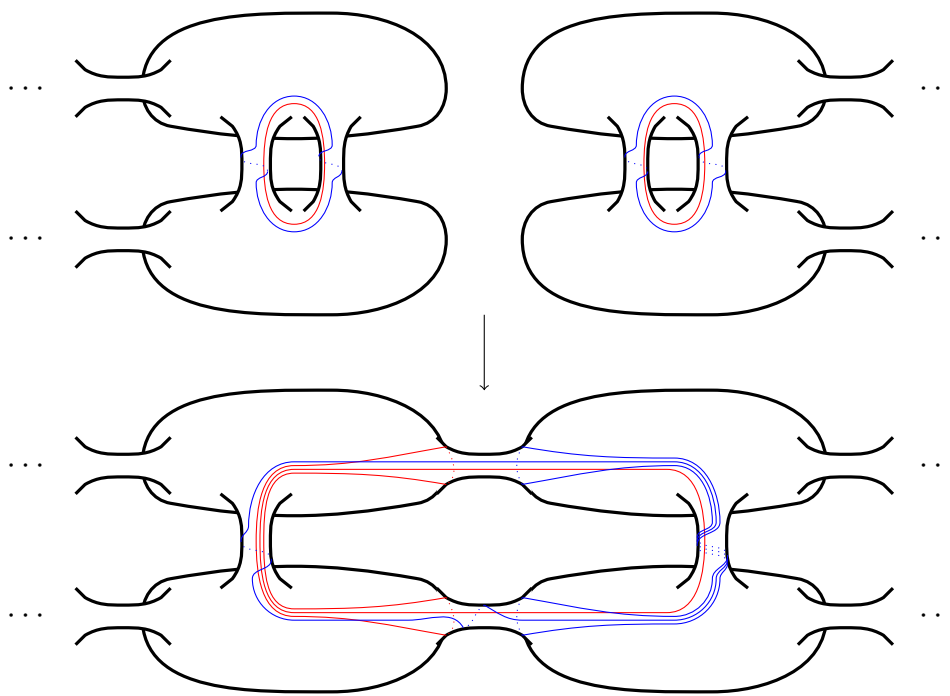

How to obtain the cutting curves
is explained in Remark~\ref{obs:cutting-ciclo}. Figure~\ref{fig:curvasciclo} describes
this process in our case, showing how to obtain the new pair of cutting
curves from the ones that existed before the plumbing.
The first pair of cutting (red and blue) curves is obtained as in the
tree case: they are obtained as connected sum of the preexistent ones.
The second pair of cutting curves is constructed as explained
in Remark~\ref{obs:cutting-ciclo}, as the union of two parallel
copies of a preexistent curve and the boundaries of the identified annuli.

\section{Explicit examples}\label{sec:ejm}

Let us consider some examples of graph manifolds for which we will give a Heegaard splitting. These examples come from links of normal surface singularities.

\begin{ejm}
Let $M$ be the link of the $\mathbb{A}_n$ singularity, which
is a lens space~$L(n,n-1)$. The graph of this manifold is a linear tree with $n-1$
vertices with $([0],-2)$ decorations.

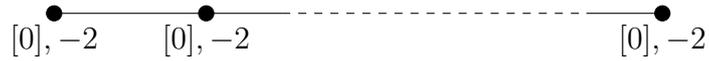
\begin{figure}[ht]
\begin{center}
 \begin{tikzpicture}[scale=2,
vertice/.style={draw,circle,fill,minimum size=0.2cm,inner
sep=0}
]
\draw (-2,0) --(-.5,0) ;
\node[vertice] at (-2,0) {};
\node[vertice] at (-1,0) {};
\node[vertice] at (2,0) {};
\draw[dashed] (-.5,0) --(1.5,0) ;
\draw (1.5,0) --(2,0) ;
\node[below] at (-2,0) {$[0],-2$};
\node[below] at (-1,0) {$[0],-2$};
\node[below] at (2,0) {$[0],-2$};
\end{tikzpicture}
\caption{$\mathbb{A}_n$ graph}
\label{fig:an}
\end{center}
\end{figure}

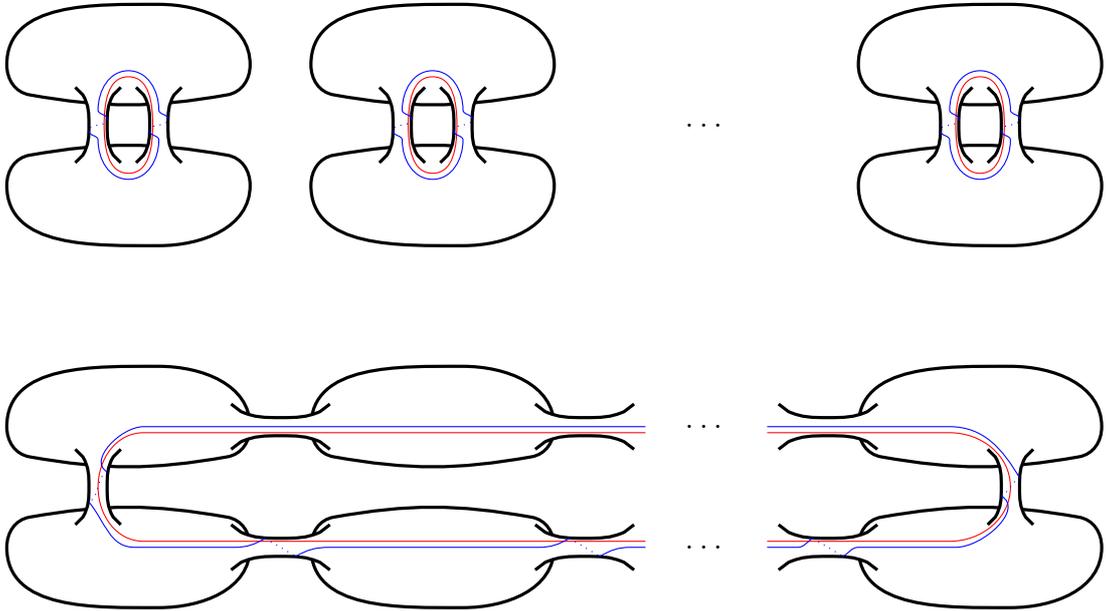
\begin{figure}[ht]
\begin{center}
\begin{tikzpicture}[scale=0.4]

\def\geno{

    \draw[very thick] (2.25+1,1) to [out=10,in=270] (3+1,2.0)
    to [out=90, in=0] (0+1,4.0)
    to [out=180, in=90] (-3-1, 2)
    to [out=270, in=170] (-2.25-1,1) to[out=-10,in=190] (2.25+1,1);}

\def\tubo{
    \fill[white] (-0.75,1.25) -- (-0.5,1)
    to [out=300, in=90] (-0.3,0)
    to [out=270, in = 60] (-0.5,-1)
    -- (-0.75,-1.25)-- (0.75,-1.25) -- (0.5,-1)
    to [out=-240, in=-90] (0.3,0)
    to [out=-270, in = -120] (0.5,1)
    -- (0.75,1.25)-- (-0.75,1.25);

    \draw[ very thick] (-0.75,1.25) -- (-0.5,1)
    to [out=300, in=90] (-0.3,0)
    to [out=270, in = 60] (-0.5,-1)
    -- (-0.75,-1.25);
    \draw[very thick] (0.75,1.25) -- (0.5,1)
    to [out=240, in=90] (0.3,0)
    to [out=270, in = 120] (0.5,-1)
    -- (0.75,-1.25);
    }

\def\curvaroja{
\draw[red] (0.0, 1.6)
to [out=180, in=90] (-0.8, 0.0)
to [out=270, in=180] (0.0, -1.6)
to [out=0, in=270] (0.8, 0.0)
to [out=90, in=0] (0.0, 1.6);
}

\def\curvaazul{
\draw[blue] (1.3, 0.2)
to [out=90, in=270] (1.0, 0.5)
to [out=90, in=0] (0.0, 1.8)
to [out=180, in=90] (-1.0, 0.5)
to [out=270, in=90] (-0.7, 0.2);
\draw[blue, dotted] (-0.7, 0.2)
to [out=270, in=90] (-1.3, -0.2);
\draw[blue]  (-1.3,-0.2)
to [out=270, in=90] (-1.0, -0.5)
to [out=270, in=180] (0.0, -1.8)
to [out=0, in=270] (1.0, -0.5)
to [out=90, in=270] (0.7, -0.2);
\draw[blue, dotted] (0.7, -0.2)
to [out=90, in=270] (1.3, 0.2);
}

\geno
\begin{scope}[yscale=-1]
\geno
\end{scope}
\begin{scope}[xshift=-1cm]
\tubo
\end{scope}
\begin{scope}[xshift=1cm]
\tubo
\end{scope}
\curvaroja
\curvaazul

\begin{scope}[xshift=-10cm]
\geno
\begin{scope}[yscale=-1]
\geno
\end{scope}
\begin{scope}[xshift=-1cm]
\tubo
\end{scope}
\begin{scope}[xshift=1cm]
\tubo
\end{scope}
\curvaroja
\curvaazul
\end{scope}

\node (puntos) at (9cm,0) {$\dots$};

\begin{scope}[xshift=18cm]
\geno
\begin{scope}[yscale=-1]
\geno
\end{scope}
\begin{scope}[xshift=-1cm]
\tubo
\end{scope}
\begin{scope}[xshift=1cm]
\tubo
\end{scope}
\curvaroja
\curvaazul
\end{scope}

\begin{scope}[xshift=-10cm, yshift=-12cm]
\geno
\begin{scope}[yscale=-1]
\geno
\end{scope}
\begin{scope}[xshift=-1cm]
\tubo
\end{scope}
\end{scope}
\begin{scope}[ yshift=-12cm]
\geno
\begin{scope}[yscale=-1]
\geno
\end{scope}
\end{scope}
\begin{scope}[xshift=18cm, yshift=-12cm]
\geno
\begin{scope}[yscale=-1]
\geno
\end{scope}
\begin{scope}[xshift=1cm]
\tubo
\end{scope}
\end{scope}

\begin{scope}[rotate=90,xshift=-10cm, yshift=5cm,yscale=1.3]
\tubo
\end{scope}
\begin{scope}[rotate=90,xshift=-14cm, yshift=5cm,yscale=1.3]
\tubo
\end{scope}
\begin{scope}[rotate=90,xshift=-10cm, yshift=-5cm,yscale=1.3]
\tubo
\end{scope}
\begin{scope}[rotate=90,xshift=-14cm, yshift=-5cm,yscale=1.3]
\tubo
\end{scope}
\begin{scope}[rotate=90,xshift=-10cm, yshift=-13cm,yscale=1.3]
\tubo
\end{scope}
\begin{scope}[rotate=90,xshift=-14cm, yshift=-13cm,yscale=1.3]
\tubo
\end{scope}

\node (puntos) at (9cm,-10) {$\dots$};
\node (puntos) at (9cm,-14) {$\dots$};

\draw[red]
(7,-10.2) --(-9.5,-10.2)
to [out=180, in=90] (-11,-12)
to [out=270, in=180] (-9.5,-13.8)
--(7,-13.8);
\draw[red]
(11,-10.2)  -- (17, -10.2)
to [out=0, in=90] (19,-12)
to [out=270, in=0] (17,-13.8)
--(11,-13.8);

\draw[blue]
(7,-10) -- (-9.5,-10)
to [out=180, in=150] (-10.7,-11.5);
\draw[blue, dotted]
(-10.7, -11.5)
to [out=210, in=30] (-11.3, -12.5);
\draw[blue]
(-11.3, -12.5)
to [out=310, in=180] (-9.8, -14)
-- (-6.5, -14)
to [out=0, in=210] (-5.5,-13.7) ;
\draw[blue, dotted]
(-5.5, -13.7) to [out=330, in=150] (-4.5, -14.3);
\draw[blue]
(-4.5, -14.3)
to [out=30, in=180] (-3.4, -14) -- (3.5, -14)
to [out=0, in= 210] (4.5, -13.7);
\draw[blue,dotted]
(4.5, -13.7) to [out=330, in=150] (5.5, -14.3);
\draw[blue]
(5.5,-14.3) to [out=30, in=180] (6.5, -14) --(7, -14);

\draw[blue]
(11,-14) -- (12, -14)
to[out=0, in=210] (12.5, -13.7);
\draw[blue, dotted]
(12.5, -13.7) to [out=330, in=150] (13.5, -14.3);

\draw[blue] (13.5, -14.3) to [out=30, in=180]
(14, -14) -- (17.2, -14)
to [out=0, in=330] (18.7, -12.3);
\draw[blue, dotted]
(18.7, -12.3) to [out=30, in=210] (19.3, -11.7);
\draw[blue] (19.3, -11.7)
to [out=120, in=0] (17, -10)--(11, -10);

\end{tikzpicture}

\caption{Heegard diagram of the $\mathbb{A}_n$ graph.}
\end{center}

\end{figure}

With our method we obtain a genus~$1$ Heegaard splitting where
the two curves intersect~$n$ times.
\end{ejm}

From now we will drop the genus weight if it vanishes.

\begin{ejm}
Let us consider the plumbing manifold associated with a graph
with one vertex and Euler number $-n$, link of a quotient singularity,
i.e. the lens space $L(n,1)$. With our method we obtain a
Heegaard splitting of genus~$n-1$. Using Neumann plumbing calculus
(namely $(n-1)$ $+1$-blow-ups and one $-1$-blow-down), we can transform it in the graph of Figure~\ref{fig:an}, where the weights equal~$2$. The
Heegaard splitting coincides with the one from the previous example,
with a reversed orientation.
\end{ejm}

\begin{ejm}
The plumbing manifold of Figure~\ref{fig:L51} is also a lens space~$L(5,2)$
and it admits a Heegaard splitting of genus~$1$. However, our method
provides a genus-$2$ Heegaard splitting.

 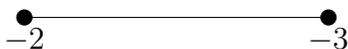
\begin{figure}[ht]
\begin{center}
 \begin{tikzpicture}[scale=2,
vertice/.style={draw,circle,fill,minimum size=0.2cm,inner
sep=0}
]
\draw (-2,0) --(0,0) ;
\node[vertice] at (-2,0) {};
\node[vertice] at (0,0) {};

\node[below] at (-2,0) {$-2$};

\node[below] at (0,0) {$-3$};

\end{tikzpicture}
\caption{A quotient singularity.}
\label{fig:L51}
\end{center}
\end{figure}

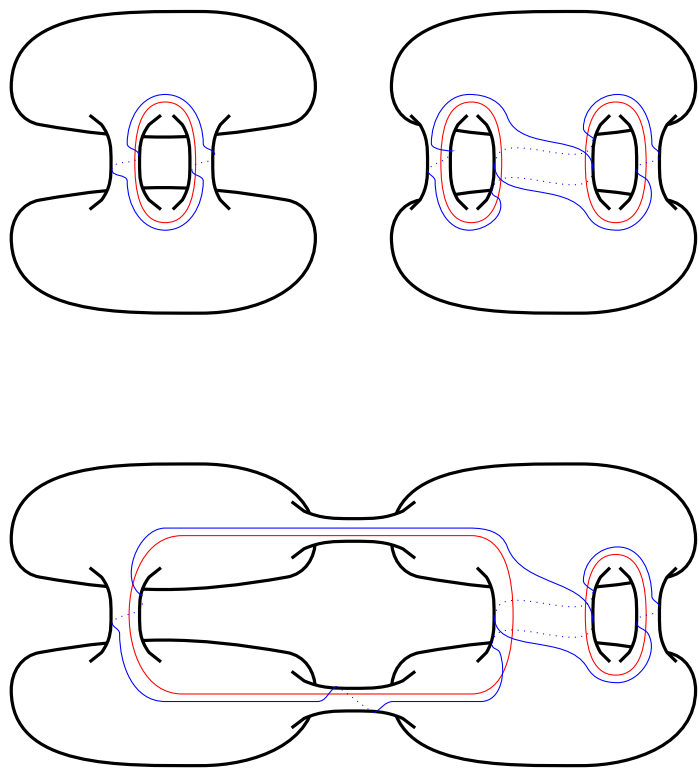
\begin{figure}[ht]
 \begin{center}
  \begin{tikzpicture}[scale=0.5]

\def\geno{

    \draw[very thick] (2.25+1,1) to [out=10,in=270] (3+1,2.0)
    to [out=90, in=0] (0+1,4.0)
    to [out=180, in=90] (-3-1, 2)
    to [out=270, in=170] (-2.25-1,1) to[out=-10,in=190] (2.25+1,1);}

\def\tubo{
    \fill[white] (-0.75,1.25) -- (-0.5,1)
    to [out=300, in=90] (-0.3,0)
    to [out=270, in = 60] (-0.5,-1)
    -- (-0.75,-1.25)-- (0.75,-1.25) -- (0.5,-1)
    to [out=-240, in=-90] (0.3,0)
    to [out=-270, in = -120] (0.5,1)
    -- (0.75,1.25)-- (-0.75,1.25);

    \draw[ very thick] (-0.75,1.25) -- (-0.5,1)
    to [out=300, in=90] (-0.3,0)
    to [out=270, in = 60] (-0.5,-1)
    -- (-0.75,-1.25);
    \draw[very thick] (0.75,1.25) -- (0.5,1)
    to [out=240, in=90] (0.3,0)
    to [out=270, in = 120] (0.5,-1)
    -- (0.75,-1.25);
    }

\def\tubogordo{
    \fill[white] (-1.75,1.25) -- (-1.5,1)
    to [out=300, in=90] (-1.3,0)
    to [out=270, in = 60] (-1.5,-1)
    -- (-1.75,-1.25)-- (1.75,-1.25) -- (1.5,-1)
    to [out=-240, in=-90] (1.3,0)
    to [out=-270, in = -120] (1.5,1)
    -- (1.75,1.25)-- (-1.75,1.25);

    \draw[ very thick] (-1.75,1.25) -- (-1.5,1)
    to [out=300, in=90] (-1.3,0)
    to [out=270, in = 60] (-1.5,-1)
    -- (-1.75,-1.25);
    \draw[very thick] (1.75,1.25) -- (1.5,1)
    to [out=240, in=90] (1.3,0)
    to [out=270, in = 120] (1.5,-1)
    -- (1.75,-1.25);
    }

\geno

\begin{scope}[yscale=-1]
\geno

\end{scope}

\tubogordo

\begin{scope}[xshift=-2.75cm]
\tubo
\end{scope}

\begin{scope}[xshift=2.75cm]
\tubo
\end{scope}

\begin{scope}[xshift=-10cm]
\geno

\begin{scope}[yscale=-1]
\geno
\end{scope}

\begin{scope}[xshift=-1cm,xscale=1.25]
\tubo

\end{scope}

\begin{scope}[xshift=1cm]
\tubo

\end{scope}
\end{scope}

\def\curvaroja{
\draw[red] (0.0, 1.6)
to [out=180, in=90] (-0.8, 0.0)
to [out=270, in=180] (0.0, -1.6)
to [out=0, in=270] (0.8, 0.0)
to [out=90, in=0] (0.0, 1.6);
}
\begin{scope}[xshift=-1.9cm]
\curvaroja
\end{scope}

\begin{scope}[xshift=1.9cm]
\curvaroja
\end{scope}

\begin{scope}[xshift=-9.95cm]
\curvaroja
\end{scope}

\begin{scope}[xshift=-9.95cm]
\draw[blue] (1.3, 0.2)
to [out=90, in=270] (1.0, 0.5)
to [out=90, in=0] (0.0, 1.8)
to [out=180, in=90] (-1.0, 0.5)
to [out=270, in=90] (-0.7, 0.2);
\draw[blue, dotted] (-0.7, 0.2)
to [out=270, in=90] (-1.4, -0.2);
\draw[blue]  (-1.4,-0.2)
to [out=270, in=90] (-1.0, -0.5)
to [out=270, in=180] (0.0, -1.8)
to [out=0, in=270] (1.0, -0.5)
to [out=90, in=270] (0.7, -0.2);
\draw[blue, dotted] (0.7, -0.2)
to [out=90, in=270] (1.3, 0.2);
\end{scope}

\begin{scope}[xshift=-1.95cm]
\draw[blue] (3.25, -0.2)
to [out=90, in=290] (1., 1.2)
to [out=110, in=0] (0.0, 1.8)
to [out=180, in=90] (-1.0, 0.5)
to [out=270, in=0] (-0.5, 0.3);
\draw[blue, dotted] (3.25, -0.2)
to [out=270, in=90] (0.6, -0.8);
\draw[blue]  (-1.1,-0.2)
to [out=270, in=90] (-0.9, -0.5)
to [out=270, in=180] (0.0, -1.8)
to [out=0, in=290] (0.8, -1.0)
to [out=110, in=270] (0.6, -0.8);
\draw[blue, dotted] (-0.5, 0.3)
to [out=270, in=90] (-1.1, -0.3);
\end{scope}

\begin{scope}[xshift=1.95cm]
\draw[blue] (1.1, 0.2)
to [out=90, in=270] (0.9, 0.5)
to [out=90, in=0] (0.0, 1.8)
to [out=180, in=90] (-0.9, 0.9)
to [out=270, in=0] (-0.65, 0.6);
\draw[blue, dotted] (-0.65, 0.6)
to [out=270, in=90] (-3.2, -0.0);
\draw[blue]  (-3.25,-0.0)
to [out=270, in=120] (-0.8, -1.4)
to [out=300, in=180] (0.0, -1.8)
to [out=0, in=300] (0.8, -0.5)
to [out=120, in=270] (0.5, -0.2);
\draw[blue, dotted] (0.5, -0.2)
to [out=90, in=270] (1.1, 0.2);
\end{scope}

\begin{scope}[yshift=-12cm]
\geno

\begin{scope}[yscale=-1]
\geno

\end{scope}

\tubogordo

\begin{scope}[xshift=2.75cm]
\tubo
\end{scope}

\begin{scope}[xshift=-10cm]
\geno

\begin{scope}[yscale=-1]
\geno
\end{scope}

\begin{scope}[xshift=-1cm,xscale=1.25]
\tubo

\end{scope}

\end{scope}

\begin{scope}[rotate=90,xshift=2.25cm,yshift=5cm,yscale=1.3]
\tubo
\end{scope}

\begin{scope}[rotate=90,xshift=-2.25cm,yshift=5cm,yscale=1.3]
\tubo
\end{scope}

\def\curvaroja{
\draw[red] (0.0, 1.6)
to [out=180, in=90] (-0.8, 0.0)
to [out=270, in=180] (0.0, -1.6)
to [out=0, in=270] (0.8, 0.0)
to [out=90, in=0] (0.0, 1.6);
}

\draw[red] (-1.9cm, 2.1) --(-9.5cm,2.1)
to [out=180, in=90] (-10.9, 0.0)
to [out=270, in=180] (-9.5, -2.1)
-- (-1.9cm, -2.1)
to [out=0, in=270] (-0.8, 0.0)
to [out=90, in=0] (-1.9,2.1);

\begin{scope}[xshift=1.9cm]
\curvaroja
\end{scope}

\begin{scope}[xshift=-1.95cm]
\draw[blue] (3.25, -0.2)
to [out=90, in=290] (1., 1.8)
to [out=110, in=0] (0.0, 2.3)--(-8,2.3)
to [out=180, in=150] (-8.6, 0.5);
\draw[blue, dotted] (3.25, -0.2)
to [out=270, in=90] (0.6, -0.8);
\draw[blue]  (-9.4,-0.2)
to [out=270, in=90] (-9.2, -0.5)
to [out=270, in=180] (-8, -2.3)--(-4,-2.3)
to [out=0, in=210] (-3.5,-1.9);
\draw[blue,dotted] (-3.5,-1.9)
to [out=330,in=150] (-2.5, -2.6);
\draw[blue] (-2.5, -2.6)
to [out=30, in=180] (-2, -2.3) --(0.3,-2.3)
to [out=0, in=290] (0.8, -1.0)
to [out=110, in=270] (0.6, -0.8);
\draw[blue, dotted] (-8.6, 0.3)
to [out=270, in=90] (-9.4, -0.3);
\end{scope}

\begin{scope}[xshift=1.95cm]
\draw[blue] (1.1, 0.2)
to [out=90, in=270] (0.9, 0.5)
to [out=90, in=0] (0.0, 1.8)
to [out=180, in=90] (-0.9, 0.9)
to [out=270, in=0] (-0.65, 0.6);
\draw[blue, dotted] (-0.65, 0.6)
to [out=270, in=90] (-3.2, -0.0);
\draw[blue]  (-3.25,-0.0)
to [out=270, in=120] (-0.8, -1.4)
to [out=300, in=180] (0.0, -1.8)
to [out=0, in=300] (0.8, -0.5)
to [out=120, in=270] (0.5, -0.2);
\draw[blue, dotted] (0.5, -0.2)
to [out=90, in=270] (1.1, 0.2);
\end{scope}
\end{scope}

\end{tikzpicture}
\caption{Heegaard diagram of the quotient singularity.}
 \end{center}

\end{figure}

\end{ejm}

\begin{ejm}
The link of the singularity defined by $z^2+x^3+y^5=0$ ($\mathbb{E}_8$-singularity) is the Poincar{\'e} sphere. Our method provides
a Heegaard splitting of genus~$3$, where the central vertex
needs four drills (three negative ones).

 \begin{figure}[ht]
\begin{center}
\begin{tikzpicture}[scale=2,
vertice/.style={draw,circle,fill,minimum size=0.2cm,inner
sep=0}
]
\draw (-3,0) --(3,0) ;

\foreach \a in {-3,-2,...,3}
{
\node[vertice] at (\a,0) {};
\node[above] at (\a,0) {$-2$};
}

\draw (-1,0) --(-1,-1) ;
\node[vertice] at (-1,-1) {};

\node[right] at (-1,-1) {$-2$};

\end{tikzpicture}

\end{center}
%
 \begin{center}
  \begin{tikzpicture}[scale=0.5]

\def\geno{

    \draw[very thick] (2.25+1,1) to [out=10,in=270] (3+1,2.0)
    to [out=90, in=0] (0+1,4.0)
    to [out=180, in=90] (-3-1, 2)
    to [out=270, in=170] (-2.25-1,1) to[out=-10,in=190] (2.25+1,1);}
\def\asa{
    \fill[white]
    (1.5,3.0)
    to [out=150, in=270] (1.0, 4.0)
    to [out=90, in=0.0] (0.0, 5.0)
    to [out=180, in=90] (-1.0,4.0)
    to [out=270, in=30] (-1.5,3.0)--
    (-0.25,3.0)
    to [out=150, in=270] (-0.5,4.0)
    to [out=90, in=180] (0,4.5)
    to [out=0, in=90] (0.5,4.0)
    to [out=270, in=30] (0.25, 3.0);

    \draw[very thick] (1.5,3.0)
    to [out=150, in=270] (1.0, 4.0)
    to [out=90, in=0.0] (0.0, 5.0)
    to [out=180, in=90] (-1.0,4.0)
    to [out=270, in=30] (-1.5,3.0);

    \draw[very thick] (0.25,3.0)
    to [out=30, in=270] (0.5,4.0)
    to [out=90, in=0] (0,4.5)
    to [out=180, in=90] (-0.5,4.0)
    to [out=270, in=150] (-0.25, 3.0);
    }

\def\tubo{
    \fill[white] (-0.75,1.25) -- (-0.5,1)
    to [out=300, in=90] (-0.3,0)
    to [out=270, in = 60] (-0.5,-1)
    -- (-0.75,-1.25)-- (0.75,-1.25) -- (0.5,-1)
    to [out=-240, in=-90] (0.3,0)
    to [out=-270, in = -120] (0.5,1)
    -- (0.75,1.25)-- (-0.75,1.25);

    \draw[ very thick] (-0.75,1.25) -- (-0.5,1)
    to [out=300, in=90] (-0.3,0)
    to [out=270, in = 60] (-0.5,-1)
    -- (-0.75,-1.25);
    \draw[very thick] (0.75,1.25) -- (0.5,1)
    to [out=240, in=90] (0.3,0)
    to [out=270, in = 120] (0.5,-1)
    -- (0.75,-1.25);
    }

\begin{scope}[xscale=1.5]

	\geno

	\begin{scope}[yscale=-1,yshift=1cm]
	\geno
	\end{scope}
\end{scope}
	\begin{scope}[xshift=-3cm,yscale=1.7,yshift=-0.3cm]
	\tubo
	\end{scope}

	\begin{scope}[xshift=3cm,yscale=1.7,yshift=-0.3cm]
	\tubo
	\end{scope}
\begin{scope}[yscale=1.7,xscale=2,yshift=-0.3cm]
\tubo
\end{scope}

\draw[red]
(-0.2,0) to[out=90,in=0] (-1.75,2)
to[out=180,in=90](-3,0)
to[out=270,in=180] (-1.75,-3)
to [out=0,in=270] (-0.2,0);

\draw[red]
(0.2,0) to[out=90,in=180] (1.75,2)
to[out=0,in=90](3,0)
to[out=270,in=0] (1.75,-3)
to [out=180,in=270] (0.2,0);

\draw[blue] (0.7,-1.5)
to [out=270,in=90] (0.4,-2.5)
to[out=270,in=180] (2,-3.2)
to[out=0,in=330] (2.7,-1.6);
\draw[blue,dotted] (2.7,-1.6)
to[out=30,in=210] (3.3,-1.4);
\draw[blue] (3.3,-1.4)
to[out=150,in=-30] (2.7,-1.2);
\draw[blue,dotted] (2.7,-1.2)
to[out=30,in=210] (3.3,-1);
\draw[blue] (3.3,-1)
to[out=150,in=-30] (2.7,-0.8);
\draw[blue,dotted] (2.7,-0.8)
to[out=30,in=210] (3.3,-0.6);
\draw[blue] (3.3,-0.6)
to[out=150,in=-30] (2.7,-0.4);
\draw[blue,dotted] (2.7,-0.4)
to[out=30,in=210] (3.3,-0.2);
\draw[blue] (3.3,-0.2)
to[out=150,in=-30] (2.7,0);
\draw[blue,dotted] (2.7,0)
to[out=30,in=210] (3.3,0.2);
\draw[blue] (3.3,0.2)
to[out=140,in=270] (3.1,1)
to[out=90,in=0] (1.75,2.2)
to[out=180,in=90] (0.1,0.7)
to[out=270,in=30] (-0.6,0);
\draw[blue,dotted] (-0.6,0)
to[out=-30,in=150] (0.6,-0.5);
\draw[blue] (0.6,-0.5)
to[out=-30,in=150] (-0.55,-1);
\draw[blue,dotted] (-0.55,-1)
to[out=-30,in=150] (0.7,-1.5);

\draw[blue](0.7,-1.3)
to[out=240,in=0] (-2,-3.2)
to[out=180,in=300] (-3.3,-1.5);
\draw[blue,dotted] (-3.3,-1.5)
to[out=60,in=240] (-2.7,-1.2);
\draw[blue] (-2.7,-1.2)
to[out=120,in=330] (-3.3,-0.9);
\draw[blue,dotted] (-3.3,-0.9)
to[out=30,in=210] (-2.7,-0.6);
\draw[blue] (-2.7,-0.6)
to[out=120,in=330] (-3.3,-0.3);
\draw[blue,dotted] (-3.3,-0.3)
to[out=30,in=210] (-2.7,0);
\draw[blue] (-2.7,0)
to[out=150,in=270] (-3.1,0.7)
to[out=90,in=180] (-2,2.2)
to[out=0,in=90] (-0.1,1)
to[out=270,in=30] (-0.6,0.2);
\draw[blue,dotted] (-0.6,0.2)
to[out=-30,in=150] (0.6,-0.3);
\draw[blue] (0.6,-0.3)
to[out=-30,in=150] (-0.55,-0.8);
\draw[blue,dotted] (-0.55,-0.8)
to[out=-30,in=150] (0.7,-1.3);

\end{tikzpicture}
 \end{center}
\caption{Heegaard diagram of the $\mathbb{E}_8$-singularity.}
\label{fig:E8}
\label{fig:hegdiagE8}
\end{figure}
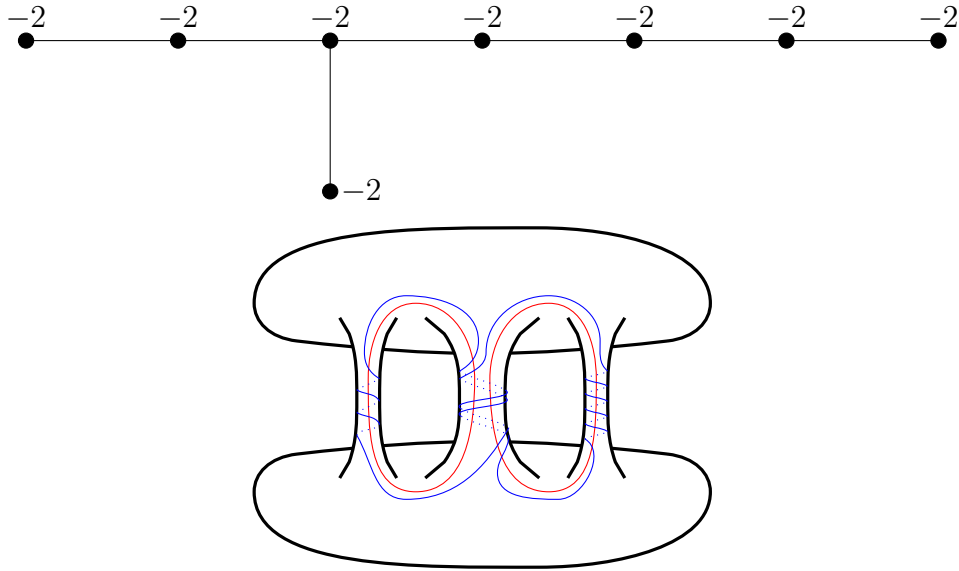

It is possible to make a simpler Heegaard splitting. Using $+1$-blow-ups of~\cite{neu:81}
(and one $-1$-blow-down),
we can modify the Euler numbers:  $2$ for the lower vertex and ~$-1$
in the central vertex. In that case, using the procedure in Remark~\ref{obs:handle-vv}, we
can make a float gluing along the main cylinder, obtaining a Heegaard splitting of
genus~$2$.
\end{ejm}

\begin{ejm}
The graph manifold of Figure~\ref{fig:nonSeifert}
is also the link of a normal surface singularity (which cannot
be quasihomogeneous) and admits a Heegaard splitting of
genus~$5$.

 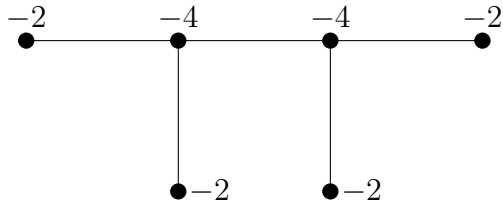
\begin{figure}[ht]
\begin{center}
\begin{tikzpicture}[scale=2,
vertice/.style={draw,circle,fill,minimum size=0.2cm,inner
sep=0}
]
\draw (-1,0) --(2,0) ;

\foreach \a in {-1,0,1,2}
{
\node[vertice] at (\a,0) {};
}

\node[above] at (-1,0) {$-2$};
\node[above] at (0,0) {$-4$};
\node[above] at (2,0) {$-2$};
\node[above] at (1,0) {$-4$};
\draw (0,0) --(0,-1) ;
\node[vertice] at (0,-1) {};
\node[right] at (0,-1) {$-2$};
\draw (1,0) --(1,-1) ;
\node[vertice] at (1,-1) {};
\node[right] at (1,-1) {$-2$};

\end{tikzpicture}
\caption{Non-Seifert manifold.}
\label{fig:nonSeifert}
\end{center}
\end{figure}
\end{ejm}


\begin{thebibliography}{10}

\bibitem{BoNe}
J.~Bodn\'ar and A.~N\'emethi, \emph{Lattice cohomology and rational cuspidal
  curves}, Math. Res. Lett. \textbf{23} (2016), no.~2, 339--375. \MR{3512889}

\bibitem{BOt:91}
M.~Boileau and J.-P. Otal, \emph{Scindements de {H}eegaard et groupe des
  hom\'eotopies des petites vari\'et\'es de {S}eifert}, Invent. Math.
  \textbf{106} (1991), no.~1, 85--107.

\bibitem{BZ:84}
M.~Boileau and H.~Zieschang, \emph{Heegaard genus of closed orientable
  {S}eifert {$3$}-manifolds}, Invent. Math. \textbf{76} (1984), no.~3,
  455--468.

\bibitem{hgd}
P.~Heegaard, \emph{Sur l'``{A}nalysis situs''}, Bull. Soc. Math. France
  \textbf{44} (1916), 161--242.

\bibitem{LaNe:15}
T.~L\'aszl\'o and A~N\'emethi, \emph{Reduction theorem for lattice cohomology},
  Int. Math. Res. Not. IMRN (2015), no.~11, 2938--2985. \MR{3373041}

\bibitem{MorSchu:98}
Y.~Moriah and J.~Schultens, \emph{Irreducible {H}eegaard splittings of
  {S}eifert fibered spaces are either vertical or horizontal}, Topology
  \textbf{37} (1998), no.~5, 1089--1112.

\bibitem{Ne:17}
A.~N\'emethi, \emph{Links of rational singularities, {L}-spaces and {LO}
  fundamental groups}, Invent. Math. \textbf{210} (2017), no.~1, 69--83.
  \MR{3698339}

\bibitem{neu:81}
W.D. Neumann, \emph{A calculus for plumbing applied to the topology of complex
  surface singularities and degenerating complex curves}, Trans. Amer. Math.
  Soc. \textbf{268} (1981), no.~2, 299--344.

\bibitem{OzSz:04a}
P.~Ozsv\'ath and Z.~Szab\'o, \emph{Holomorphic disks and three-manifold
  invariants: properties and applications}, Ann. of Math. (2) \textbf{159}
  (2004), no.~3, 1159--1245.

\bibitem{OzSz:04b}
\bysame, \emph{Holomorphic disks and topological invariants for closed
  three-manifolds}, Ann. of Math. (2) \textbf{159} (2004), no.~3, 1027--1158.

\bibitem{rolf}
D.~Rolfsen, \emph{Knots and links}, Mathematics Lecture Series, no.~7, Publish
  or Perish, Inc., Berkeley CA, 1970.

\bibitem{sarkar-wang}
S.~Sarkar and J.~Wang, \emph{An algorithm for computing some {H}eegaard {F}loer
  homologies}, Ann. of Math. (2) \textbf{171} (2010), no.~2, 1213--1236.

\bibitem{Schu:04}
J.~Schultens, \emph{Heegaard splittings of graph manifolds}, Geom. Topol.
  \textbf{8} (2004), 831--876.

\bibitem{wal:67}
F.~Waldhausen, \emph{Eine klasse von $3$-dimensionalen mannigfaltigkeiten {I}},
  Invent. Math. \textbf{3} (1967), 308--333.

\bibitem{wal:67a}
\bysame, \emph{Eine klasse von $3$-dimensionalen mannigfaltigkeiten {I}{I}},
  Invent. Math. \textbf{4} (1967), 87--117.

\end{thebibliography}

\def\cprime{$'$}
\providecommand{\bysame}{\leavevmode\hbox to3em{\hrulefill}\thinspace}
\providecommand{\MR}{\relax\ifhmode\unskip\space\fi MR }
\providecommand{\MRhref}[2]{%
  \href{http://www.ams.org/mathscinet-getitem?mr=#1}{#2}
}
\providecommand{\href}[2]{#2}

\end{document}